\documentclass[reqno, 11pt]{amsart}
\usepackage{amsmath, amsfonts, amsthm, amssymb, stmaryrd, color, cite, mathrsfs}

\usepackage{hyperref}

\newtheorem{theorem}{Theorem}[section]
\newtheorem{lemma}{Lemma}[section]
\newtheorem{proposition}{Proposition}[section]

\theoremstyle{definition}

\theoremstyle{remark}

\numberwithin{equation}{section}
\allowdisplaybreaks

\def\f{\frac}

\def\hf1{^\f{1}{1-\xi^2}}

\def\be{\begin{equation}}
\def\ee{\end{equation}}
\def\bs{\begin{split}}
\def\es{\end{split}}
\def\ba{\begin{align}}
\def\ea{\end{align}}

\title[stochastic Landau-Lifshitz-Bloch equation]
{Invariant measures and
their limiting behavior  of the
 Landau-Lifshitz-Bloch equation in unbounded domains}

\author[D. Huang]{Daiwen Huang}
\address{Institute of Applied Physics and Computational Mathematics, Beijing 100088,  China}
\email{huang\_daiwen@iapcm.ac.cn}

\author[Z. Qiu]{Zhaoyang Qiu}
\address{School of Applied Mathematics, Nanjing University of Finance and Economics, Nanjing, 210046, China.}
\email{zhqmath@163.com}

\author[B. Wang]{Bixiang Wang}
\address{Department of Mathematics, New Mexico Institute of Mining and Technology, Socorro, NM 87801, USA.}
\email{bwang@nmt.edu.}


\date{\today}

\begin{document}
\baselineskip=1.2\baselineskip
\begin{abstract}
This paper deals with
the existence and limiting
behavior of invariant measures
of the stochastic  Landau-Lifshitz-Bloch equation
driven by linear multiplicative noise
and additive noise
defined in the entire space
$\mathbb{R}^d$ for $d=1,2$,
which  describes the phase spins in   ferromagnetic materials around the Curie temperature.
We first establish the
existence and uniqueness
of solutions by a domain expansion method.
We then prove   the existence of invariant measures  by the weakly
Feller argument.
In the case $d=1$,
we show the uniform
tightness of the set of all
invariant measures of the
stochastic equation, and prove
any limit of a sequence
of invariant measures of the
perturbed equation must be
an invariant measure of the
limiting system.
The
  cut-off arguments, stopping time techniques and   uniform tail-ends
   estimates of solutions are developed to overcome the difficulty
   caused by   the high-order
   nonlinearity and  the non-compactness of Sobolev embeddings in unbounded domains.
\end{abstract}

\maketitle
 \section{Introduction}
The Landau-Lifshitz-Bloch (LLB) equation is
a model  for simulating the heat-assisted magnetic recording in   ferromagnetic materials around the Curie temperature,  which is derived from the Landau-Lifshitz-Gilbert equation at low temperatures and the Ginzburg-Landau theory of phase
transitions. And the LLB equation explains the light-induced demagnetization with powerful femtosecond lasers, see, e.g., \cite{A1,ga} for the physical backgrounds.
Let  $\mathscr{D}$ be a domain
in $\mathbb{R}^d$ with $d\le 2$.
Then
a specific form of the LLB  equation
in
$\mathscr{D}$
 is given by
$$\frac{\partial \mathbf{u}}{\partial t}=\gamma \mathbf{u}\times \mathrm{H}_{e\!f\!f}+ \mathrm{H}_{e\!f\!f},$$
where the effective field $\mathrm{H}_{e\!f\!f}$ has the form
$$\mathrm{H}_{e\!f\!f}=\Delta\mathbf{u}-(1+\kappa|\mathbf{u}|^2)\mathbf{u},$$
and $\mathbf{u}=({u}_1, {u}_2, {u}_3)$ is the average spin polarization, $\gamma>0$ represents the gyromagnetic ratio and $\kappa>0$ is a constant related to the temperature.
Without loss of generality,
 we will assume $\kappa=\gamma=1$.
  The notation $\times$ denotes the vector cross product, and $|\cdot|$ is the Euclidean norm of $\mathbb{R}^3$. Then the deterministic LLB equation can be written more precisely as
\begin{eqnarray*}
d\mathbf{u}=\Delta \mathbf{u} dt+ \mathbf{u}\times \Delta \mathbf{u}dt-(1+|\mathbf{u}|^{2})\mathbf{u}dt.
\end{eqnarray*}
For the deterministic LLB equation,  Le\cite{le}
first  proved the existence of weak solutions in a smooth bounded domain using the Galerkin  approximation scheme and the method of compactness.
Then Guo et al. \cite{guo} established
 the well-posedness of strong solutions of the equation on a finite-dimensional
  closed Riemannian manifold.
  In this paper, we  will consider
  the stochastic LLB equation
   driven by a linear multiplicative noise
   and additive noise.

\subsection{Motivations}
  In the   ferromagnetic
  theory,  a significant issue  is to describe the transition between different equilibrium states induced by thermal fluctuation, which was proposed initially by
   N\'{e}el \cite{N1}.    The random fluctuation is
  often introduced into the equation by perturbing the effective field  $\mathrm{H}_{e\!f\!f}$ using the Gaussian noise, see \cite{ga1};
  that is,
  $\mathrm{H}_{e\!f\!f}$
  is replaced
   by $\mathrm{H}_{e\!f\!f}+\frac{\circ d W}{dt}$ where $W(t) =\sum_{k =1}^\infty
   \mathbf{f}_k W_k(t)$ is a   Wiener noise defined in a  complete
   filtered probability space $(\Omega,\mathcal{F}, \{\mathcal{F}_t\}_{t\geq 0}, \mathrm{P})$,   $\{\mathbf{f}_k\}_{k= 1}^\infty$ is a sequence of functions satisfying
\begin{align}\label{1.1}
\sum_{k =1}^\infty
\|\mathbf{f}_k\| _{W^{1,\infty}(\mathbb{R}^d,
\mathbb{R}^3
)\cap H^1(\mathbb{R}^d, \mathbb{R}^3)} <\infty,
\end{align}
and $\{W_k\}_{k\geq 1}$ is a sequence  of independent real-valued Wiener processes.

Consider the case where
  the ferromagnet fills in the entire space
  $\mathbb{R}^d$, then
  the stochastic LLB  equation  in
  $\mathbb{R}^d$
  reads as
\be\label{equ1.1a}
d\mathbf{u}=\Delta \mathbf{u} dt+ \mathbf{u}\times \Delta \mathbf{u}dt-(1+|\mathbf{u}|^{2})\mathbf{u}dt+\sum_{k= 1}^\infty
(\mathbf{u}\times \mathbf{f}_k+\mathbf{f}_k)\circ d W_k,
\ee
which is driven by
a linear multiplicative noise and additive noise
in the sense of
the Stratonovich differential.
 As explained   in \cite{eva},
such noise  can deduce a Boltzmann distribution valid for the full range of temperatures.

Note that the
  Stratonovich  stochastic equation
  \eqref{equ1.1a} is equivalent to the
  It\^{o} equation in $\mathbb{R}^d$:
\begin{equation}
\label{equ1.1}
d\mathbf{u}=\Delta \mathbf{u} dt+ \mathbf{u}\times \Delta \mathbf{u}dt-(1+|\mathbf{u}|^{2})\mathbf{u}dt
+\frac{1}{2}\sum_ {k= 1}^\infty
(\mathbf{u}\times \mathbf{f}_k)\times \mathbf{f}_kdt
+\sum_ {k= 1}^\infty(\mathbf{u}\times \mathbf{f}_k+\mathbf{f}_k)d W_k,
\end{equation}
which is supplemented with
   the   initial data:
\begin{align*}
\mathbf{u}(0, x)=\mathbf{u}_0(x),
\quad x\in \mathbb{R}^d.
\end{align*}

Due to its physical importance and mathematical challenges,
  the stochastic
LLB  model  has been studied recently
by several authors.
When the   equation
is defined in a bounded domain $\mathscr{D}$,
the   existence of martingale   solutions
as well as the existence of a unique strong solutions
of
\eqref{equ1.1} with
initial data
in  $H^1(\mathscr{D})$
were proved by  Jiang et al. \cite{W1} and   Brze\'{z}niak et al. \cite{B1} respectively.
In addition, the existence of invariant
measures of \eqref{equ1.1}
was proved in \cite{B1}
by the weakly Feller method.
When the equation
is driven by a transport noise,
the existence and the regularity
of solutions  were established by
 Qiu et al. \cite{q1}. Brze\'{z}niak et al. \cite{B5} considered the large deviations and transitions between
equilibria for stochastic Landau-Lifshitz-Gilbert equation.
 Note that in all these articles,
 the stochastic equation is defined in
 a bounded domain in $\mathbb{R}^d$ or the torus. As far as the authors
 are aware, it seems that there is no
 result available in the literature
 on the well-posedness and
 invariant  measures of \eqref{equ1.1}
 when the domain is unbounded.
 The main goal of the present paper
 is to deal with this problem and
  study the existence and the limiting
 behavior of invariant measures
 of \eqref{equ1.1} defined in
 the entire space $\mathbb{R}^d$.
  The reader is referred to
   \cite{ben, ben2, CG, Da2, F1, gao,
 Ha, H1}
 and \cite{bm, bos, bmo, cb,JH, k1, k3, m1, bx3, b}
 for the well-posedness and invariant measures of
 stochastic partial differential equations
 defined in bounded domains
 and unbounded domains, respectively.
 It is worth mentioning that
 the  main difficulty
in the case of unbounded domains
 lies in the fact
that Sobolev embeddings
are no longer compact which introduces
an essential obstacle for establishing
the tightness of  distributions
of a family  of solutions.

\subsection{Main results}
As mentioned above,
the main purpose of this
paper is to
investigate the existence and
limiting behavior of
invariant measures
of \eqref{equ1.1}.
To that end, we first   establish the existence and uniqueness of   solutions
of \eqref{equ1.1}
in $\mathbb{R}^d$ for $d=1,2$,
which is stated below.

\begin{theorem}\label{t3}
Let
$ (\Omega,\mathcal{F}, \{\mathcal{F}_t\}_{t\geq 0}, \mathrm{P}, W) $  be a given
stochastic basis.
If \eqref{1.1} holds and
    $\mathbf{u}_0\in H^1(\mathbb{R}^d,
    \mathbb{R}^3)$ with  $d=1,2$,
   then the stochastic LLB equation \eqref{equ1.1} admits a unique global   solution $\mathbf{u}$ in the following sense:

${\rm (i)}$. The  process $\mathbf{u}$ is
weakly continuous,   $H^1(\mathbb{R}^d,
    \mathbb{R}^3)$-valued,  $\mathcal{F}_t$-progressively measurable such that
\begin{align}\label{1.3}\mathbf{u}\in L^p(\Omega; L^\infty(0,T; H^1(\mathbb{R}^d,
    \mathbb{R}^3))\cap L^2(0,T;H^2(\mathbb{R}^d,
    \mathbb{R}^3) )),  \quad \forall \ p\ge 1,
    \end{align}
    and
    \be\label{1.3a}
    \mathbf{u} \times
    \Delta \mathbf{u}
    \in
L^{pr}(\Omega; L^r(0,T; L^2
(\mathbb{R}^d,
    \mathbb{R}^3))),
    \quad
    \forall \ p\ge1, \ r\in[1,2).
    \ee

${\rm (ii)}$.  For all $t\in [0,T]$, the following equality holds in $L^2(\mathbb{R}^d,
\mathbb{R}^3 )$, $\mathrm{P}$-almost surely,
\begin{align}\label{3.1}
&\mathbf{u}(t)=\mathbf{u}_0+\int_{0}^{t}\Delta \mathbf{u}(s)ds+\int_{0}^{t}\mathbf{u}(s)\times \Delta \mathbf{u}(s)ds\nonumber\\
&\qquad-\int_{0}^{t}(1+|\mathbf{u}(s)|^{2})\mathbf{u}(s)ds
+\frac{1}{2}\sum_{k= 1}^\infty
\int_{0}^{t}(\mathbf{u}(s)\times \mathbf{f}_k)\times \mathbf{f}_kds
+\sum_ {k= 1}^\infty\int_{0}^{t}(\mathbf{u}(s)\times \mathbf{f}_k+\mathbf{f}_k)d W_k .
\end{align}
\end{theorem}

In the case of bounded domains\cite{B1,W1},
the existence of solutions
of the stochastic LLB equation
was proved
by   the finite-dimensional approximation scheme
and the compactness argument,
where the compactness of
Sobolev embeddings was used
to show the tightness of approximate solutions.
Since we consider the stochastic
equation \eqref{equ1.1}
in the unbounded domain
$\mathbb{R}^d$ and the Sobolev embeddings
are no longer compact
in unbounded domains, the arguments
of
  \cite{B1,W1} do not apply to
 our case.
   In this paper, we will adopt a domain expansion method to show the existence of   solutions
   of \eqref{equ1.1}.
   More precisely,
   we first consider the initial-boundary value problem of \eqref{equ1.1}, and  a sequence of approximate
   solutions  $\mathbf{u}^n$ is obtained
   defined in the  bounded domain
   $\mathscr{D}_n=\{x\in \mathbb{R}^d:
   |x|<n\}$, then derive the uniform
   estimates of approximate solutions. We emphasize that in an effort to pass the limit, a sequence of auxiliary processes
          is constructed in the entire space $\mathbb{R}^d$ and finally
   take the limit of  the sequence of new processes
   to obtain a solution of \eqref{equ1.1}
   defined  in  $\mathbb{R}^d$.
     Moreover,
     in order to pass to the
     limit of approximate solutions,
     we need to establish the
     tightness of distributions
     of the sequence, for which we must
     overcome the difficulty
     caused by the non-compactness
     of Sobolev embeddings
     in  $\mathbb{R}^d$.
     We will apply the idea of
     uniform tail-ends estimates
     to solve the problem by
     showing that the approximate solutions
     are uniformly small on the complement
     of $\mathscr{D}_n$ when
     $n$ is sufficiently large
     as in
        \cite{b,bx3}
        for the   reaction-diffusion equation.
          By
          the uniform smallness
          of the tails of solutions
          in $L^2(\mathbb{R}^d)
          \setminus \mathscr{D}_n$
          (see Lemma \ref{est_v} (iv))
          and the compactness
          of embeddings in
          bounded domains,
          we   then  prove the tightness of
          distributions of the
          approximate sequence
          (see Lemma \ref{tight1}).
          Finally, we show   the limit
          of the approximate
          solutions  is a   martingale solution
          of the stochastic
          LLB equation
          in $\mathbb{R}^d$,
          which along with   the pathwise  uniqueness of solutions and the
           Yamada-Watanabe theorem
           implies the   existence of a
           unique  strong  solution.

With the well-posedness in hand,
we then investigate
the existence and uniform tightness
of invariant measures of \eqref{equ1.1}
with respect to noise intensity. More precisely,
we consider the following equation
with a parameter $\varepsilon
\in [0,1]$:
\begin{align}\label{1.3*}
&d\mathbf{u}^\varepsilon=\Delta \mathbf{u}^\varepsilon dt+ \mathbf{u}^\varepsilon\times \Delta \mathbf{u}^\varepsilon dt-(1+|\mathbf{u}^\varepsilon|^{2})\mathbf{u}^\varepsilon dt+\varepsilon\sum_{k=1}^\infty(\mathbf{u}^\varepsilon \times \mathbf{f}_k+\mathbf{f}_k)\circ d W_k,
\end{align}
in $\mathbb{R}^d$,
where $\varepsilon$
represents the intensity of noise.

The main result of the paper on the
invariant measures of
the stochastic LLB equation
\eqref{1.3*} is presented  below.

\begin{theorem}\label{th1}
If \eqref{1.1} holds, then:

{\rm (i)}. For $d=1,2$,  the stochastic equation
\eqref{equ1.1} has at least one invariant
measure in $H^1(\mathbb{R}^d, \mathbb{R}^3)$.

{\rm  (ii)}. For $d=1$, the set  $\bigcup_{\varepsilon
\in [0,1]} \mathcal{I}^\varepsilon$
is tight in $H^1(\mathbb{R}, \mathbb{R}^3)$
where
$\mathcal{I}^\varepsilon$ is the
collection of all invariant measures of
\eqref{1.3*} corresponding to $\varepsilon$.

Furthermore,
for any  $\varepsilon_n \to \varepsilon_0\in [0,1]$
 and
$\mu^{\varepsilon_n}
\in \mathcal{I}^{\varepsilon_n}$,
there exists a subsequence
$\{\mu^{\varepsilon_{n_k}}\}_{k=1}^\infty$
of
$\{\mu^{\varepsilon_{n}}\}_{n=1}^\infty$
and
 $\mu^{\varepsilon_0}
\in \mathcal{I}^{\varepsilon_0}$
such that
$
\mu^{\varepsilon_{n_k}}
\to
\mu^{\varepsilon_0}$ weakly.
 \end{theorem}

 Note that Theorem \ref{th1} (ii)
 implies that any limit of a sequence
 of invariant measures of \eqref{1.3*}
 must be an invariant measure
 of the limiting equation.

 Based on the
 uniform tail-ends estimates of solutions
 in $L^2(\Omega, C([0,T]; L^2(\mathbb{R}^d, \mathbb{R}^3)))$,
 we will prove Theorem \ref{th1} (i)
 for the existence of invariant measures
 of \eqref{equ1.1}
 by the weakly Feller method as in
\cite{B1, Mas}.

The most difficult
part of the paper is to prove
Theorem \ref{th1} (ii), and
currently we only obtain  that
result  for $d=1$.
For $d=2$,
Theorem \ref{th1} (ii) remains open.

Next, we explain why the proof
of
Theorem \ref{th1} (ii) is so
challenging.
In order to establish the uniform
tightness
of the set
 $\bigcup_{\varepsilon
\in [0,1]} \mathcal{I}^\varepsilon$
in $H^1(\mathbb{R}^d, \mathbb{R}^3)$,
we must derive the uniform
tail-ends estimates of solutions
in    $H^1(\mathbb{R}^d, \mathbb{R}^3)$
rather than in $L^2(\mathbb{R}^d, \mathbb{R}^3)$, because the
uniform tail-ends estimates
in $L^2(\mathbb{R}^d, \mathbb{R}^3)$
are not sufficient in this case.
To obtain the uniform
tail-ends estimates of solutions
in    $H^1(\mathbb{R}^d, \mathbb{R}^3)$,
we need  to use the energy equation
of solutions in  $H^1(\mathbb{R}^d, \mathbb{R}^3)$.
But,  the  energy equation
of solutions in  $H^1(\mathbb{R}^d, \mathbb{R}^3)$ is not available for $d=2$.
To understand this issue, let's rewrite
\eqref{3.1}  as
\begin{align}\label{3.1wan}
 \mathbf{u}(t)=\mathbf{u}_0  +
 \int_{0}^{t}
\mathbf{F}  (\mathbf{u}(s)) ds
+\sum_ {k= 1}^\infty\int_{0}^{t}(\mathbf{u}(s)\times \mathbf{f}_k+\mathbf{f}_k)d W_k ,
\end{align}
where
$$
 \mathbf{F}  (\mathbf{u} )
 =
\Delta \mathbf{u} + \mathbf{u} \times \Delta \mathbf{u}
 - (1+|\mathbf{u} |^{2})\mathbf{u}
+\frac{1}{2}\sum_{k= 1}^\infty
 (\mathbf{u} \times \mathbf{f}_k)\times \mathbf{f}_k .
$$

The high-order nonlinear term
$\mathbf{u} \times \Delta \mathbf{u} $
has lowest regularity among all terms
in $\mathbf{F}  (\mathbf{u} )  $,
and hence  it dominates the regularity of
$\mathbf{F}  (\mathbf{u} )  $.
 If  we assume
 \be\label{3.1wana}
 \mathbf{u} \times \Delta \mathbf{u}
\in L^2(0,T; L^2(\mathbb{R}^d, \mathbb{R}^3)),
\quad  \text{P-almost surely},
\ee
then by
\eqref{1.3} we find  that
$  \mathbf{F} (\mathbf{u} )
\in L^2(0,T; L^2(\mathbb{R}^d, \mathbb{R}^3))$
$\mathrm{P}$-almost surely,
and hence by   \eqref{1.3}
and  \cite[Theorem 3.2]{kry1981}
 we infer that
 $\mathbf{u}
 \in C([0,T]; H^1  (\mathbb{R}^d, \mathbb{R}^3
))$
$\mathrm{P}$-almost surely,
and  $\mathbf{u} $ satisfies the
energy equation in
$H^1  (\mathbb{R}^d, \mathbb{R}^3
 )$:
 $$
 \|\mathbf{u} (t)\|^2_{H^1  (\mathbb{R}^d, \mathbb{R}^3
)}
 =
\|\mathbf{u}_0\|^2_{H^1  (\mathbb{R}^d, \mathbb{R}^3
)}
+2
\int_0^t
( (I-\Delta) \mathbf{u} (s),
\mathbf{F} ( \mathbf{u} (s) ))
_ {L^2 (\mathbb{R}^d, \mathbb{R}^3
)} ds
$$
 \be\label{3.1wanb}
+
 \sum_ {k= 1}^\infty\int_{0}^{t}
 \| \mathbf{u}(s)\times \mathbf{f}_k+\mathbf{f}_k
 \|^2_{H^1  (\mathbb{R}^d, \mathbb{R}^3
)} ds
+2
\sum_ {k= 1}^\infty\int_{0}^{t}(
(I-\Delta) \mathbf{u} (s),
 \mathbf{u}(s)\times \mathbf{f}_k+\mathbf{f}_k
 )_{L^2(\mathbb{R}^d,
 \mathbb{R}^3)} d W_k,
\ee
 for all $t\in [0,T]$, $\mathrm{P}$-almost surely.

 Unfortunately, by \eqref{1.3a} we find that
 for $d=2$,
 $\mathbf{u} \times \Delta \mathbf{u}
\in L^r(0,T; L^2(\mathbb{R}^d, \mathbb{R}^3))
$
only for $r<2$
but
not for $r=2$,  P-almost surely, and hence
the condition
\eqref{3.1wana} is not satisfied
for $d=2$.
Consequently, we cannot
obtain
 the energy equation \eqref{3.1wanb}
for $\mathbf{u} $
in
$H^1  (\mathbb{R}^2, \mathbb{R}^3)$,
and thus we are unable to
establish the uniform
tail-ends estimates of solutions
in $H^1  (\mathbb{R}^2, \mathbb{R}^3)$.
We may formally derive the
uniform
tail-ends estimates
in $H^1  (\mathbb{R}^2, \mathbb{R}^3)$,
but those calculations cannot be justified
due to the lower regularity of
 $\mathbf{F} (\mathbf{u})$.

Nevertheless, if $d=1$, then  we have
the embedding $H^1(\mathbb{R}, \mathbb{R}^3)\hookrightarrow L^\infty(\mathbb{R}, \mathbb{R}^3)$, and hence by \eqref{1.3}
we infer that
$
\mathbf{u} \times \Delta \mathbf{u}
\in L^2(0,T; L^2(\mathbb{R} , \mathbb{R}^3))$,
 {P-almost surely};
 that is,
 $\mathbf{u} \times \Delta \mathbf{u}$
  satisfies \eqref{3.1wana} for $d=1$.
  As a result,
  we find that $\mathbf{u}
 \in C([0,T]; H^1  (\mathbb{R} , \mathbb{R}^3
))$
$\mathrm{P}$-almost surely,
and  $\mathbf{u} $ satisfies the
energy equation   in
$H^1  (\mathbb{R} , \mathbb{R}^3
 )$.
 This is why we are able to
 obtain the
    uniform
tail-ends estimates of solutions
in $H^1  (\mathbb{R} , \mathbb{R}^3)$,
but not
in $H^1  (\mathbb{R}^2, \mathbb{R}^3)$
(see Lemma \ref{tail_3}).

  On the other hand,
  in order to prove
 a limit of a sequence
 of invariant measures of \eqref{1.3*}
 is   an invariant measure
 of the limiting equation, we must
 establish the
    convergence
    of solutions  in probability
    in $H^1  (\mathbb{R}^d , \mathbb{R}^3)$.
    Again, since the solutions
    have  no sufficient regularity
    for $d=2$, we are unable to
    obtain   the
    convergence
    of solutions  in probability
    in $H^1  (\mathbb{R}^2 , \mathbb{R}^3)$.
    However, if $d=1$, by
    the embedding $H^1(\mathbb{R}, \mathbb{R}^3)\hookrightarrow L^\infty(\mathbb{R}, \mathbb{R}^3)$,
    the
    convergence
    of solutions  in probability
    in $H^1  (\mathbb{R}  , \mathbb{R}^3)$
    can be established
    (see Lemma \ref{ucsp}).

    We remark that
    by the idea of uniform tail-ends
 estimates, we
 can also   prove the existence
 of invariant measures
 of \eqref{equ1.1a}
 for $d=1$
 by the Feller property
 of the Markov semigroup
 with respect to the strong topology
 instead of the weak topology
 (see Lemma \ref{fp}), which
  provides  an alternative
  approach to show the existence
  of invariant measures for $d=1$.

The rest of this paper is organized as follows. In section 2, we establish the existence and uniqueness of global   solutions by the domain expansion method. In section 3, we prove the existence of invariant measures
by the weakly  Feller method.
In the last section,
we prove Theorem \ref{th1}
and the existence of invariant measures
for $d=1$ by the Feller property of
solutions with respect to the
strong topology.

For later sections, we
  recall the following
Skorokhod-Jakubowski representation theorem
from \cite{bm, jak1}.
\begin{proposition}
\label{prop_sj}
 Suppose  $X$
 is  a topological space, and there
  exists a sequence
of continuous functions $g_n: X\rightarrow \mathbb{R}$ that separates points of $X$.
If    $\{\mu_n\}_{n=1}^\infty$
is a tight sequence of
   probability measures
  on $(X, \mathcal{B} (X) )$, then
  there exists a subsequence
  $\{\mu_{n_k}\}_{k=1}^\infty$
  of  $\{\mu_n\}_{n=1}^\infty$,
    a probability space $(
  \widetilde{
  \Omega},
  \widetilde{\mathcal{F}},
  \widetilde{\mathrm{P}})$, $X$-valued random variables
   $v_{ k}$  and $v$    such that
   the law of $v_{ k}$ is
   $\mu_{n_k}$ for all $k\in \mathbb{N}$ and
      $v_{ k} \rightarrow v$  $\widetilde{
  \mathrm{P}}$-almost surely
   as $k \rightarrow \infty$.
\end{proposition}

\section{Existence and uniqueness of solutions}

In this section, we
prove  the global well-posedness of strong   solutions
 of  \eqref{equ1.1}
 defined in $\mathbb{R}^d$ with
   $d=1,2$. Due to the non-compactness of
   the standard Sobolev embeddings in  unbounded
    domains, we adopt a domain expansion
   argument
   to
   show  the existence of solutions
   of \eqref{equ1.1}.
   Specifically, we first consider the approximate solutions of equation \eqref{equ1.1}
   defined  in a bounded domain $\mathscr{D}_n=\{x\in \mathbb{R}^d: |x|< n\}$
   for every
   $n\in \mathbb{N}$, then
   take the limit of these
   approximate
   solutions as  $n\rightarrow\infty$
   to obtain a solution
   of \eqref{equ1.1}
    defined in  the entire space $\mathbb{R}^d$.

\subsection{Construction of approximate solutions}
In this subsection, we will  define
a sequence of approximate solutions
for \eqref{equ1.1}.
 We first introduce the Sobolev spaces
that will be used in the sequel.
Let $\mathscr{D}$
be an open set in
$\mathbb{R}^d$
with $d=1,2$.
Given $p\ge 1$, let
$$
L^p(\mathscr{D}, \mathbb{R}^3)
=\left \{
\mathbf{u}: \mathscr{D}
\to \mathbb{R}^3:
\int_{\mathscr{D}}
|\mathbf{u} (x)|^p dx <\infty
\right \} .
$$
The norm of
$
L^p(\mathscr{D}, \mathbb{R}^3)
$ is  denoted by
$\| \cdot \|_{L^p(\mathscr{D}, \mathbb{R}^3) }$.
In particular,
for $p=2$, the inner
product
and the norm of
$
L^2(\mathscr{D}, \mathbb{R}^3)
$
are written as
$(\cdot, \cdot)_{L^2(\mathscr{D}, \mathbb{R}^3)}$
and $\| \cdot \|_{L^2(\mathscr{D}, \mathbb{R}^3)}$,
respectively.
For every   $k\in \mathbb{N}$
and  $p \ge 1$, let
$$
W^{k,p}(\mathscr{D},
\mathbb{R} ^3 )
=\left \{
\mathbf{u}
\in L^p
( \mathscr{D},
  \mathbb{R}^3) :
  \partial ^\alpha
  \mathbf{u}
  \in
  L^p
  ( \mathscr{D},
  \mathbb{R}^3),
  \ |\alpha | \le k
 \right \}.
 $$
 The norm of
 $W^{k,p}(\mathscr{D},
 \mathbb{R} ^3 )$
 is denoted by
 $\|\cdot \|_{W^{k,p}(\mathscr{D},
 	\mathbb{R} ^3 )}$.
 	Let
$W_0^{k,p}(\mathscr{D},
\mathbb{R}^3 )$
be the closure of $C_0^\infty(\mathscr{D}, \mathbb{R}^3)$ with respect to
 the topology of $W^{k,p}(\mathscr{D}, \mathbb{R}^3)$.
   For  $p=2$, we
   write  $H^k(\mathscr{D},
   \mathbb{R}^3)
   =  W^{k,2}(\mathscr{D} ,
   \mathbb{R} ^3
   )$
    which is a Hilbert space
    with inner product $(\cdot,\cdot)_{H^k(
  \mathscr{D},
  \mathbb{R} ^3)}=\sum\limits_{|\alpha|\leq k}(\partial^\alpha \cdot, \partial^\alpha \cdot)_{L^2(\mathscr{D}, \mathbb{R}^3)}$.
  Similarly, we also
  denote  $W_0^{k,2}(\mathscr{D} ,
   \mathbb{R} ^3
   )$ by
    $H^k_0 (\mathscr{D},
   \mathbb{R}^3)$.
    If no confusion occurs, we will use
   $L^p (\mathscr{D} )$
   for  $L^p (\mathscr{D},
\mathbb{R} ^3 )$, and use  similar notations
  for other spaces.

 Let $Y $ be a  separable  Hilbert space
 with inner product $ (\cdot, \cdot )_Y $.
 Denote by
   $Y_w$ the    space
   $Y$
endowed with  weak topology.  Given $T>0$
let
$$
C([0,T]; Y_w)=
\left \{
\mathbf{u}: \mathbf{u} \ \text{is  a  continuous  function  from }  [0,T]
\ \text{to}\ Y_w \right \},
$$
which is equipped  with the following topology:
$\mathbf{u}_n \to \mathbf{u}$ in
$C([0,T]; Y_w)$  if for every $y\in Y$,
$\lim\limits_{n\to \infty}
\sup\limits_{t\in [0,T]}
|(\mathbf{u}_n(t) - \mathbf{u}(t), y)_Y|
=0$.
Given $R>0$, let
$C([0,T]; Y^R_w)$ be a subset of
$C([0,T]; Y_w)$ given by
$$
C([0,T]; Y^R_w)=
\left \{
\mathbf{u}\in C([0,T]; Y_w): \sup_{t\in [0,T]}
\| \mathbf{u}(t)\|_Y
\le R    \right \}.
$$
Since the weak topology of
$Y$ on the  closed ball
of radius $R$ centered at the origin
is metrizable, we see that
$C([0,T]; Y^R_w)$ is a complete metric space.

To construct approximate solutions
for \eqref{equ1.1}, we choose a smooth
function $\theta : \mathbb{R}^d
\to \mathbb{R}$ such that
$0\le \theta (x) \le 1$
for all $x\in \mathbb{R}^d$ and
\begin{equation}\label{cutoff}
\theta (x)=\left\{\begin{array}{ll}
\!\!0, ~ \text{ if  } |x|\geq {\frac 34} ,\\
\!\!1,~ \text{ if  } |x|\leq \frac{1}{2}.
\end{array}\right.
\end{equation}
Given $n\in \mathbb{N}$,
let
$\theta_n (x) = \theta  \left ({\frac xn}
\right )$ for all $x\in \mathbb{R}^d$.

For every $n\in \mathbb{N}$,
consider the solution
$ \mathbf{u}^n$ of the  equation
defined  in    $\mathscr{D}_n
=\{x\in \mathbb{R}^d: |x|
<n \}$:
\begin{equation}\label{e1*}
\left\{\begin{array}{ll}
\!\!d\mathbf{u}^n=\Delta \mathbf{u}^n dt+\mathbf{u}^n\times \Delta \mathbf{u}^ndt-(1+|\mathbf{u}^n|^{2})\mathbf{u}^ndt+ \sum_{k= 1}^\infty(\mathbf{u}^n \times \mathbf{f}_k+\mathbf{f}_k)\circ d W_k,\\
\!\!\mathbf{u}^n|_{\partial\mathscr{D}_n}=0,\\
\!\!\mathbf{u}^n(0,x)=\theta_n
(x)\mathbf{u}_0(x), \quad x\in \mathbb{R}^d,
\end{array}\right.
\end{equation}
where $\mathbf{u}_0
\in H^1(\mathbb{R}^d)$.

By   \cite{B1},
we know that if   \eqref{1.1} holds,
then
 for every
 $n\in \mathbb{N}$ and
 $ \mathbf{u}_0
\in H^1(\mathbb{R}^d)$,
problem
\eqref{e1*}  has
 a unique solution
 $\mathbf{u}^n$ which is a
 weakly continuous $H^1(\mathscr{D}_n)$-valued progressively measurable
 stochastic process such that
 for every $T>0$
 and $p\ge 2$,
 $$
 \mathbf{u}^n\in L^p(\Omega;
 L^\infty(0,T; H_0^1(\mathscr{D}_n))\cap L^2(0,T;H^2(\mathscr{D}_n) ))  .
 $$
 The uniform estimates of the sequence
 $\{\mathbf{u}^n\}_{n=1}^\infty$
 are collected in the following lemma
 which can be found in
    \cite{B1}.

    \begin{lemma}\label{est1}
    Let     \eqref{1.1} hold
    and  $ \mathbf{u}_0
\in H^1(\mathbb{R}^d)$.
Then the
    sequence
 $\{\mathbf{u}^n\}_{n=1}^\infty$
 of solutions of \eqref{e1*} has the properties:

 {\rm (i)}. For every $p\ge 1$,
  $R>0$ and $T>0$,
 there exists a positive  number
 $C_1=C_1(p,R,T) $ independent of
  $n\in
 \mathbb{N}$ such that
 for all  $ \mathbf{u}_0
 \in H^1(\mathbb{R}^d)$
 with  $ \|\mathbf{u}_0\|_{
   H^1(\mathbb{R}^d)}
\le R $,
\begin{align} \label{1.50}
\mathrm{E}\left(\sup_{s\in[0,T]}
\|\mathbf{u}^n(s)\|^{2p}_{H_0^1(\mathscr{D}_n)}
\right)
+\mathrm{E}\left(\int_{0}^{T}
\|\mathbf{u}^n(s)\|^2_{H^2(\mathscr{D}_n)}ds
\right)^p
\leq C_1,
\end{align}
and
\begin{align} \label{1.50a}
\mathrm{E}
\left(\int_{0}^{T}\|
\mathbf{u}^n(s)
\times
\Delta \mathbf{u}^n(s)
\|^{\frac 43}
_{L^2 (\mathscr{D}_n)}ds
\right)^p
\leq C_1.
\end{align}

{\rm (ii)}.  For every
 $R>0$ and $T>0$,
 there exists a positive  number
 $C_2=C_2( R,T) $ independent of
  $n\in
 \mathbb{N}$ such that
 for all  $ \mathbf{u}_0
 \in H^1(\mathbb{R}^d)$
 with  $ \|\mathbf{u}_0\|_{
   H^1(\mathbb{R}^d)}
\le R $,
 \begin{equation} \label{1.50b}
  \mathrm{E}
  \left (
  \| \mathbf{u}^n\|^{16}
  _{ W^{{\frac 3{16}},  16 }
  	(0,T; L^2 (\mathscr{D}_n) )}
  \right )
  \le C_2.
  \end{equation}

  {\rm (iii)}.   For every
  $R>0$,
 there exists a positive  number
 $C_3=C_3(R) $ independent of
  $n\in
 \mathbb{N}$ such that
 for all  $t\ge 0$ and  $ \mathbf{u}_0
 \in H^1(\mathbb{R}^d)$
 with  $ \|\mathbf{u}_0\|_{
   H^1(\mathbb{R}^d)}
\le R $,
\begin{align} \label{1.50c}
\mathrm{E}\left(
\|\mathbf{u}^n(t)\|^2_{H_0^1(\mathscr{D}_n)}
\right)
+  \int_{0}^{t}
\mathrm{E} \left (
\|\mathbf{u}^n(s)\|^2_{H^2(\mathscr{D}_n)}
\right ) ds
\leq C_3 + C_3 t.
\end{align}
  \end{lemma}

     Next, we derive the uniform estimates
     on the tails of
       the sequence
     $\{\mathbf{u}^n\}_{n=1}^\infty$, which
     will be used to overcome
     the non-compactness of Sobolev
     embeddings in unbounded domains.

\begin{lemma}\label{tail_1}
If \eqref{1.1} holds,
then for every   $\varepsilon>0$,
$T>0$  and
 $\mathbf{u}_0
 \in H^1
 (\mathbb{R}^d)$,
 there exists
 $M=M(\varepsilon,
 T,  \mathbf{u}_0) \ge 1$
 such that
 for all $m\ge M$
    and $n\in \mathbb{N} $,
$$
 \mathrm{E}
\left ( \sup_{t\in [0,T]}
 \int_{m<|x| <n }
 |
 {\mathbf{ u}}^n(t,
\mathbf{u}_0 )(x) |^2 dx
\right )  <\varepsilon.
$$
\end{lemma}

\begin{proof}
Let
$\phi (x) =1-\theta (x)$ for all
$x\in \mathbb{R}^d$, where $\theta$ is
the smooth function given by \eqref{cutoff}.
Given $m\in \mathbb{N}$, denote by
   $\phi_m(x)=\phi(\frac{x}{m})$.
   For all $n, m \in \mathbb{N}$,
   by
     It\^{o}'s formula,  we  get
     for all $r, t\in [0,T]$
     with $r\le t$, $\mathrm{P}$-almost surely,
  \begin{align}\label{tail_1 p1}
& \|\phi_m\mathbf{u}^n
(r)
\|_{L^2
(\mathscr{D}_n) }^2
-
\|\phi_m\mathbf{u}^n
(0)
\|_{L^2
(\mathscr{D}_n) }^2\nonumber\\
&=2\int_0^r ( \Delta \mathbf{u}^n
(s), \phi^2_m\mathbf{u}^n(s ) )_{L^2
(\mathscr{D}_n) } ds
+2
\int_0^r
( \mathbf{u}^n(s)\times \Delta \mathbf{u}^n
(s), \phi^2_m\mathbf{u}^n(s) )_{L^2
(\mathscr{D}_n) } ds\nonumber\\
&\quad-2
\int_0^r
((1+|\mathbf{u}^n(s)|^{2})\mathbf{u}^n(s),
\phi^2 _m\mathbf{u}^n(s) )_{L^2
(\mathscr{D}_n) }ds
+2\sum_{k=1}^\infty
\int_0^r
(\phi_m \mathbf{f}_k  , \phi_m\mathbf{u}^n(s) )_{L^2
(\mathscr{D}_n)}
dW_k
\nonumber\\
&\quad+\sum_{k=1}^\infty
\int_0^r
\|\phi_m (\mathbf{u}^n(s) \times \mathbf{f}_k+\mathbf{f}_k)\|_{L^2
(\mathscr{D}_n) }^2ds
+\sum_{k=1}^\infty\int_0^r
(
(\mathbf{u}^n(s) \times \mathbf{f}_k)\times \mathbf{f}_k, \phi^2_m\mathbf{u}^n(s))_{L^2
(\mathscr{D}_n) }ds,
\end{align}
where we  have used the
identity  $(a\times b, a)=0$ to cancel the multiplicative noise.

For the first
term  on the right-hand side of
\eqref{tail_1 p1} we have for all $r\le t$,
\begin{align}\label{tail_1 p2}
&2\int_0^r ( \Delta \mathbf{u}^n (s), \phi_m^2
\mathbf{u}^n(s) )_{L^2
(\mathscr{D}_n) }ds
=-2\int_0^r (\nabla \mathbf{u}^n(s), \nabla
(\phi^2_m \mathbf{u}^n(s) ))_{L^2
(\mathscr{D}_n) } ds \nonumber\\
&= -2\int_0^r (\nabla \mathbf{u}^n
(s), \phi^2_m
\nabla\mathbf{u}^n(s) )_{L^2
(\mathscr{D}_n) } ds
-
\frac{4}{m}
\int_0^r
(\nabla\mathbf{u}^n (s) , \phi_m
\nabla \phi({x/m}) \mathbf{u}^n(s) )_{L^2
(\mathscr{D}_n) } ds \nonumber\\
 & \leq \frac{c_1}{m}
\int_0^t
\|\mathbf{u}^n (s) \|_{H^1_0
 (\mathscr{D}_n) }^2 ds,
\end{align}
where  $c_1>0$ is
a constant independent of both $n$ and  $m$.
For  the second and third
terms on the right-hand side of
\eqref{tail_1 p1} we have
\begin{align}\label{tail_1 p3}
2( \mathbf{u}^n
(s) \times \Delta \mathbf{u}^n (s), \phi^2_m
\mathbf{u}^n(s))_{L^2
(\mathscr{D}_n) }=2(\mathbf{u}^n
(s) \times (\phi^2_m \Delta \mathbf{u}^n (s) ), \mathbf{u}^n
(s) )_{L^2
(\mathscr{D}_n) }=0,
\end{align}
and
\begin{align}
-2
((1+|\mathbf{u}^n(s)|^{2})\mathbf{u}^n(s),
\phi^2 _m\mathbf{u}^n(s) )_{L^2
(\mathscr{D}_n) }=-2\|\phi_m\mathbf{u}^n(s)\|^2_{L^2
(\mathscr{D}_n) }-2(|\mathbf{u}^n(s)|^{2}, \phi^2 _m|\mathbf{u}^n(s)|^2)_{L^2
(\mathscr{D}_n) }\leq 0.
\end{align}
For the stochastic term
on the right-hand side of  \eqref{tail_1 p1},
by the BDG inequality we get
$$
2
\mathrm{E}
\left (\sup_{0\le r\le t}
\left |\int_0^r
\sum_{k=1}^\infty
 (\phi_m \mathbf{f}_k  , \phi_m\mathbf{u}^n(s) )_{L^2
(\mathscr{D}_n) }
dW_k \right |
\right )
\le 2 c_2
\mathrm{E}
\left (
\int_0^t
\sum_{k=1}^\infty
 \left |
 (\phi_m \mathbf{f}_k  , \phi_m\mathbf{u}^n(s) )_{L^2
(\mathscr{D}_n) }
 \right |^2  ds
\right )^{\frac 12}
$$
\be\label{tail_1 p3a}
\le
c_2 \sum_{k=1}^\infty
 \| \phi_m \mathbf{f}_k\|^2_{L^2(\mathbb{R}^d) }
 + c_2
 \int_0^t
 \mathrm{E}
\left ( \| \phi_m\mathbf{u}^n(s)
\|^2 _{L^2
(\mathscr{D}_n) }
\right )
ds.
\ee

 For the last two terms on the right-hand side
 of \eqref{tail_1 p1} we have
 for all $0\le r\le t$,
 \begin{align}\label{tail_1 p4}
& \sum_{k=1}^\infty
\int_0^r
\|\phi_m (\mathbf{u}^n(s) \times \mathbf{f}_k+\mathbf{f}_k)\|_{L^2
(\mathscr{D}_n) }^2ds
+\sum_{k=1}^\infty\int_0^r
(
(\mathbf{u}^n(s) \times \mathbf{f}_k)\times \mathbf{f}_k, \phi^2_m\mathbf{u}^n(s))_{L^2
(\mathscr{D}_n) }ds
   \nonumber\\
&\leq
2 \sum_{k=1}^\infty\int_0^t \|(\phi_m
\mathbf{u}^n(s) )\times \mathbf{f}_k\|
_{L^2(\mathscr{D}_n)}^2 ds +
2t \sum_{k=1}^\infty\|\phi_m
\mathbf{f}_k\|_{L^2(\mathscr{D}_n)}^2
 \nonumber\\
&\quad +\sum_{k=1}^\infty
\int_0^t \|((\phi_m
\mathbf{u}^n (s) )\times \mathbf{f}_k)\times \mathbf{f}_k\|
_{L^2(\mathscr{D}_n)}
\|\phi_m \mathbf{u}^n (s) \|_{L^2
(\mathscr{D}_n)
}ds \nonumber\\
&\leq 3 \sum_{k=1}^\infty\|\mathbf{f}_k\|^2_{L^\infty
(\mathbb{R} ^d) }
\int_0^t \|\phi_m
\mathbf{u}^n(s) \|_{L^2
(\mathscr{D}_n)
}^2 ds
+ 2t \sum_{k=1}^\infty\|\phi_m \mathbf{f}_k\|
_{L^2(\mathscr{D}_n) }^2 .
\end{align}

It follows from
  \eqref{tail_1 p1}-\eqref{tail_1 p4} that
  for all $t\in [0,T]$ and $n, m \in \mathbb{N}$,
$$
 \mathrm{E}
 \left (\sup_{0\le r \le t}
  \|\phi_m \mathbf{u}^n (r)
 \|_{L^2(\mathscr{D}_n)
 }^2
 \right )
 \leq
 \mathrm{E}
 \left (
 \|\phi_m \mathbf{u}^n (0)
 \|_{L^2(\mathscr{D}_n)
 }^2
 \right )
 $$
 $$
 +\left (c_2
 +
  3\sum_{k=1}^\infty\|\mathbf{f}_k\|^2_{L^\infty
(\mathbb{R} ^d) }
\right )
\int_0^t
 \mathrm{E}
 \left (
 \|\phi_m
\mathbf{u}^n (s) \|_{L^2
(\mathscr{D}_n)
}^2
\right ) ds
$$
\begin{equation}\label{tail_1 p5}
 + (2T + c_2) \sum_{k=1}^\infty\|\phi_m \mathbf{f}_k\|
_{L^2(\mathbb{R}^d)  }^2
 +
 \frac{c_1}{m}
 \int_0^t
 \mathrm{E}
 \left (
 \|\mathbf{u}^n (s) \|_{H_0^1(\mathscr{D}_n) }^2
 \right ) ds.
 \end{equation}

By \eqref{1.50}  and \eqref{tail_1 p5}
we find that there exists
$c_2 = c_2 (T, \mathbf{u}_0)>0$ such that
for all $t\in [0,T]$
and $n, m \in \mathbb{N}$,
$$
 \mathrm{E}
 \left (\sup_{0\le r \le t}
  \|\phi_m \mathbf{u}^n (r)
 \|_{L^2(\mathscr{D}_n)
 }^2
 \right )
 \leq
 \mathrm{E}
 \left (
 \|\phi_m \mathbf{u}^n (0)
 \|_{L^2(\mathscr{D}_n)
 }^2
 \right )
 $$
 $$
 +\left (c_2
 +
  3\sum_{k=1}^\infty\|\mathbf{f}_k\|^2_{L^\infty
(\mathbb{R} ^d) }
\right )
\int_0^t
 \mathrm{E}
 \left (\sup_{0\le r\le s}
  \|\phi_m
\mathbf{u}^n (r) \|_{L^2
(\mathscr{D}_n)
}^2
\right ) ds
$$
\begin{equation}\label{tail_1 p6}
 + (2T + c_2) \sum_{k=1}^\infty\|\phi_m \mathbf{f}_k\|
_{L^2(\mathbb{R}^d)  }^2
 +
 \frac{c_1 c_2}{m}.
 \end{equation}
 By \eqref{tail_1 p6}
  and Gronwall's lemma we obtain that
  for all $t\in [0,T]$
and $n, m \in \mathbb{N}$,
\be\label{tail_1 p7}
 \mathrm{E}
 \left (\sup_{0\le r\le t}
 \|\phi_m \mathbf{u}^n (r)
 \|_{L^2(\mathscr{D}_n)
 }^2
 \right )
 \leq
 \left (
 \|\phi_m \mathbf{u}_0
 \|_{L^2(\mathbb{R}^d )
 }^2
 + (2T+ c_2)  \sum_{k=1}^\infty\|\phi_m \mathbf{f}_k\|
_{L^2(\mathbb{R}^d)  }^2
 +
 \frac{c_1c_2}{m}
 \right )
 e^{ c_3 T
 },
 \ee
where $c_3
=   c_2 + 3\sum_{k=1}^\infty\|\mathbf{f}_k\|^2_{L^\infty
(\mathbb{R} ^d) }
 $.
 Note that as $m \to \infty$,
 $$
\|\phi_m \mathbf{u}_0
 \|_{L^2(\mathbb{R}^d )
 }^2
 + (2T+ c_2)
  \sum_{k=1}^\infty\|\phi_m \mathbf{f}_k\|
_{L^2(\mathbb{R}^d)  }^2
$$
$$
 \le
\int_{|x|>{\frac 12 m}}
 | \mathbf{u}_0 (x)|^2 dx
   + (2T+c_2) \int_{|x|>{\frac 12 m}}
    \sum_{k=1}^\infty|\mathbf{f}_k (x) |^2 dx
    \to 0,
  $$
which along with \eqref{tail_1 p7}
shows that  for every $\varepsilon>0$,
there exists $M=M(\varepsilon, T, \mathbf{u}_0)\ge 1$,
such that
for all $m\ge M$, $n\in \mathbb{N}$ and $t\in [0,T]$,
$$
\mathrm{E} \left (\sup_{0\le r\le t}
\int_{ {\frac 34 m}< |x|< n}
 | \mathbf{u}^n (r, \mathbf{u} _0) (x)|^2 dx
 \right )
 \le
 \mathrm{E}
 \left (\sup_{0\le r\le t}
 \|\phi_m \mathbf{u}^n (r, \mathbf{u}_0 )
 \|_{L^2(\mathscr{D}_n)
 }^2
 \right )
 <\varepsilon,
$$
as desired.
 \end{proof}

 Let
 $\mathbf{v}^n(t,x)
 = \theta_n (x) \mathbf{u}^n(t,x )$
 for all $n\in \mathbb{N}$,
 $t\ge 0$  and $x\in
 \mathscr{D}_n$.
 We then extend
 $\mathbf{v}^n(t,x)$
 for $x\in
  \mathbb{R}^d\setminus\mathscr{D}_n$
 by zero extension, and  thus consider
 $\mathbf{v}^n(t,x)$
 as a function defined   for all $t\ge 0$ and
     $x\in \mathbb{R}^d$.
     By \eqref{cutoff} we find that
     the sequence
     $\{\mathbf{v}^n\}_{n=1}^\infty$
     enjoys the uniform estimates
     as established by Lemmas \ref{est1}
     and \ref{tail_1}
     which are presented below.

      \begin{lemma}\label{est_v}
    Let     \eqref{1.1} hold
    and  $ \mathbf{u}_0
\in H^1(\mathbb{R}^d)$.
Then the
    sequence
 $\{\mathbf{v}^n\}_{n=1}^\infty$
   has the properties:

 {\rm (i)}. For every $p\ge 1$,
  $R>0$ and $T>0$,
 there exists a positive  number
 $C_1=C_1(p,R,T) $ independent of
  $n\in
 \mathbb{N}$ such that
 for all  $ \mathbf{u}_0
 \in H^1(\mathbb{R}^d)$
 with  $ \|\mathbf{u}_0\|_{
   H^1(\mathbb{R}^d)}
\le R $,
\begin{align} \label{est_v 1}
\mathrm{E}\left(\sup_{s\in[0,T]}
\|\mathbf{v}^n(s, \mathbf{u}_0 )\|^{2p}_{H^1 (\mathbb{R}^d) }
\right)
+\mathrm{E}\left(\int_{0}^{T}
\|\mathbf{v}^n(s, \mathbf{u}_0 )\|^2_{H^2 (\mathbb{R}^d) }ds
\right)^p
\leq C_1,
\end{align}
and
\begin{align} \label{est_v 2}
\mathrm{E}
\left(\int_{0}^{T}\|
\mathbf{v}^n(s, \mathbf{u}_0 )
\times
\Delta \mathbf{v}^n(s, \mathbf{u}_0 )
\|^{\frac 43}
_{L^2  (\mathbb{R}^d) }ds
\right)^p
\leq C_1.
\end{align}

{\rm  (ii)}.  For every
 $R>0$ and $T>0$,
 there exists a positive  number
 $C_2=C_2(R,T) $ independent of
  $n\in
 \mathbb{N}$ such that
 for all  $ \mathbf{u}_0
 \in H^1(\mathbb{R}^d)$
 with  $ \|\mathbf{u}_0\|_{
   H^1(\mathbb{R}^d)}
\le R $,
 \begin{equation} \label{est_v 3}
  \mathrm{E}
  \left (
  \| \mathbf{v}^n (\cdot, \mathbf{u}_0  )\|^{16}
  _{ W^{{\frac 3{16}} , 16 }
  	(0,T; L^2  (\mathbb{R}^d)  )}
  \right )
  \le C_2.
  \end{equation}

  {\rm (iii)}.   For every
  $R>0$,
 there exists a positive  number
 $C_3=C_3(R) $ independent of
  $n\in
 \mathbb{N}$ such that
 for all  $t\ge 0$ and  $ \mathbf{u}_0
 \in H^1(\mathbb{R}^d)$
 with  $ \|\mathbf{u}_0\|_{
   H^1(\mathbb{R}^d)}
\le R $,
\begin{align} \label{est_v 4}
\mathrm{E}\left(
\|\mathbf{v}^n(t, \mathbf{u}_0   )\|^2_{H^1 (\mathbb{R}^d)}
\right)
+  \int_{0}^{t}
\mathrm{E} \left (
\|\mathbf{v}^n(s, \mathbf{u}_0  )\|^2_{H^2 (\mathbb{R}^d) }
\right ) ds
\leq C_3 + C_3 t.
 \end{align}

{\rm (iv)}.  For every   $\varepsilon>0$,
$T>0$  and
 $\mathbf{u}_0
 \in H^1
 (\mathbb{R}^d)$,
 there exists
 $M=M(\varepsilon,
 T,  \mathbf{u}_0) \ge 1$
 such that
 for all $m\ge M$
     and $n\in \mathbb{N} $,
\be\label{est_v 5}
 \mathrm{E}
\left ( \sup_{t\in [0,T]}
 \int_{ |x| >m  }
 |
 {\mathbf{ v}}^n(t,
\mathbf{u}_0 )(x) |^2 dx
\right )  <\varepsilon.
\ee
  \end{lemma}

\subsection{Tightness of laws of approximate solutions}

In this subsection, we establish the
tightness of  the
sequence
of
distributions of
$\{\mathbf{v}^n
(\cdot, \mathbf{u}_0)
\}_{n=1}^\infty$,
which is given below.

\begin{lemma} \label{tight1}
If \eqref{1.1} holds,
then
 the
sequence
of
distributions of
$\{\mathbf{v}^n
(\cdot, \mathbf{u}_0)
\}_{n=1}^\infty$
 is tight in  the space $C([0,T];L^2 (\mathbb{R}^d))\cap C([0,T];H_w^1 (\mathbb{R}^d))\cap L_w^2(0,T;H^2 (\mathbb{R}^d))$,
 where
 $L_w^2(0,T;H^2 (\mathbb{R}^d))$ is the space
 $L^2(0,T;H^2 (\mathbb{R}^d))$ endowed with the weak topology.
\end{lemma}

\begin{proof}
By  \eqref{est_v 1}, \eqref{est_v 3}   and   Chebyshev's
 inequality, we deduce that for
 every
 $T>0$ and  $\mathbf{u}_0\in H^1
 (\mathbb{R}^d)$,
\begin{align*}
&\mathrm{P}\left(\sup_{t\in [0,T]}\|
 \mathbf{v}^n(t, \mathbf{u}_0)
  \|^2_{H^1(\mathbb{R}^d) }
  +\int_{0}^{T}\| \mathbf{v}^n(t, \mathbf{u}_0)
  \|^2_{H^2(\mathbb{R}^d)}dt
+\| \mathbf{v}^n(t, \mathbf{u}_0) \|^2_{W^{{
\frac 3{16}}, 16}(0,T; L^2
(\mathbb{R}^d)
)}> R\right)\nonumber\\
&\leq \frac{1}{R}\mathrm{E}
 \left(\sup_{t\in [0,T]}\|
 \mathbf{v}^n(t, \mathbf{u}_0)
  \|^2_{H^1(\mathbb{R}^d) }
  +\int_{0}^{T}\| \mathbf{v}^n(t, \mathbf{u}_0)
  \|^2_{H^2(\mathbb{R}^d)}dt
+\| \mathbf{v}^n(t, \mathbf{u}_0) \|^2_{W^{{
\frac 3{16}}, 16}(0,T; L^2
(\mathbb{R}^d)
)} \right)
  \nonumber\\
&\leq \frac{c_1}{R} ,
\end{align*}
where $c_1=c_1 (T, \mathbf{u}_0)>0$.
Therefore,  for every $\varepsilon>0$,
there exists $
 R(\varepsilon)>0$ such that
 \be  \label{tight1 p1}
 \mathrm{P}\left(\sup_{t\in [0,T]}\|
 \mathbf{v}^n(t, \mathbf{u}_0)
  \|^2_{H^1(\mathbb{R}^d) }
  +\int_{0}^{T}\| \mathbf{v}^n(t, \mathbf{u}_0)
  \|^2_{H^2(\mathbb{R}^d)}dt
+\| \mathbf{v}^n(t, \mathbf{u}_0) \|^2_{W^{{
\frac 3{16}}, 16}(0,T; L^2
(\mathbb{R}^d)
)}
>  R (\varepsilon) \right)
<{\frac 12} \varepsilon.
\ee

On the other hand,
by \eqref{est_v 5}, for
every $\varepsilon>0$, $T>0$
and $m\in \mathbb{N}$,  we find that
there exists   $N_m \in \mathbb{N}$ depending
only on
$m$,
  $T$ and $ \varepsilon$
  such that
  for all  $n\in \mathbb{N}$,
\begin{align}\label{tight1 p2}
\mathrm{E}\left(
\sup_{t\in [0,T]}
\int_{|x|>{\frac 12} N_m}
 \left |
 \mathbf{v}^n(t, \mathbf{u}_0) (x)
 \right |^2dx\right)\leq \frac{\varepsilon}{2^{2m+1}}.
\end{align}
By \eqref{tight1 p2}
and Chebyshev's inequality we obtain
that  for all $n\in \mathbb{N}$,
\begin{align*}
&\mathrm{P}
\left(\sup_{t\in [0,T]} \int_{|x|>{\frac 12} N_m}
 \left |
 \mathbf{v}^n(t, \mathbf{u}_0) (x)
 \right |^2dx
  > \frac{1}{2^m}\right)
  \leq \frac{\varepsilon}{2^{m+1}},
\end{align*}
and hence   for all $n\in \mathbb{N}$,
\begin{align}\label{tight1 p3}
\mathrm{P}\left(\bigcup_{m=1}^\infty
\left\{\sup_{t\in [0,T]}
\int_{|x|>{\frac 12} N_m}
 \left |
 \mathbf{v}^n(t, \mathbf{u}_0) (x)
 \right |^2dx
 > \frac{1}{2^m}
  \right\} \right )
 \leq \sum_{m=1}^\infty\frac{\varepsilon}{2^{m+1}}
 \leq \frac{\varepsilon}{2}.
\end{align}
Recall that
$\phi (x) =1-\theta (x)$ and
$\phi_m (x)= \phi
\left (\frac xm
\right )$ for all $x\in \mathbb{R}^d$.
By \eqref{tight1 p3} we get
that  for all    $n\in \mathbb{N}$,
 \begin{align}\label{tight1 p4}
 \mathrm{P}
 \left(
 \left\{\sup_{t\in [0,T]}  \|\phi_{N_m}
 \mathbf{v}^n(t, \mathbf{u}_0)
  \|_{L^2(\mathbb{R} ^d) }
  ^2\leq \frac{1}{2^m}, ~{\rm for~ all}~ m\in \mathbb{N}  \right\}\right)\geq 1-\frac{\varepsilon}{2}.
\end{align}

Define the sets
\begin{align}
&\mathcal{S}_1^\varepsilon=\left\{
{\mathbf{v}}: \sup_{t\in [0,T]}\|
{\mathbf{v}}(t)\|^2_{H^1(\mathbb{R} ^d)
}
+\int_{0}^{T}\|
{\mathbf{v}}(t)\|^2_{H^2(\mathbb{R} ^d)
}dt
+\| {\mathbf{v}}\|^2_{W^{{\frac 3{16}}, 16
}(0,T; L^2(\mathbb{R} ^d)
)}\leq R(\varepsilon)\right\},\label{tight1 p5}\\
&\mathcal{S}_2^\varepsilon=\left\{
{\mathbf{v}}: \sup_{t\in [0,T]} \|\phi_{N_m}
 {\mathbf{v}}(t)\|_{L^2
(\mathbb{R} ^d)
}^2\leq \frac{1}{2^m}, ~{\rm for~ all}~ m\in \mathbb{N}
 \right\},\label{tight1 p6}
\end{align}
and
$$\mathcal{S}^\varepsilon=\mathcal{S}_1^\varepsilon\bigcap \mathcal{S}_2^\varepsilon.$$
It follows from \eqref{tight1 p1}, \eqref{tight1 p4}-\eqref{tight1 p6} that for all $n\in \mathbb{N}$,
$$
\mathrm{P}(
\mathbf{v}^n(t, \mathbf{u}_0)
 \in \mathcal{S}^\varepsilon)> 1-\varepsilon.$$

It remains to verify that the set $\mathcal{S}^\varepsilon$ is compact in $$C([0,T];L^2
(\mathbb{R}^d ))\cap C([0,T];H_w^1
(\mathbb{R}^d))\cap L_w^2(0,T;H^2
(\mathbb{R}^d)).$$
Note that
every  topology of
$C([0,T];L^2
(\mathbb{R}^d ) )$,
$  C([0,T];H_w^1
(\mathbb{R}^d))$
and $   L_w^2(0,T;H^2
(\mathbb{R}^d))$
is metrizable on  $\mathcal{S}^\varepsilon$.
The precompactness
of $\mathcal{S}^\varepsilon$
 in $  L_w^2(0,T;H^2(\mathbb{R}^d))$ follows
 directly from its    boundedness
 in $  L^2(0,T;H^2(\mathbb{R}^d))$.

 We now  show that
 the set  $\mathcal{S}^\varepsilon$ is precompact in
  $C([0,T];L^2(\mathbb{R}^d) )$ by the
   Arzela-Ascoli theorem
 for which    we need to
 check  that
  $\mathcal{S}^\varepsilon$ is
  equicontinuous  in $C([0,T];L^2(\mathbb{R}^d) )$
  and
 for every $t\in [0,T]$,
 the set $
 \{\mathbf{v}(t):   \mathbf{v} \in
 \mathcal{S}^\varepsilon
 \} $
  is precompact in $L^2
  (\mathbb{R}^d)$.
  The  equicontinuity
  of
  $\mathcal{S}^\varepsilon$
   in $C([0,T];L^2(\mathbb{R}^d) )$
  follows from the embedding
  $  W^{{
\frac 3{16}}, 16}(0,T; L^2
(\mathbb{R}^d) )
\hookrightarrow
 C([0,T];L^2 (\mathbb{R}^d))$.

On the other hand, by \eqref{tight1 p6}
we see that
for every  $\eta>0$, there exists $m_0
\in \mathbb{N}$ such that for all
$ \mathbf{v} \in \mathcal{S}^\varepsilon$,
\begin{align}\label{tight1 p7}
\sup_{t\in [0,T]}
\|\phi_{N_{m_0}}  \mathbf{v}  (t)\|_{L^2
(\mathbb{R}^d) }^2<  \frac{\eta^2}{16}.
\end{align}
Note that for every $t\in[0,T]$,
 the set $\{ \mathbf{v}  (t)
 :   \mathbf{v}  \in \mathcal{S}^\varepsilon\}$
 is bounded in $H^1(\mathbb{R}^d)$.
 By the compactness of
 the  embedding
 $H^1 (\mathscr{D}_{m_0})
 \hookrightarrow
 L^2
 (\mathscr{D}_{m_0})$, we see that
 the set
 $\{
 (1-  \phi_{N_{m_0}})  \mathbf{v}  (t):
 \mathbf{v}  \in \mathcal{S}^\varepsilon\}$
 is precompact in
 $L^2
 (\mathscr{D}_{m_0})$,
 and hence
  it has a finite open cover of
  radius
  ${\frac 14} \eta $ in  $L^2
 (\mathscr{D}_{m_0})$,
  which along
  with \eqref{tight1 p7}
  implies that
  the set $\{ \mathbf{v}  (t)
 :   \mathbf{v}  \in \mathcal{S}^\varepsilon\}$
 has
  a finite open cover of
  radius
  $  \eta $ in  $L^2
 (\mathbb{R}^d )$.
   Due to the arbitrariness of $\eta>0$, we
   infer that
   the set $\{ \mathbf{v}  (t)
 :   \mathbf{v}  \in \mathcal{S}^\varepsilon\}$
   is precompact in   $L^2
 (\mathbb{R}^d )$,
 and hence
 the set  $\mathcal{S}^\varepsilon$ is precompact in
  $C([0,T];L^2(\mathbb{R}^d) )$ by  the Arzela-Ascoli theorem.

 We finally prove the set
   $\mathcal{S}^\varepsilon$ is precompact in
  $C([0,T];H^1_w (\mathbb{R}^d) )$.
  Let $\{\mathbf{z}_n\}_{n=1}^\infty$
  be a sequence in  $\mathcal{S}^\varepsilon$.
  We will prove  $\{\mathbf{z}_n\}_{n=1}^\infty$
  has a convergent subsequence
  in $C([0,T]; H^1_w(\mathbb{R}^d) )$.
  Since $\mathcal{S}^\varepsilon$ is precompact
  in $C([0,T];L^2(\mathbb{R}^d) )$,
  we know that
  there exists a subsequence
  $\{ \mathbf{z}_{n_k}\}_{k=1}^\infty$
  of $\{\mathbf{z}_n\}_{n=1}^\infty$
  and $ \mathbf{z} \in  C([0,T];L^2(\mathbb{R}^d) )$
  such
  that $ \mathbf{z}_{n_k} \to  \mathbf{z}$
  in $C([0,T];L^2(\mathbb{R}^d) )$.
  By the construction of
   $\mathcal{S}^\varepsilon$, we infer that
  $\mathbf{z}\in   \mathcal{S}^\varepsilon$, and
   for every $\psi\in H^2(\mathbb{R}^d)$,
   $$
    \sup_{t\in [0,T]}
    \left |
    ( \mathbf{z}_{n_k}(t) - \mathbf{z} (t), \psi)_{H^1(\mathbb{R}^d)}
    \right |
   = \sup_{t\in [0,T]}
    \left | ( \mathbf{z}_{n_k}(t)- \mathbf{z}(t),
    (I-\Delta) \psi)_{L^2(\mathbb{R}^d)}
    \right |
   $$
   $$
   \le
    \sup_{t\in [0,T]}
     \|  \mathbf{z} _{n_k}(t)- \mathbf{z} (t)\|_{L^2(\mathbb{R}^d)}
     \|
    (I-\Delta) \psi\|_{L^2(\mathbb{R}^d)}
    \to 0.
   $$
   By the
   density of
   $H^2(\mathbb{R}^d)$
   in $H^1(\mathbb{R}^d)$
   and the uniform boundedness of
   $\mathcal{S}^\varepsilon$ in
   $H^1(\mathbb{R}^d)$, we find that
   for every $\psi\in H^1(\mathbb{R}^d)$,
   $$
    \sup_{t\in [0,T]}
    \left |
    ( \mathbf{z}_{n_k}(t) - \mathbf{z} (t), \psi)_{H^1(\mathbb{R}^d)}
    \right | \to 0,
    $$
    and thus
     $ \mathbf{z}_{n_k} \to  \mathbf{z}$
  in $C([0,T]; H^1_w(\mathbb{R}^d) )$, as desired.
\end{proof}

By Lemma \ref{tight1} and   the Skorokhod-Jakubowski representation theorem, we obtain the following result.

\begin{lemma}
\label{sj1}
If  \eqref{1.1} holds, then
there exist a   stochastic basis
 $(\widetilde{\Omega}, \widetilde{\mathcal{F}},
  \{\widetilde{\mathcal{F}}_t\}
  _{t\ge 0},  \widetilde{\mathrm{P}})$,
 random variables
 $(\widetilde{\mathbf{v}}, \widetilde{W} )$
 and
  $(\widetilde{\mathbf{v}}^n, \widetilde{W}^n)$,
  $n\in \mathbb{N}$,
   defined on $(\widetilde{\Omega}, \widetilde{\mathcal{F}},
    \{\widetilde{\mathcal{F}}_t\}
  _{t\ge 0},  \widetilde{\mathrm{P}})$ such that
   in the space $(C([0,T];L^2
   (\mathbb{R}^d)
   )\cap C([0,T];H_w^1(\mathbb{R}^d))
   \cap L_w^2(0,T;H^2(\mathbb{R}^d)))
    \times C([0,T];\mathbb{R}^\infty)$:

{\rm (i)}. The law  of $(\widetilde{\mathbf{v}}^n, \widetilde{W}^n)$ is
the same as    $( {\mathbf{v}}^n, W)$
for every $n\in \mathbb{N}$.

{\rm (ii)}.   $(\widetilde{\mathbf{v}}^n, \widetilde{W}^n)\rightarrow (\widetilde{\mathbf{v}}, \widetilde{W} )$,
 $\widetilde{\mathrm{P}}$-almost surely.

{\rm (iii)}.  The random variables
$ \widetilde{\mathbf{v}}$ and
 $\widetilde{\mathbf{v}}^n$ share the same
 uniform estimates  with
 $  {\mathbf{v}^n}$
 as
 given by Lemma \ref{est_v}.
     In particular,
for every  $n\in \mathbb{N}$,
$T>0$ and $p\ge 1$,
\begin{align}\label{sj1 1}
&\widetilde{\mathrm{E}}\left(\sup_{t\in[0,T]}\|\widetilde{\mathbf{v}}^n(t)\|^{2p}_{H^1
(\mathbb{R}^d)
}\right)
+\widetilde{\mathrm{E}}\left(\int_{0}^{T}\|\widetilde{\mathbf{v}}^n(t)\|^2_{H^2
(\mathbb{R}^d)
}dt \right )^p
\leq C_1,
\end{align}
and
\begin{align}\label{sj1 2}
&\widetilde{\mathrm{E}}\left(\sup_{t\in[0,T]}\|\widetilde{\mathbf{v}} (t)\|^{2p}_{H^1
(\mathbb{R}^d)
}\right)
+\widetilde{\mathrm{E}}\left(\int_{0}^{T}\|\widetilde{\mathbf{v}} (t)\|^2_{H^2
(\mathbb{R}^d)
}dt \right )^p
\leq C_1,
\end{align}
where $C_1>0$ is the same
number as in \eqref{est_v 1}.

{\rm (iv)}.
$\widetilde{\mathbf{v}}^n\rightarrow \widetilde{\mathbf{v}}$ in $L^2(\widetilde{\Omega}; L^2(0,T;H^1
 (\mathbb{R}^d)
))$.
 \end{lemma}

\begin{proof}
We  will
verify that all conditions
of
Proposition  \ref{prop_sj}
are fulfilled.
We first show that
there exists a sequence of   continuous real-valued
functions
defined in each of the space
    appearing in
$$(C([0,T];L^2(\mathbb{R}^d))
\cap C([0,T];H_w^1(\mathbb{R}^d))\cap
 L_w^2(0,T;H^2(\mathbb{R}^d) ))
   \times C([0,T];\mathbb{R}^\infty),$$
which separate  points.
Such  continuous functions
indeed exist in
  $C([0,T];L^2(\mathbb{R}^d))$
  and  $C([0,T];\mathbb{R}^\infty)$
  since they are both   separable Banach
spaces.

Since
$ L ^2(0,T;H^2(\mathbb{R}^d))$
is a separable Hilbert space,
any countable dense subset
of $ \left ( L ^2(0,T;H^2(\mathbb{R}^d))
\right )^*$
forms a sequence
of continuous
functions
on $ L _w^2(0,T;H^2(\mathbb{R}^d))$
that separate points.

 Let $\{\psi_n\}_{n=1}^\infty$
 be a dense subset of $H^1(\mathbb{R}^d) $
 and $\{r_j\}_{j=1}^\infty$
 be the set of all rational numbers in $[0,T]$.
 Given $n, j\in \mathbb{N}$,
 define
 $g_{n,j}:     C([0,T];H_w^1(\mathbb{R}^d)
 ) \to \mathbb{R}$
  by
  $$g_{n,j}( \xi)=(\xi (r_j ),  \psi_n   )_{H^1(\mathbb{R}^d)},
  \quad \forall \ \xi \in
   C([0,T];H_w^1(\mathbb{R}^d)
 ).
 $$
 Then
 $g_{n,j}$ is continuous
 on $
 C([0,T];H_w^1(\mathbb{R}^d)
 )$ and
  the sequence
 $\{ g_{n,j}\}_{n,j=1}^\infty$
 separates  the points of
 $C([0,T];H_w^1(\mathbb{R}^d)
 )$.

  On the other hand,
  by
  Lemma \ref{tight1} we see that
  the distributions
  of the sequence
  $\{(\mathbf{v}^n, W)\}_{n=1}^\infty$
  are  tight in
   $(C([0,T];L^2(\mathbb{R}^d))\cap
   C([0,T];H_w^1(\mathbb{R}^d))
   \cap L_w^2(0,T;H^2(\mathbb{R}^d)))
   \times C([0,T];\mathbb{R}^\infty)$,
   and hence
   by  Proposition
   \ref{prop_sj},
   there exist a
   subsequence
   of $\{(\mathbf{v}^n, W)\}_{n=1}^\infty$
   (which is not relabeled),
   a
    probability space $(\widetilde{\Omega}, \widetilde{\mathcal{F}}, \widetilde{\mathrm{P}})$,
 random variables
 $(\widetilde{\mathbf{v}}, \widetilde{W} )$
 and
  $(\widetilde{\mathbf{v}}^n, \widetilde{W}^n)$,
  $n\in \mathbb{N}$,
   defined on $(\widetilde{\Omega}, \widetilde{\mathcal{F}}, \widetilde{\mathrm{P}})$
   such that (i) and (ii) are fulfilled.
   Then the uniform estimates of \eqref{sj1 1}
   and \eqref{sj1 2} follows from
   (i) and Lemma \ref{est_v}
   immediately.

     Finally, we show (iv).
 By (ii), we know
 that  $\widetilde{\mathbf{v}}^n\rightarrow \widetilde{\mathbf{v}}$ in $C([0,T];L^2(\mathbb{R}^d) )$,
   $\widetilde{\mathrm{P}}$-almost
 surely.
 Then, by \eqref{sj1 1}-\eqref{sj1 2}
   and Vitali's theorem, we
   find that  $\widetilde{\mathbf{v}}^n\rightarrow \widetilde{\mathbf{v}}$ in
   $L^2(\widetilde{\Omega}; C([0,T];L^2(\mathbb{R}^d)))$,
   and hence
\begin{align*}
&\widetilde{\mathrm{E}}\int_{0}^{T}\|\widetilde{\mathbf{v}}^n(t)- \widetilde{\mathbf{v}}(t)\|_{H^1(\mathbb{R}^d)}^2dt=\widetilde{\mathrm{E}}
\int_{0}^{T}(\widetilde{\mathbf{v}}^n(t)
- \widetilde{\mathbf{v}}(t),
(I-\Delta) (\widetilde{\mathbf{v}}^n(t)- \widetilde{\mathbf{v}}(t)))_{L^2 (\mathbb{R}^d) }dt\nonumber\\
&\leq \widetilde{\mathrm{E}}\int_{0}^{T}\|\widetilde{\mathbf{v}}^n(t)- \widetilde{\mathbf{v}}(t)\|_{H^2(\mathbb{R}^d)
}\|\widetilde{\mathbf{v}}^n(t)
- \widetilde{\mathbf{v}}(t)\|_{L^2(\mathbb{R}^d)}dt\nonumber\\
&\leq\widetilde{ \mathrm{E}}\left(\sup_{t\in [0,T]}\|\widetilde{\mathbf{v}}^n(t)- \widetilde{\mathbf{v}}(t)\|_{L^2
(\mathbb{R}^d)
}\int_{0}^{T}\|\widetilde{\mathbf{v}}^n
(t)
- \widetilde{\mathbf{v}}(t)\|_{H^2
(\mathbb{R}^d)
}dt\right)
\nonumber\\
  &\leq
  T^{\frac 12} \left[\widetilde{\mathrm{E}}\left(\sup_{t\in [0,T]}\|\widetilde{\mathbf{v}}^n(t)- \widetilde{\mathbf{v}}(t)\|^2
  _{L^2
  (\mathbb{R}^d)
  }\right)\right]^\frac{1}{2}
  \left[\widetilde{\mathrm{E}}
  \int_{0}^{T}
   \|\widetilde{\mathbf{v}}^n(t)
- \widetilde{\mathbf{v}}(t)\|^2_{H^2
(\mathbb{R}^d)
}dt
   \right]^\frac{1}{2}
   \nonumber\\
  &\leq
  T^{\frac 12} \left[\widetilde{\mathrm{E}}\left(\sup_{t\in [0,T]}\|\widetilde{\mathbf{v}}^n(t)- \widetilde{\mathbf{v}}(t)\|^2
  _{L^2
  (\mathbb{R}^d)
  }\right)\right]^\frac{1}{2}
  \left[2
  \widetilde{\mathrm{E}}
  \int_{0}^{T}
   (\|\widetilde{\mathbf{v}}^n (t)\|^2_{H^2
(\mathbb{R}^d)
}
   +
 \|\widetilde{\mathbf{v}}(t)\|^2_{H^2
(\mathbb{R}^d)
}
)
dt
   \right]^\frac{1}{2}
   \to 0,
\end{align*}
that is,
$\widetilde{\mathbf{v}}^n\rightarrow \widetilde{\mathbf{v}}$ in $L^2(\widetilde{\Omega}; L^2(0,T;H^1
 (\mathbb{R}^d)
))$, which
 completes the proof.
\end{proof}

\subsection{Passage to
	 the limit
of approximate solutions}
We are now in a position
to complete the proof of
Theorem \ref{t3} by
showing  the process
$\widetilde{\mathbf{v}}$
of the limit of the
sequence
$\{
\widetilde{\mathbf{v}}^n
\}_{n=1}^\infty$
is a solution of \eqref{equ1.1}.

{\bf Proof of Theorem  \ref{t3}}.
If
$\varphi\in C_0^\infty(\mathbb{R}^d)$,
then by \eqref{e1*} we have,
for all $t\in [0,T]$
and $n\in \mathbb{N}$,
 $$
  (\mathbf{u}^n
  (t), \varphi )_{L^2
  (\mathscr{D}_n) }
=  (\mathbf{u}^n
(0), \varphi )_{L^2
	(\mathscr{D}_n) }
   +
  \int_0^t
  (\Delta \mathbf{u}^n
  (s), \varphi)_{L^2
  	(\mathscr{D}_n) } ds
  $$
  $$
  +
  \int_0^t
  ( \mathbf{u}^n
  (s) \times \Delta \mathbf{u}^n
  (s), \varphi )
  _{L^2
  	(\mathscr{D}_n) }
   ds
   -
  \int_0^t
  \left (
  (1+|\mathbf{u}^n (s) |^{2})\mathbf{u}^n
  (s),
  \varphi \right )_{L^2
  	(\mathscr{D}_n) } ds
  $$
  \be\label{ps 0}
  +
  {\frac 12}
  \sum_{k=1}^\infty
  \int_0^t
  \left (
  (\mathbf{u}^n
  (s)  \times \mathbf{f}_k)
  \times \mathbf{f}_k ,
  \varphi
  \right )_{L^2
  	(\mathscr{D}_n) } ds
  +
  \sum_{k=1}^\infty
   \int_0^t
  (\mathbf{u}^n
  (s)  \times \mathbf{f}_k
  +\mathbf{f}_k, \varphi )
  _{L^2
  	(\mathscr{D}_n) }
   d W_k.
  \ee
  Since
  $\varphi$ has a compact
  support in $\mathbb{R}^d$,
  there exists
  $n_0\in \mathbb{N}$ such that
  $\varphi (x)
  =0$
  for all
  $x\in
  \mathbb{R}^d
  \setminus  \mathscr{D}_{n_0}$.
  Therefore,
  we have
  for all $n\ge n_0$,
   \be\label{ps 1}
   (\mathbf{u}^n
  (t), \varphi )_{L^2
  	(\mathscr{D}_n) }
  =\int_{\mathscr{D}_{n_0
  	}}
   \mathbf{u}^n
  (t, x)  \varphi (x)
  dx.
 \ee
  Recall that
   $\mathbf{v}^n
  (t,x)
  =\theta_n (x)
   \mathbf{u}^n
  (t,x)$
  and $\theta_n (x)
  =1$
  for $|x|\le {\frac 12}n$,
  we know that
    $\mathbf{v}^n
  (t,x)
  =
  \mathbf{u}^n
  (t,x)$
  for all
  $t\ge 0$
  and
   $ |x|\le {\frac 12}n$, and
   hence
   for all $n \ge 2 n_0$,
   \be\label{ps 1a}
   \int_{\mathscr{D}_{n_0
   }}
   \mathbf{u}^n
   (t, x)  \varphi (x)
   dx
   =
   \int_{\mathscr{D}_{n_0
   }}
   \mathbf{v}^n
   (t, x)  \varphi (x)
   dx
   =
   \int_{
   	\mathbb{R}^d}
   \mathbf{v}^n
   (t, x)  \varphi (x)
   dx.
   \ee
   Notice that
   $\mathbf{v}^n
   (t, x)$
   is defined for all
   $x\in \mathbb{R}^d$,
   while
    $\mathbf{u}^n
   (t, x)$
   is defined only
   for
   $x\in \mathscr{D}_n$.
   By
  \eqref{ps 1}-\eqref{ps 1a} we get,
   for all $n \ge 2 n_0$,
    \be\label{ps 2}
   (\mathbf{u}^n
   (t), \varphi )_{L^2
   	(\mathscr{D}_n) }
   =(\mathbf{v }^n
   (t), \varphi )_{L^2
   	 (\mathbb{R}^d)  }.
  \ee
  By the same reason,
  for all $n\ge 2n_0$,
  the inner product
  in
  ${L^2
  	(\mathscr{D}_{n} ) }
  $
  of every term
  in \eqref{ps 0}
  can be replaced
  by the inner product
  in
  ${L^2
  	( \mathbb{R}^d ) }
  $,
  and the variable
   $\mathbf{u}^n$
   can be replaced by
    $\mathbf{v}^n$.
    Consequently,
    for
    all $t \in [0,T]$
    and $n\ge 2n_0$, we  have,
    $$
    (\mathbf{v}^n
    (t), \varphi )_{L^2
   ( \mathbb{R}^d ) }
    =  (\theta_n \mathbf{u}_0
    , \varphi )_{L^2
    	( \mathbb{R}^d) }
    +
    \int_0^t
    (\Delta \mathbf{v}^n
    (s), \varphi)_{L^2
    	( \mathbb{R}^d) } ds
    $$
    $$
    +
    \int_0^t
    ( \mathbf{v}^n
    (s) \times \Delta \mathbf{v}^n
    (s), \varphi )
    _{L^2
    	( \mathbb{R}^d) }
    ds
    -
    \int_0^t
    \left (
    (1+|\mathbf{v}^n (s) |^{2})\mathbf{v}^n
    (s),
    \varphi \right )_{L^2
    	( \mathbb{R}^d) } ds
    $$
    \be\label{ps 3}
    +
    {\frac 12}
   \sum_{k=1}^\infty
    \int_0^t
    \left (
    (\mathbf{v}^n
    (s)  \times \mathbf{f}_k)
    \times \mathbf{f}_k,
    \varphi
    \right )_{L^2
    	( \mathbb{R}^d ) } ds
    +
    \sum_{k=1}^\infty
    \int_0^t
    (\mathbf{v}^n
    (s)  \times \mathbf{f}_k
    +\mathbf{f}_k, \varphi )
    _{L^2
    	( \mathbb{R}^d) }
    d W_k.
    \ee
    By Lemma \ref{sj1} (i),
    we see that
        $(\widetilde{\mathbf{v}}^n, \widetilde{W}^n)$ has the same law
        as      $( {\mathbf{v}}^n, W)$
    for every $n\in \mathbb{N}$.
    Then by the argument of
    \cite{cww, ww}, we find that
      $(\widetilde{\mathbf{v}}^n, \widetilde{W}^n)$
      also satisfies
      system \eqref{ps 3}.
      From now on, we will
      deal with the
      variables
        $(\widetilde{\mathbf{v}}^n, \widetilde{W}^n)$
        instead of
         $( {\mathbf{v}}^n, W)$,
         and for the
         sake of  simplicity
        in notation,
         we  will
         use
        $ {\mathbf{v}}^n $
        for both
        the variables
          $ {\mathbf{v}}^n $
          and
            $ \widetilde{\mathbf{v}}^n $,
            and use  $  {W}^n $
            for  $\widetilde{W}^n $.

    By Lemma \ref{sj1} (iv) we infer that
    for every
    $\psi \in
    L^\infty(\Omega\times [0,T])$,
    $$
    \mathrm{E}
    \left (
    \int_0^T
    \psi (t)
    \left (
    \int_0^t
    \left (
    (\Delta \mathbf{v}^n
    (s), \varphi)_{L^2
    	( \mathbb{R}^d) }
    	-
    	(\Delta \mathbf{v}
    (s), \varphi)_{L^2
    	( \mathbb{R}^d) }
    	\right )
    	 ds
    \right ) dt
    \right )
    $$
    $$
    \le
    \mathrm{E}
    \left (
    \int_0^T
    \psi (t)
    \left (
    \int_0^t
    \| \nabla \mathbf{v}^n
    (s) -\nabla \mathbf{v}
    (s)\| _{L^2
    	( \mathbb{R}^d) }
    	\| \nabla \varphi\|_{L^2
    	( \mathbb{R}^d) }
    	 ds
    \right ) dt
    \right )
    $$
 $$
    \le T
    \|\psi\|_{L^\infty
    (\Omega\times [0,T])}
    	\|\nabla  \varphi\|_{L^2
    	( \mathbb{R}^d) }
    \mathrm{E}
    \left (
    \int_0^T
    \| \nabla \mathbf{v}^n
    (s) -\nabla \mathbf{v}
    (s)\| _{L^2
    	( \mathbb{R}^d) }
    	 ds
    \right )
    $$
  \be\label{ps 4}
    \le T^{\frac 32}
    \|\psi\|_{L^\infty
    (\Omega\times [0,T])}
    	\|\nabla  \varphi\|_{L^2
    	( \mathbb{R}^d) }
     \|   \mathbf{v}^n
    -  \mathbf{v}
     \| _{L^2(\Omega, L^2(0,T; H^1
    	( \mathbb{R}^d) )) }
    	\to 0,
     \ee
 as $n\to \infty$.

 For the third term on the right-hand side
 of \eqref{ps 3}, we have for all
 $t\in [0,T]$,
\begin{align*}
&
\left |
\int_{0}^{t}(\mathbf{v}^n(s)\times \Delta \mathbf{v}^n(s)-\mathbf{v}(s)\times \Delta \mathbf{v}(s), \varphi)
_{L^2
(\mathbb{R}^d)}
ds
\right |
\nonumber\\
&=\left |
\int_{0}^{t}((\mathbf{v}^n(s)-\mathbf{v}(s))\times \Delta \mathbf{v}^n(s)+\mathbf{v}(s)\times \Delta (\mathbf{v}^n(s)-\mathbf{v}(s)), \varphi)
_{L^2
(\mathbb{R}^d)}
ds
\right |
\nonumber\\
& \le
\left |
\int_{0}^{t}((\mathbf{v}^n(s)-\mathbf{v}(s))\times \Delta \mathbf{v}^n(s), \varphi)
_{L^2
(\mathbb{R}^d)}
ds \right |\nonumber\\
&\quad +
\left |
\int_{0}^{t}(\nabla\mathbf{v}(s)\times \nabla(\mathbf{v}^n(s)-\mathbf{v}(s)), \varphi)
_{L^2
(\mathbb{R}^d)}
+
(\mathbf{v}(s)\times \nabla(\mathbf{v}^n(s)-\mathbf{v}(s)), \nabla\varphi)
_{L^2
(\mathbb{R}^d)}
ds
\right |
\nonumber\\
&\leq
\|\varphi \|_{L^\infty
(\mathbb{R}^d)}
\int_{0}^{t}\|\mathbf{v}^n(s)-\mathbf{v}(s)\|
_{L^2
(\mathbb{R}^d)}
 \|\Delta \mathbf{v}^n(s)\|
 _{L^2
(\mathbb{R}^d)}
 ds\nonumber\\
&\quad+
\|\varphi \|_{L^\infty
(\mathbb{R}^d)}
\int_{0}^{t}
\|\nabla \mathbf{v}(s)\|
_{L^2
(\mathbb{R}^d)}
\|\nabla \mathbf{v}^n(s)-
\nabla \mathbf{v}(s)\| _{L^2
(\mathbb{R}^d)}ds \nonumber\\
 &\quad
 +
 \|\nabla \varphi
 \|_{L^\infty
(\mathbb{R}^d)}
\int_0^t
\| \mathbf{v}(s)\|_{L^2
(\mathbb{R}^d)}
  \|\nabla \mathbf{v}^n(s)-
  \nabla \mathbf{v}(s)\|
  _{L^2
(\mathbb{R}^d)} ds\nonumber\\
&\leq c_1
\left(\int_{0}^{T}\|\mathbf{v}^n(t)-\mathbf{v}(t)\|^2
_{H^1
(\mathbb{R}^d)}
 dt\right)^\frac{1}{2}
 \left(\int_{0}^{T}
 \left (
 \| \mathbf{v}^n(t)\|^2
 _{H^2
(\mathbb{R}^d)}
+
\| \mathbf{v} (t)\|^2
 _{H^1
(\mathbb{R}^d)}
\right )
 dt\right)^\frac{1}{2},
\end{align*}
where
$c_1>0$ depends only
on $\varphi$.
Then by Lemma \ref{sj1} (iii)-(iv)
we obtain that
for any
$\psi
\in L^\infty
(\Omega \times [0,T])$,
 \begin{align*}
 &
 \mathrm{E}
 \left (
 \int_{0}^{T}
 \psi (t)
 \left (
 \int_{0}^{t}(\mathbf{v}^n(s)\times \Delta \mathbf{v}^n(s)-\mathbf{v}(s)\times \Delta \mathbf{v}(s), \varphi)
_{L^2
(\mathbb{R}^d)}
ds
\right ) dt
\right ) \nonumber\\
&
\le
\|\psi\|_{L^\infty
(\Omega
\times [0,T])}
\mathrm{E}
\left (
\int_{0}^{T}\left|\int_{0}^{t}(\mathbf{v}^n(s)\times \Delta \mathbf{v}^n(s)-\mathbf{v}(s)\times \Delta \mathbf{v}(s), \varphi)ds\right|dt
\right )\nonumber\\
&\leq   c_1 T
\|\psi\|_{L^\infty
(\Omega
\times [0,T])}
\mathrm{E}
\left (
 \left(\int_{0}^{T}\|\mathbf{v}^n(t)-\mathbf{v}(t)\|^2
_{H^1
(\mathbb{R}^d)}
 dt\right)^\frac{1}{2}
 \left(\int_{0}^{T}
 \left (
 \| \mathbf{v}^n(t)\|^2
 _{H^2
(\mathbb{R}^d)}
+
\| \mathbf{v} (t)\|^2
 _{H^1
(\mathbb{R}^d)}
\right )
 dt\right)^\frac{1}{2}
 \right )
 \end{align*}
\be\label{ps 5}
\leq
 c_2
\| \mathbf{v}^n -\mathbf{v} \|
_{L^2
(\Omega, L^2(0,T;
 H^1
(\mathbb{R}^d) ))}
  \left (
 \| \mathbf{v}^n \|
 _{L^2(\Omega, L^2(0,T;
 H^2
(\mathbb{R}^d) ))}
+
\| \mathbf{v}  \|
_{L^2(\Omega, L^2(0,T;
 H^1
(\mathbb{R}^d) ))}
\right )
 \rightarrow 0,
\ee
as $n\to \infty$, where $c_2
=  c_1 T
\|\psi\|_{L^\infty
(\Omega
\times [0,T])}$.

 For the fourth term on the right-hand
 side of \eqref{ps 3}, by
 the embedding
 $H^1
(\mathbb{R}^d)
\hookrightarrow
L^4
(\mathbb{R}^d)$, we have
 $$
\left |
\int_{0}^{t}(|\mathbf{v}^n(s)|^{2}\mathbf{v}^n(s)-|\mathbf{v}(s)|^{2}\mathbf{v}(s), \varphi)_{L^2(\mathbb{R}^d)}ds
\right |
$$
$$
\leq
{\frac 32}
\|\varphi\|_{L^\infty
(\mathbb{R}^d)}
\int_{0}^{t}
\left (
\|\mathbf{v}^n(s)\|^2_{L^4
(\mathbb{R}^d)}
+
\|\mathbf{v}(s)\|^2_{L^4
(\mathbb{R}^d)}
\right )
\|\mathbf{v}^n(s)-\mathbf{v}(s)\|
_{L^2
(\mathbb{R}^d)} ds
 $$
 \be\label{ps 6}\leq
{\frac 32}T^{\frac 12}
\|\varphi\|_{L^\infty
(\mathbb{R}^d)}
\left (\sup_{s\in [0,T]}
\|\mathbf{v}^n(s)\|^2_{H^1
(\mathbb{R}^d)}
+ \sup_{s\in [0,T]}
\|\mathbf{v}(s)\|^2_{H^1
(\mathbb{R}^d)}
\right )
\|\mathbf{v}^n -\mathbf{v} \|
_{L^2(0,T; L^2
(\mathbb{R}^d))} .
 \ee
By
\eqref{ps 6}
and Lemma \ref{sj1} (iii)-(iv)
we obtain that
for any
$\psi
\in L^\infty
(\Omega \times [0,T])$,
 $$
 \mathrm{E}
 \left ( \int_0^T
 \psi (t)
 \left  (
\int_{0}^{t}(|\mathbf{v}^n(s)|^{2}\mathbf{v}^n(s)-|\mathbf{v}(s)|^{2}\mathbf{v}(s), \varphi)_{L^2(\mathbb{R}^d)}ds
\right ) dt
\right )
 $$
 $$
 \le
 \|\psi\|_{L^\infty
 (\Omega \times [0,T]}
 \mathrm{E}
 \left ( \int_0^T
 \left  |
\int_{0}^{t}(|\mathbf{v}^n(s)|^{2}\mathbf{v}^n(s)-|\mathbf{v}(s)|^{2}\mathbf{v}(s), \varphi)_{L^2(\mathbb{R}^d)}ds
\right |dt
\right )
 $$
 $$
 \le  c_3
 \mathrm{E}
 \left (
 \left (\sup_{s\in [0,T]}
\|\mathbf{v}^n(s)\|^2_{H^1
(\mathbb{R}^d)}
+ \sup_{s\in [0,T]}
\|\mathbf{v}(s)\|^2_{H^1
(\mathbb{R}^d)}
\right )
\|\mathbf{v}^n -\mathbf{v} \|
_{L^2(0,T; L^2
(\mathbb{R}^d))}
\right )
$$
\be\label{ps 7}
 \le  c_3
 \left (
\|\mathbf{v}^n \|^2_{
L^4(\Omega,
L^\infty
(0,T; H^1
(\mathbb{R}^d) ))}
+
\|\mathbf{v}  \|^2_{
L^4(\Omega,
L^\infty
(0,T; H^1
(\mathbb{R}^d) ))}
\right )
\|\mathbf{v}^n -\mathbf{v} \|
_{L^2(\Omega, L^2(0,T; L^2
(\mathbb{R}^d)))}
\to 0,
\ee
as $n\to \infty$, where
 $c_3
 = {\frac 32} T^{\frac 32}
\|\varphi\|_{L^\infty
(\mathbb{R}^d)}
 \|\psi\|_{L^\infty
 (\Omega \times [0,T])}$.

 Similarly,
 for the fifth term on the right-hand side
 of
 \eqref{ps 3} we have
 for any
$\psi
\in L^\infty
(\Omega \times [0,T])$,
\be\label{ps 8-2}
 \mathrm{E}
 \left ( \int_0^T
 \psi (t)
 \left  (
\sum_{k=1}^\infty
\int_{0}^{t}
(
(\mathbf{v}^n(s)
\times f_k)
\times f_k
- (\mathbf{v} (s)
\times f_k)
\times f_k, \ \varphi)_{L^2(\mathbb{R}^d)} ds
\right ) dt
\right )
\to 0,
\ee
 as $n\to \infty$.

 We next consider the convergence of
 the last term on the
 right-hand side
 of \eqref{ps 3}.
 By  \eqref{1.1} and Lemma \ref{sj1} (iv), we see
 that for almost all  $\omega \in \Omega$,
\begin{align*}
\sum_{k=1}^\infty\int_{0}^{t}(\mathbf{v}^n(s)\times \mathbf{f}_k-\mathbf{v}(s)\times \mathbf{f}_k,
\ \varphi)
_{L^2(\mathbb{R}^d)}
ds \rightarrow 0,
\quad \text{as } n \to \infty,
\end{align*}
in $L^2(0,T; \mathbb{R})$, which along
with    \cite[Lemma 2.1]{ANR} shows that
\begin{align}\label{ps 8}
\sum_{k=1}^\infty\int_{0}^{t}(( \mathbf{v}^n(s)\times
 \mathbf{f}_k+ \mathbf{f}_k) , \varphi)
 _{L^2(\mathbb{R}^d)} dW^n_k
 \rightarrow \sum_{k=1}^\infty\int_{0}^{t} ((\mathbf{v}(s)\times \mathbf{f}_k+\mathbf{f}_k)  , \varphi)_{L^2(\mathbb{R}^d)}
 dW_k ,
\end{align}
in $L^2(0,T)$ in probability
 as $n\to \infty$.
Moreover, by \eqref{sj1 1} and \eqref{1.1} we have
\begin{align*}
&\sup_{n\geq 1}\mathrm{E}
\left (
\int_{0}^{T}
(( \mathbf{v}^n(t)\times \mathbf{f}_k+
\mathbf{f}_k), \varphi)^4dt
\right )
<\infty,
\end{align*}
which together with \eqref{ps 8}
  yields
  \begin{align}\label{ps 9}
\sum_{k=1}^\infty\int_{0}^{t}(( \mathbf{v}^n(s)\times
 \mathbf{f}_k+ \mathbf{f}_k) , \varphi)
 _{L^2(\mathbb{R}^d)} dW^n_k
 \rightarrow \sum_{k=1}^\infty\int_{0}^{t} ((\mathbf{v}(s)\times \mathbf{f}_k+\mathbf{f}_k)  , \varphi)_{L^2(\mathbb{R}^d)}
 dW_k ,
\end{align}
 in $L^2(\Omega\times [0,T])$.
 Therefore, for any
 $\psi\in L^\infty
 (\Omega \times [0,T])$ we get
 from \eqref{ps 9} that
$$
 \mathrm{E}\left (
 \int_0^T
 \psi (t)
 \left (
\sum_{k=1}^\infty\int_{0}^{t}(( \mathbf{v}^n(s)\times
 \mathbf{f}_k+ \mathbf{f}_k) , \varphi)
 _{L^2(\mathbb{R}^d)} dW^n_k
 \right ) dt \right )
 $$
  \begin{align}\label{ps 10}
 \rightarrow
 \mathrm{E}\left (
 \int_0^T
 \psi (t)
 \left (
 \sum_{k=1}^\infty\int_{0}^{t} ((\mathbf{v}(s)\times \mathbf{f}_k+\mathbf{f}_k)  , \varphi)_{L^2(\mathbb{R}^d)}
 dW_k
 \right )dt
 \right ).
\end{align}

Taking the limit of \eqref{ps 3} as
$n\to \infty$,
by \eqref{ps 4}-\eqref{ps 5},
\eqref{ps 7}-\eqref{ps 8-2}, \eqref{ps 10}
and
Lemma \ref{sj1} (iv), we obtain that
for any
 $\psi\in L^\infty
 (\Omega \times [0,T])$
 and $\varphi \in C_0^\infty (\mathbb{R}^d)$,
$$
   \mathrm{E}\left (
 \int_0^T
 \psi (t)
   (\mathbf{v}
    (t), \varphi )_{L^2
   ( \mathbb{R}^d ) } dt
   \right )
    =
 \mathrm{E}
 \left ( \int_{0}^{T}
 \psi (t)
 (\mathbf{u}_0, \varphi)_{L^2
   ( \mathbb{R}^d ) } dt
 \right )
 $$
$$
  +\mathrm{E}
  \left (
  \int_{0}^{T} \psi (t) \left(\int_{0}^{t}(\Delta\mathbf{v}(s)+\mathbf{v}(s)\times \Delta \mathbf{v}(s)-(1+|\mathbf{v}(s)|^{2})\mathbf{v}(s), \varphi)
  _{L^2
   ( \mathbb{R}^d ) }
  ds\right)dt
  \right ) $$
  $$+\frac{1}{2}
\mathrm{E}
\left (
  \int_{0}^{T} \psi (t)
  \left (
  \sum_{k=1}^\infty\int_{0}^{t}((\mathbf{v}(s)\times \mathbf{f}_k)\times \mathbf{f}_k, \varphi)
  _{L^2(\mathbb{R})}
  ds\right )dt \right )
  $$
 \be\label{ps 11}
 +
\mathrm{E}
\left (
  \int_{0}^{T} \psi (t)
  \left (
\sum_{k=1}^\infty
\int_{0}^{t}((\mathbf{v}(s)\times \mathbf{f}_k+\mathbf{f}_k)  , \varphi)
_{L^2
   ( \mathbb{R}^d ) }
dW_k \right ) dt
\right ) .
\ee
By  the arbitrariness of
  $\psi\in L^\infty
 (\Omega \times [0,T])$,
 $\varphi \in C_0^\infty (\mathbb{R}^d)$
 and   the density of $  C_0^\infty(\mathbb{R}^d)$ in
 $L^2(\mathbb{R}^d)$,
 we infer
 from \eqref{ps 11}
  that
   $\mathbf{v}$
   satisfies   \eqref{3.1},
   and thus   $(\Omega,\mathcal{F}, \{\mathcal{F}_t\}_{t\geq 0}, \mathrm{P}, W, \mathbf{v})$
   is a global martingale solution of the stochastic equation \eqref{equ1.1}.

   Note that
   \eqref{1.3} follows from
   the regularity of
   $\widetilde{\mathbf{v}}$ given by \eqref{sj1 2}.
   In addition, by \eqref{est_v 2} we have
   $$
   \textrm{E}
   \left (
   \int_0^T
   \| \mathbf{v}(s, \mathbf{u}_0)
   \times
   \Delta
   \mathbf{v}(s, \mathbf{u}_0)
   \|^{\frac 43}_{L^2
   (\mathbb{R}^d)}ds
   \right )^p
   <\infty,\quad \forall \  p\ge 1,
   $$
   which shows that
   \eqref{1.3a} is true
   for $r\in [1, {\frac 43}]$.
   Next, we prove
   \eqref{1.3a} is  also true
   for $r\in [{\frac 43}, 2)$.
   We just consider the case
   $d=2$ since $d=1$ is simpler.

    Given
 $r\in
  [{\frac 43}, 2)$ and $d=2$,
 let $a={\frac 2r} -1$
 and $q={\frac 2a}-2$.
 Since  $r\in [{\frac 43}, 2)$, we have
 $a\in (0, {\frac 12}]$
 and $q\ge 2$.
 Recall the following
 Gagliardo-Nirenberg  inequality
 for $d=2$:
 \be\label{est1a p1}
 \|\mathbf{v}\|_{L^\infty
 ({\mathbb{R}^2}) }
 \le c_1 \| \Delta
 \mathbf{v}\|_{L^2({\mathbb{R}^2})}^a
  \|
 \mathbf{v}\|_{L^q
 ({\mathbb{R}^2})}^{1-a},
 \ee
 where $c_1>0$ is a constant.
 Since $H^1(\mathbb{R}^2)
 \hookrightarrow L^q
 (\mathbb{R}^2)$ for any $
 q\ge 2$, by \eqref{est1a p1} we get
 \be\label{est1a p2}
 \|\mathbf{v}\|_{L^\infty
 ({\mathbb{R}^2}) }
 \le c_2 \| \Delta
 \mathbf{v}\|_{L^2({\mathbb{R}^2})}^a
  \|
 \mathbf{v}\|_{H^1
 ({\mathbb{R}^2})}^{1-a},
 \ee
 where $c_2>0$ is a constant.
 Then by \eqref{est1a p2}
we have
$$
\int_0^T
\|\mathbf{v} (s, \mathbf{u}_0)
\times
\Delta \mathbf{v} (s, \mathbf{u}_0)   \|^r_{ L^2(\mathbb{R}^2)}
ds
\le
\int_0^T
\| \mathbf{v} (s, \mathbf{u}_0) \|^r_{
L^\infty (\mathbb{R}^2)
}
\|
\Delta  \mathbf{v} (s, \mathbf{u}_0) \|^r_{
L^2(\mathbb{R}^2)
}
ds
 $$
 $$
 \le
 c_2^r
\int_0^T
\|\mathbf{v} (s, \mathbf{u}_0)
  \|^{(1-a)r}_{H^1
 (\mathbb{R}^2) }
\|
\Delta \mathbf{v} (s, \mathbf{u}_0)   \|^{(a+1)r} _{L^2( \mathbb{R}^2)}
ds
\le
 c_2^r
\sup_{s\in [0,T]}
\|
\mathbf{v} (s, \mathbf{u}_0)
\|^{(1-a)r}_{H^1
 ( \mathbb{R}^2)}\int_0^T
\|
\Delta  \mathbf{v} (s, \mathbf{u}_0) \|^2 _{L^2 ( \mathbb{R}^2)}
ds ,
 $$
 which along with \eqref{1.3}
 shows that for any $p\ge 1$,
 $$
 \textrm{E}
 \left (
\int_0^T
\|\mathbf{v} (s, \mathbf{u}_0)
\times
\Delta \mathbf{v} (s, \mathbf{u}_0)   \|^r_{ L^2(\mathbb{R}^2)}
ds
\right )^p
$$
 $$
 \le
 c_2^{pr}
 \textrm{E}
 \left (
\sup_{s\in [0,T]}
\|
\mathbf{v} (s, \mathbf{u}_0)
\|^{(1-a)pr}_{H^1
 ( \mathbb{R}^2)}
 \left ( \int_0^T
 \|
\Delta  \mathbf{v} (s, \mathbf{u}_0) \|^2 _{L^2 ( \mathbb{R}^2)}
ds
\right )^p
\right )
 $$
  $$
 \le
 c_2^{pr}
  \left ( \textrm{E}
  \left (
 \sup_{s\in [0,T]}
\|
\mathbf{v} (s, \mathbf{u}_0)
\|^{2(1-a)pr}_{H^1
 ( \mathbb{R}^2)}
 \right )
 \right )^{\frac 12}
  \left (
 \textrm{E} \left (
 \left ( \int_0^T
 \|
\Delta  \mathbf{v} (s, \mathbf{u}_0) \|^2 _{L^2 ( \mathbb{R}^2)}
ds
\right )^{2p} \right )
\right )^{\frac 12} <\infty,
 $$
and hence \eqref{1.3a}
is valid for all $p\ge 1$
and $r\in [1, 2)$.

  Finally
  the pathwise uniqueness of
   solutions of
\eqref{equ1.1}
follows from the continuity of solutions
with respect to initial data
in $L^2(\mathbb{R}^d)$
which is established in
   Lemma \ref{conin} below.
   This
  along with
the Yamada-Watanabe theorem
completes the proof of
 Theorem \ref{t3}.

Next,  we
  prove  the continuity of solutions of
 \eqref{equ1.1}
 in initial
data
in $L^2(\mathbb{R}^d)$ which will be
used to establish the Feller
property of the Markov semigroup associated
with \eqref{equ1.1}.

\begin{lemma}\label{conin}
Suppose \eqref{1.1} holds,
 $T>0$
and  $\mathbf{u}_{0},
\mathbf{u}_{0,n} \in H^1(\mathbb{R}^d)$
such  that
$\mathbf{u}_{0,n}
\to \mathbf{u}_{0}$
in $H^1(\mathbb{R}^d)$.
Then
the solutions
   $\mathbf{u} (\cdot, \mathbf{u}_{0,n})$
and
$\mathbf{u} (\cdot, \mathbf{u}_{0})$
of \eqref{equ1.1}  satisfy, for any
$p\ge 1$,
  $$
   \lim_{n\to \infty}
   \mathrm{E}
   \left (\sup_{t\in [0,T]}
   \|  \mathbf{u} ( t, \mathbf{u}_{0,n})-
   \mathbf{u} ( t, \mathbf{u}_{0})
     \|^p_{ L^2(\mathbb{R}^d) }
   \right )
=0.
  $$
 \end{lemma}

 \begin{proof}
 Let
 $\mathbf{u}_n (t)
 =\mathbf{u} (t, \mathbf{u}_{0,n})$,
 $\mathbf{u}  (t)
 =\mathbf{u} (t, \mathbf{u}_{0 })$,
 and
 $
 \mathbf{v}_n (t)
 = \mathbf{u} (t, \mathbf{u}_{0,n})
 -
 \mathbf{u} (t, \mathbf{u}_{0})$
 for $t\in [0,T]$.
 Then by \eqref{equ1.1} we  get
 $$
 d \mathbf{v}_n
 =
 \left (\Delta  \mathbf{v} _n
 + \mathbf{v}_n \times  \Delta  \mathbf{u}  _n
 +
 \mathbf{u}    \times
 \Delta \mathbf{v}_n
 -\mathbf{v} _n
 -\left (
 |\mathbf{u}  _n |^2
 \mathbf{u}  _n -
 |\mathbf{u}   |^2
 \mathbf{u}
 \right )
 \right ) dt
 $$
 \be\label{coin p0}
 +{\frac 12}
 \sum_{k=1}^\infty
  (\mathbf{v}_n  \times
 \mathbf{f}  _k)
 \times \mathbf{f} _k  dt
 +\sum_{k=1}^\infty
 (\mathbf{v} _n \times
 \mathbf{f}  _k ) dW_k.
\ee
 By It\^{o}'s formula we obtain
\be\label{conin p1}
 d \| \mathbf{v}_n\|^2_{L^2(\mathbb{R}^d)}
+ 2
 \left (
   \| \nabla \mathbf{v}_n\|^2_{L^2(\mathbb{R}^d)}
+
 \| \mathbf{v}_n\|^2_{L^2(\mathbb{R}^d)}
 \right ) dt
$$
$$=2   \left ( \mathbf{v}_n,
 (\mathbf{u}   \times
 \Delta \mathbf{v} _n  )
 \right )_{L^2(\mathbb{R}^d)} dt
 -2  \left (
 \mathbf{v} _n
 ,
  \left (
 |\mathbf{u}  _n |^2
 \mathbf{u}  _n -
 |\mathbf{u}   |^2
 \mathbf{u}
 \right )
 \right ) _{L^2(\mathbb{R}^d)}dt.
 \ee
 Note that
 $$ -2  \left ( \mathbf{v} _n
,
  \left (
 |\mathbf{u} _n |^2
 \mathbf{u}  _n -
 |\mathbf{u}  |^2
 \mathbf{u}
 \right )
 \right )_{L^2(\mathbb{R}^d)}
 = -2 \int_{\mathbb{R}^d} (\mathbf{u}  _n
 -
 \mathbf{u}  )
 \cdot
  \left (
 |\mathbf{u}  _n |^2
 \mathbf{u}  _n -
 |\mathbf{u}   |^2
 \mathbf{u}
 \right ) dx \le 0.
 $$
 Then by \eqref{conin p1} we get
 for all $t\in [0,T]$, $\mathrm{P}$-almost surely,
 $$
 \| \mathbf{v}_n (t)\|^2_{L^2(\mathbb{R}^d)}
 +  2 \int_0^t \| \nabla \mathbf{v
 } _n(s) \|^2_{L^2(\mathbb{R}^d)} ds
 +2\int_0^t
 \| \mathbf{v}_n (s) \|^2_{L^2(\mathbb{R}^d)}ds
 $$
 \be\label{conin p2}
 \le
 \| \mathbf{v} _n(0)\|^2_{L^2(\mathbb{R}^d)}
 + 2 \int_0^t  \left ( \mathbf{v}_n (s),
 (\mathbf{u}   (s)  \times
 \Delta \mathbf{v}_n  (s)   )
 \right ) _{L^2(\mathbb{R}^d)} ds.
  \ee
  For the last term on the right-hand side
  of  \eqref{conin p2} we will use the following inequality to control it
     \be\label{conin p2a}
  \|  \mathbf{v} \|_{L^4(\mathbb{R}^d)}
  \le c_1
  \| \nabla  \mathbf{v} \|^{\frac d4}
  _{L^2(\mathbb{R}^d)}\|    \mathbf{v} \|^{\frac {4-d} {4}}
  _{L^2(\mathbb{R}^d)},
  \quad \forall \ \mathbf{v} \in H^1(\mathbb{R}^d),
  \ee
   where
 $c_1>0$  is a    constant.
 Then by \eqref{conin p2a} and Young's inequality we get
  $$
  2 \int_0^t  \left ( \mathbf{v}_n (s),
 (\mathbf{u}    (s)  \times
 \Delta \mathbf{v}_n  (s)   )
 \right ) _{L^2(\mathbb{R}^d)} ds
 =
 -
 2 \int_0^t  \left ( \mathbf{v}_n (s),
 (\nabla \mathbf{u}    (s)  \times
 \nabla  \mathbf{v} _n (s)   )
 \right ) _{L^2(\mathbb{R}^d)} ds
 $$
 $$
 \le
 2 \int_0^t \| \mathbf{v}_n (s)\|_{L^4(\mathbb{R}^d)}
 \|\nabla \mathbf{u}    (s)\|_{L^4(\mathbb{R}^d)}
 \|
 \nabla  \mathbf{v}_n  (s)   \|
   _{L^2(\mathbb{R}^d)} ds
 $$
 $$
 \le 2
c_1^2  \int_0^t \| \mathbf{v} _n(s)\|^{\frac {4-d}{4}}
_{L^2(\mathbb{R}^d)}
 \| \mathbf{u}   (s)\|^{\frac {4-d}{4}}_{H^1(\mathbb{R}^d)}
 \| \mathbf{u}    (s)\|^{\frac d4}_{H^2(\mathbb{R}^d)}
 \|
 \nabla  \mathbf{v} _n (s)   \|^{1+\frac d4}
   _{L^2(\mathbb{R}^d)} ds
 $$
  \be\label{conin p3}
 \le \int_0^t \|
 \nabla  \mathbf{v}  _n (s)   \|^2
   _{L^2(\mathbb{R}^d)} ds
   +
c_2  \int_0^t
 \| \mathbf{u}   (s)\|^2_{H^1(\mathbb{R}^d)}
 \| \mathbf{u}   (s)\|^{\frac{2d}{4-d}}_{H^2(\mathbb{R}^d)}
 \| \mathbf{v}_n (s)\|^2
   _{L^2(\mathbb{R}^d)} ds,
   \ee
   where $c_2>0$  is an absolute constant.
 By \eqref{conin p2} and \eqref{conin p3} we get
 for all $t\in [0,T]$, $\mathrm{P}$-almost surely,
 $$
 \| \mathbf{v}_n (t)\|^2_{L^2(\mathbb{R}^d)}
 +    \int_0^t \| \nabla \mathbf{v
 } _n(s) \|^2_{L^2(\mathbb{R}^d)} ds
 +2\int_0^t
 \| \mathbf{v} _n(s) \|^2_{L^2(\mathbb{R}^d)}ds
 $$
 \be\label{conin p4}
 \le
 \| \mathbf{v}_n (0)\|^2_{L^2(\mathbb{R}^d)}
 +
c_2  \int_0^t
 \| \mathbf{u}   (s)\|^2_{H^1(\mathbb{R}^d)}
 \| \mathbf{u}    (s)\|^{\frac{2d}{4-d}}_{H^2(\mathbb{R}^d)}
 \| \mathbf{v}_n  (s)\|^2
   _{L^2(\mathbb{R}^d)} ds.
  \ee
  By \eqref{conin p4}
  and Gronwall's inequality we get
  for all $t\in [0,T]$, $\mathrm{P}$-almost surely,
 $$
 \| \mathbf{v} _n(t)\|^2_{L^2(\mathbb{R}^d)}
 \le
 \| \mathbf{v}_n (0)\|^2_{L^2(\mathbb{R}^d)}
 e^{c_2 \int_0^t
  \| \mathbf{u}    (s)\|^2_{H^1(\mathbb{R}^d)}
 \| \mathbf{u}    (s)\|^{\frac{2d}{4-d}}_{H^2(\mathbb{R}^d)}
 ds}
 $$
 \be\label{conin p5}
 \le
 \| \mathbf{v} _n (0)\|^2_{L^2(\mathbb{R}^d)}
e^{c_2
 \sup_{s\in [0,T]} \| \mathbf{u}    (s)\|^2_{H^1(\mathbb{R}^d)}\int_0^T
  \| \mathbf{u}   (s)\|^{\frac{2d}{4-d}}_{H^2(\mathbb{R}^d)}
 ds}
 = c_3
 \| \mathbf{v}_n  (0)\|^2_{L^2(\mathbb{R}^d)},
  \ee
  where
  $$
  c_3=
  e^{c_2
 \sup_{s\in [0,T]} \| \mathbf{u}    (s)\|^2_{H^1(\mathbb{R}^d)}\int_0^T
  \| \mathbf{u}    (s)\|^{\frac{2d}{4-d}}_{H^2(\mathbb{R}^d)}
 ds}.
  $$
  By \eqref{sj1 2} we know  that
   $c_3<\infty$,
  $\mathrm{P}$-almost surely.
  Then by \eqref{conin p5} we see that
  $\mathrm{P}$-almost surely,
  $$\lim_{n\to \infty} \| \mathbf{v}_n  \|_{C([0,T]; L^2(\mathbb{R}^d))}
  =0.
  $$

  On the other hand, by
  \eqref{sj1 2} we see that
  if
   $\mathbf{u}_{0,n} \to
  \mathbf{u}_{0 }$ in
  $H^1(\mathbb{R}^d)$,
  then the sequence
  $\{ \mathbf{v} _n\}_{n=1}^\infty$  is uniformly integrable
  in $L^p(\Omega, C([0,T]; L^2(\mathbb{R}^d)))$
  for any $p \ge 1$,
  and hence by Vitali's theorem, we get
  for all $p\ge 1$,
  $$
   \lim_{n\to \infty} \textrm {E}
   \left (
   \| \mathbf{v} _n \|^p_{C([0,T]; L^2(\mathbb{R}^d))}
   \right )
   =0  ,
  $$  which completes the proof.
   \end{proof}


%
%

\section{Existence of invariant measures}
In this section, we
prove  the existence  of invariant measures of the Markov
 semigroup
 associated with the solution
 operators
 of the  stochastic equation \eqref{equ1.1}.
 Given $t\ge 0$
 and $ \mathbf{u}_0 \in H^1(\mathbb{R}^d)$,
 denote by
  $
  \mathbf{P}_t(  \mathbf{u}_0, \cdot)$ the transition probability
  given by
$$\mathbf{P}_t(  \mathbf{u}_0, \mathcal{O})
=\mathrm{P}(\mathbf{u}(t,   \mathbf{u}_0 )\in \mathcal{O}),
\quad \forall \ \mathcal{O}\in \mathcal{B}(H^1(\mathbb{R}^d)) ,$$
  where $\mathbf{u}(t,   \mathbf{u}_0)$ is the   solution of   \eqref{equ1.1} starting from the initial point $  \mathbf{u}_0$.

For any bounded Borel function $\varphi\in \mathcal{B}_b(H^1(\mathbb{R}^d))$, define the Markov transition semigroup
\begin{align*}
(\mathbf{P}_t\varphi)(  \mathbf{u}_0)=\mathrm{E}[\varphi(\mathbf{u}(t,   \mathbf{u}_0))],~
   \mathbf{u}_0\in H^1(\mathbb{R}^d).
\end{align*}
A probability measure $\mu$ on $\mathcal{B}(H^1(\mathbb{R}^d))$ is an invariant measure  of
$\{\mathbf{P}_t\}_{t\ge 0}$
if
$$\int_{H^1(\mathbb{R}^d)}\mathbf{P}_t\varphi d\mu=\int_{H^1(\mathbb{R}^d)}\varphi d\mu,~~ {\rm for~ all}~t\geq 0, ~\varphi\in \mathcal{B}_b(H^1(\mathbb{R}^d)).$$

We will apply the weakly Feller method
\cite{Mas} to establish the existence
of invariant measures of
$\{\mathbf{P}_t\}_{t\ge 0}$ in $H^1(\mathbb{R}^d)$,
for which we need the  following convergence result
 as  given by   Lemma \ref{sj1}.

\begin{lemma}
\label{sj2}
Suppose \eqref{1.1} holds,
$\mathbf{u}_0, \mathbf{u}_{0,n} \in H^1(\mathbb{R}^d)$
such that
$\{\mathbf{u}_{0,n}\}_{n=1}^\infty$ weakly converges
to $\mathbf{u}_0$ in  $ H^1(\mathbb{R}^d)$.
Let
$\mathbf{u}^n(\cdot, \mathbf{u}_{0,n})$ be the unique strong solution of \eqref{equ1.1} with respect to initial data $\mathbf{u}_{0,n}$.
Then
there exist a   stochastic basis
 $(\widetilde{\Omega}, \widetilde{\mathcal{F}},
  \{\widetilde{\mathcal{F}}_t\}
  _{t\ge 0},  \widetilde{\mathrm{P}})$,
 random variables
 $(\widetilde{\mathbf{u}}, \widetilde{W} )$
 and
  $(\widetilde{\mathbf{u}}^n, \widetilde{W}^n)$,
  $n\in \mathbb{N}$,
   defined on $(\widetilde{\Omega}, \widetilde{\mathcal{F}},
    \{\widetilde{\mathcal{F}}_t\}
  _{t\ge 0},  \widetilde{\mathrm{P}})$ such that
   in the space $(C([0,T];L^2
   (\mathbb{R}^d)
   )\cap C([0,T];H_w^1(\mathbb{R}^d))
   \cap L_w^2(0,T;H^2(\mathbb{R}^d)))
    \times C([0,T];\mathbb{R}^\infty)$,
    up to a subsequence:

{\rm (i)}. The law  of $(\widetilde{\mathbf{u}}^n, \widetilde{W}^n)$ is
the same as    $( {\mathbf{u}}^n, W)$
for every $n\in \mathbb{N}$.

{\rm (ii)}.   $(\widetilde{\mathbf{u}}^n, \widetilde{W}^n)\rightarrow (\widetilde{\mathbf{u}}, \widetilde{W} )$,
 $\widetilde{\mathrm{P}}$-almost surely.

{\rm (iii)}.
$(\widetilde{\Omega}, \widetilde{\mathcal{F}},
  \{\widetilde{\mathcal{F}}_t\}
  _{t\ge 0},
  \widetilde{\mathbf{u}}, \widetilde{W},  \widetilde{\mathrm{P}})$
  is a martingale solution of \eqref{equ1.1} with initial law $\delta_{\mathbf{u}_0}$.
 \end{lemma}

 \begin{proof}
 Since $\{\mathbf{u}_{0,n}\}_{n=1}^\infty$ weakly converges
to $\mathbf{u}_0$ in  $ H^1(\mathbb{R}^d)$, we find that
the sequence $\{\mathbf{u}_{0,n}\}_{n=1}^\infty$
is bounded in $H^1(\mathbb{R}^d)$, and
thus all the uniform estimates
established in the previous section
are valid for the sequence
$\{\mathbf{u}^n(\cdot, \mathbf{u}_{0,n})\}_{n=1}^\infty$.
Then   the desired result follows
from the same argument as
Lemma \ref{sj1}. The details are omitted here.
 \end{proof}

Next, we show the existence
of invariant measures
of \eqref{equ1.1}.

\begin{theorem}\label{eim_wk}
 If \eqref{1.1} holds,  then
 the stochastic equation
 \eqref{equ1.1} has an invariant measure
 in $H^1(\mathbb{R}^d)$.
 \end{theorem}

  \begin{proof}
 By Lemma \ref{sj2}, one can verify
 that the Markov semigroup
 $\{\mathbf{P}_t\}_{t\ge 0}$
 is  $bw$-Feller  in $H^1(\mathbb{R}^d)$
 in the sense that
 if
 $\{\mathbf{u}_{0,n}\}_{n=1}^\infty$
 is convergent
to $\mathbf{u}_0$ in  $ H^1(\mathbb{R}^d)$ weakly,
and
$\varphi: H^1(\mathbb{R}^d)
\to \mathbb{R}$ is bounded and weakly
continuous,
then  for all $t\ge 0$, $\mathbf{P}_t \varphi: H^1(\mathbb{R}^d)
\to \mathbb{R}$ is bounded and weakly continuous function, that is,
 $$
 (\mathbf{P}_t \varphi)
 (\mathbf{u}_{0,n})
 \to
 (\mathbf{P}_t \varphi)
 (\mathbf{u}_{0 }),
 \quad \text{as } n \to \infty.
 $$
 Moreover, the tightness of time averaged measures defined by \eqref{avet1} on $(H^1(\mathbb{R}^d), bw)$ could be obtained directly by the estimate \eqref{est_v 4}.
  Then, we obtain the existence of at least one
 invariant measure for \eqref{equ1.1} following the Maslowski-Seidler
theorem, see, e.g., \cite{B1, Mas}  for more  details.
 \end{proof}

 In the next section, we study the
 limit of invariant measures
 of \eqref{equ1.1} as the noise
 intensity $\varepsilon$ approaches $\varepsilon_0\in [0,1]$.

\section{Limiting behavior of invariant measures}

In this section, we investigate
the uniform tightness of the collection
of all invariant measures
of \eqref{1.3*} for   $\varepsilon\in [0,1]$
in the one-dimensional case.
More precisely,
we will
show the union
$\bigcup\limits_{\varepsilon \in [0,1]}
\mathcal{I}_\varepsilon$
is tight in $H^1(\mathbb{R})$
where
$\mathcal{I}_\varepsilon$
is the set of all invariant
measures of \eqref{1.3*}
corresponding to $\varepsilon$.
To that end, we must establish
the uniform tail-ends estimates of solutions
in $H^1(\mathbb{R})$ in
order to overcome the difficulty
introduced by the non-compactness of
the  Sobolev embeddings
 in unbounded domains.
 Furthermore, by the idea of uniform tail-ends
 estimates, we will provide an alternative
 method to prove the existence
 of invariant measures
 of \eqref{1.3*}
 by the Feller property
 of the Markov semigroup
 with respect to the strong topology
 instead of the weak topology.

%
%
%
%



We first establish the following uniform estimates of solutions   which will be used  for the forthcoming tail-ends arguments.

\begin{lemma} \label{ues1}
If \eqref{1.1}  holds,
then for every
$\mathbf{u}_0 \in H^1(\mathbb{R})$,
the solution $\mathbf{u}
(t,  \mathbf{u}_0)$ of \eqref{equ1.1} satisfies, for all
$t\ge 0$,
 $$
\mathrm{E}
\left (
 \|
 \mathbf{u} (t,  \mathbf{u}_0) \|^2_{H^1(\mathbb{R})}
 \right )
 +
  \int_0^t e^{s-t}  \mathrm{E}
\left ( \|
 \mathbf{u} (s,  \mathbf{u}_0) \|^2_{H^2(\mathbb{R})}
 +
 \int_{\mathbb{R}}
 |\mathbf{u} (s,  \mathbf{u}_0)|^2
 |\nabla \mathbf{u} (s,  \mathbf{u}_0)|^2  dx
 \right )
 ds
 $$
\be\label{ues1 1}
 \le
 Ce^{-t} \|
 \mathbf{u} _0  \|^2_{H^1(\mathbb{R})}
 +C,
\ee
 and
 \be\label{ues1 2}
  \int_0^t
  \mathrm{E}
\left ( \|
 \mathbf{u} (s,  \mathbf{u}_0) \|^2_{H^2(\mathbb{R})}
 +
 \int_{\mathbb{R}}
 |\mathbf{u} (s,  \mathbf{u}_0)|^2
 |\nabla \mathbf{u} (s,  \mathbf{u}_0)|^2  dx
 \right )
 ds
 \le
 C  \|
 \mathbf{u} _0  \|^2_{H^1(\mathbb{R})}
 +Ct,
 \ee
  where $C>0$ is a constant  depending
only on $\{\mathbf{f}_k\}_{k=1}^\infty$.
\end{lemma}

\begin{proof}
By It\^{o}'s formula, we get from \eqref{equ1.1} that
for  all $t>0$, $\mathrm{P}$-almost surely,
$$
d\| \mathbf{u} (t)\|^2_{L^2(\mathbb{R})}
+2  \left ( \|\nabla  \mathbf{u} (t)\|^2_{L^2(\mathbb{R})}
+   \|  \mathbf{u} (t) \|^2_{L^2(\mathbb{R})}
+  \|  \mathbf{u} (t) \|^4_{L^4(\mathbb{R})}
\right ) dt
$$
$$
= \sum_{k=1}^\infty
\|\mathbf{f}_k\|^2_{L^2(\mathbb{R})} dt
+ 2\sum_{k=1}^\infty
( \mathbf{u}(t), \mathbf{f}_k) _{L^2(\mathbb{R})}  dW_k,
$$
and hence
\be\label{ues1 p1a}
{\frac d{dt}}
 \textrm{E}  \left (
\| \mathbf{u} (t)\|^2_{L^2(\mathbb{R})}
\right )
+2   \textrm{E}  \left (
\|\nabla  \mathbf{u} (t)\|^2_{L^2(\mathbb{R})}
 +
\|  \mathbf{u} (t) \|^2_{L^2(\mathbb{R})}
  +  \|  \mathbf{u} (t) \|^4_{L^4(\mathbb{R})}
\right )
= \sum_{k=1}^\infty
\|\mathbf{f}_k\|^2_{L^2(\mathbb{R})}.
 \ee
 By Gronwall's inequality, we
 obtain
for  all $t \ge 0$, $\mathrm{P}$-almost surely,
$$
\textrm{E}  \left (
\| \mathbf{u} (t)\|^2_{L^2(\mathbb{R})}
\right )
+
\int_0^t  e^{s-t} \textrm{E}  \left (
2 \|\nabla  \mathbf{u} (s)\|^2_{L^2(\mathbb{R})}
 +
\|  \mathbf{u} (s) \|^2_{L^2(\mathbb{R})}
  +  2 \|  \mathbf{u} (s) \|^4_{L^4(\mathbb{R})}
\right )ds
$$
\be\label{ues1 p1}
\leq e^{-t}
\| \mathbf{u}_0 \|^2_{L^2(\mathbb{R})}
+
 \sum_{k=1}^\infty
 \|\mathbf{f}_k\|^2_{L^2(\mathbb{R})}.
\ee

 Note that
$$((\nabla \mathbf{u}\times \mathbf{f}_k)\times \mathbf{f}_k, \nabla\mathbf{u})=-(\nabla \mathbf{u}\times \mathbf{f}_k,\nabla \mathbf{u}\times \mathbf{f}_k),$$
$$((\mathbf{u}\times \nabla\mathbf{f}_k)\times \mathbf{f}_k, \nabla\mathbf{u})=-(\nabla \mathbf{u}\times \mathbf{f}_k, \mathbf{u}\times \nabla\mathbf{f}_k),$$
hence, by It\^{o}'s formula, we
 also have
for  all $t>0$, $\mathrm{P}$-almost surely,
$$
d \| \nabla
 \mathbf{u} (t) \|^2_{L^2(\mathbb{R})}
 +2
 \left (
 \| \Delta
 \mathbf{u} (t) \|^2_{L^2(\mathbb{R})}
 +
 \| \nabla
 \mathbf{u} (t) \|^2_{L^2(\mathbb{R})}
 +\int_{\mathbb{R}}
 |\mathbf{u} (t)|^2
 |\nabla \mathbf{u} (t)|^2  dx
 \right ) dt
 $$
 $$
 +
  4
 \int_{\mathbb{R}}
 |\mathbf{u} (t)
  \nabla \mathbf{u} (t)|^2  dxdt
  =2\sum_{k=1}^\infty
  \left (
  \nabla \mathbf{u} (t),
  \mathbf{u} (t)
  \times \nabla \mathbf{f}_k
  +\nabla \mathbf{f} _k
  \right ) dW_k
 $$
$$
+\sum_{k=1}^\infty
\left ( \nabla \mathbf{u} (t),
\  ( \mathbf{u} (t) \times   \mathbf{f} _k )
\times \nabla \mathbf{f} _k
+  \mathbf{f} _k \times \nabla \mathbf{f}_k
\right ) _{L^2(\mathbb{R})} dt
$$
$$
+\sum_{k=1}^\infty
\left (
 \mathbf{u} (t) \times  \nabla  \mathbf{f} _k
+  \nabla \mathbf{f} _k ,\
 \nabla \mathbf{u} (t) \times
  \mathbf{f} _k
  +   \mathbf{u} (t) \times
  \nabla \mathbf{f} _k
  +  \nabla \mathbf{f}_k
\right ) _{L^2(\mathbb{R})} dt,
$$
and then
$$
{\frac d{dt}}
\textrm{E}
\left (
 \| \nabla
 \mathbf{u} (t) \|^2_{L^2(\mathbb{R})}
 \right )
 +2\textrm{E}
\left (
 \| \Delta
 \mathbf{u} (t) \|^2_{L^2(\mathbb{R})}
 +
 \| \nabla
 \mathbf{u} (t) \|^2_{L^2(\mathbb{R})}
 +
 \int_{\mathbb{R}}
 |\mathbf{u} (t)|^2
 |\nabla \mathbf{u} (t)|^2  dx
  \right )
 $$
$$
\le
\sum_{k=1}^\infty\textrm{E}
\left (
 \nabla \mathbf{u} (t),
\  ( \mathbf{u} (t) \times   \mathbf{f} _k )
\times \nabla \mathbf{f} _k
+  \mathbf{f} _k \times \nabla \mathbf{f}_k
\right ) _{L^2(\mathbb{R})}
$$
\be\label{ues1 p2}
+\sum_{k=1}^\infty \textrm{E}
\left (
 \mathbf{u} (t) \times  \nabla  \mathbf{f} _k
+  \nabla \mathbf{f} _k ,\
 \nabla \mathbf{u} (t) \times
  \mathbf{f} _k
  +   \mathbf{u} (t) \times
  \nabla \mathbf{f} _k
  +  \nabla \mathbf{f}_k
\right ) _{L^2(\mathbb{R})}  .
\ee

For the first term on the right-hand side of
\eqref{ues1 p2} we have
$$
\sum_{k=1}^\infty\textrm{E}
\left (
 \nabla \mathbf{u} (t),
\  ( \mathbf{u} (t) \times   \mathbf{f} _k )
\times \nabla \mathbf{f} _k
+  \mathbf{f} _k \times \nabla \mathbf{f}_k
\right ) _{L^2(\mathbb{R})}
$$
$$
\le
\sum_{k=1}^\infty\textrm{E}
\left (
\|\mathbf{f} _k \|
_{L^\infty(\mathbb{R})}
 \|  \nabla \mathbf{f} _k \|_{L^\infty(\mathbb{R})}
\| \nabla \mathbf{u} (t)\|_{L^2(\mathbb{R})}
\| \mathbf{u} (t) \|_{L^2(\mathbb{R})}
\right )
$$
$$
+
\sum_{k=1}^\infty
\textrm{E}
\left (
\|\mathbf{f} _k \|
_{L^\infty(\mathbb{R})}
 \|  \nabla \mathbf{f} _k \|_{L^2(\mathbb{R})}
\| \nabla \mathbf{u} (t)\|_{L^2(\mathbb{R})}
\right )
$$
$$
\le
{\frac 12}
\textrm{E}
\left (
\|  \nabla \mathbf{u} (t)  \|^2_{L^2(\mathbb{R})}
\right )
+ \left (
\sum_{k=1}^\infty
 \|\mathbf{f} _k \|^2
_{W^{1,\infty}(\mathbb{R})}
\right )^2
\textrm{E}
\left (
\|  \mathbf{u} (t)\|^2_{L^2(\mathbb{R})}
\right )
$$
\be\label{ues1 p3}
+  \left (
\sum_{k=1}^\infty
 \|\mathbf{f} _k \|
_{L^{ \infty}(\mathbb{R})}
\|\nabla \mathbf{f} _k \|
_{L^{ 2}(\mathbb{R})}
\right )^2.
\ee

 For the last term on the right-hand side of
\eqref{ues1 p2} we have
$$
\sum_{k=1}^\infty \textrm{E}
\left (
 \mathbf{u} (t) \times  \nabla  \mathbf{f} _k
+  \nabla \mathbf{f} _k ,\
 \nabla \mathbf{u} (t) \times
  \mathbf{f} _k
  +   \mathbf{u} (t) \times
  \nabla \mathbf{f} _k
  +  \nabla \mathbf{f}_k
\right ) _{L^2(\mathbb{R})}
$$
$$
\le
\sum_{k=1}^\infty \textrm{E}
\left (
\|\mathbf{f} _k\|_{L^\infty(\mathbb{R})}
 \|\nabla \mathbf{f} _k\|_{L^\infty (\mathbb{R})}
 \|\mathbf{u}(t) \|_{L^2(\mathbb{R})}
 \|\nabla \mathbf{u}(t)  \|_{L^2(\mathbb{R})}
\right )
$$
$$
+
\sum_{k=1}^\infty \textrm{E}
\left (
\|\mathbf{f} _k\|_{L^\infty(\mathbb{R})}
 \|\nabla \mathbf{f} _k\|_{L^2 (\mathbb{R})}
 \|\nabla \mathbf{u}(t) \|_{L^2(\mathbb{R})}
 +
 \|\nabla \mathbf{f} _k\|^2_{L^\infty (\mathbb{R})}
  \|\mathbf{u}(t) \|^2_{L^2(\mathbb{R})}
\right )
$$
$$
+ 2
\sum_{k=1}^\infty \textrm{E}
\left (
\|\nabla \mathbf{f} _k\|_{L^\infty(\mathbb{R})}
 \|\nabla \mathbf{f} _k\|_{L^2 (\mathbb{R})}
 \|  \mathbf{u}(t) \|_{L^2(\mathbb{R})}
 +
 \|\nabla \mathbf{f} _k\|^2_{L^2 (\mathbb{R})}
\right )
$$
$$
\le
{\frac 12}
\textrm{E}
\left (
\|  \nabla \mathbf{u} (t)  \|^2_{L^2(\mathbb{R})}
\right )
+ 2 \left (
\sum_{k=1}^\infty
 \|\mathbf{f} _k \|^2
_{W^{1,\infty}(\mathbb{R})}
\right )^2
\textrm{E}
\left (
\|  \mathbf{u} (t)\|^2_{L^2(\mathbb{R})}
\right )
$$
\be\label{ues1 p4}
+  2 \left (
\sum_{k=1}^\infty
 \|\mathbf{f} _k \|
_{L^{ \infty}(\mathbb{R})}
\|\nabla \mathbf{f} _k \|
_{L^{ 2}(\mathbb{R})}
\right )^2
+\textrm{E}
\left (
\|  \mathbf{u} (t)\|^2_{L^2(\mathbb{R})}
\right )
+  2
\sum_{k=1}^\infty
 \|\nabla \mathbf{f} _k \|^2
_{L^{ 2}(\mathbb{R})} .
\ee

It follows from
\eqref{ues1 p2}-\eqref{ues1 p4} that
$$
{\frac d{dt}}
\textrm{E}
\left (
 \| \nabla
 \mathbf{u} (t) \|^2_{L^2(\mathbb{R})}
 \right )
 + \textrm{E}
\left (
 2\| \Delta
 \mathbf{u} (t) \|^2_{L^2(\mathbb{R})}
 +
 \| \nabla
 \mathbf{u} (t) \|^2_{L^2(\mathbb{R})}
 +
2\int_{\mathbb{R}}
  |\mathbf{u} (t)|^2
 |\nabla \mathbf{u} (t)|^2  dx
 \right )
 $$
 \be\label{ues1 p4a}
\le c_1
\textrm{E}
\left (
  \|\mathbf{u} (t) \|^2
 _{L^2(\mathbb{R})}
 \right )
 + c_2,
\ee
 where $c_1$ and $c_2$
 are positive numbers depending
 only on $\{\mathbf{f}_k\}_{k=1}^\infty$.
By Gronwall's inequality we get for all
$t\ge 0$,
$$
\textrm{E}
\left (
 \| \nabla
 \mathbf{u} (t) \|^2_{L^2(\mathbb{R})}
 \right )
 +
  2\int_0^t e^{s-t}  \textrm{E}
\left ( \| \Delta
 \mathbf{u} (s) \|^2_{L^2(\mathbb{R})}
 +\int_{\mathbb{R}}
 |\mathbf{u} (s)|^2
 |\nabla \mathbf{u} (s)|^2  dx
 \right )
 ds
 $$
 \be\label{ues1 p5}
\le   e^{-t}
 \| \nabla
 \mathbf{u} _0  \|^2_{L^2(\mathbb{R})}
   +
 c_1 \int_0^t
 e^{s-t}
\textrm{E}
\left (
  \|\mathbf{u} (s) \|^2
 _{L^2(\mathbb{R})}
 \right ) ds
 + c_2.
\ee
By \eqref{ues1 p1} and \eqref{ues1 p5} we get
for all $t\ge 0$,
$$
\textrm{E}
\left (
 \| \nabla
 \mathbf{u} (t) \|^2_{L^2(\mathbb{R})}
 \right )
 +
 2 \int_0^t e^{s-t}  \textrm{E}
\left ( \| \Delta
 \mathbf{u} (s) \|^2_{L^2(\mathbb{R})}
 +
 \int_{\mathbb{R}}
 |\mathbf{u} (s)|^2
 |\nabla \mathbf{u} (s)|^2  dx
 \right )
 ds
 $$
 $$
 \le
 c_3 e^{-t} \|
 \mathbf{u} _0  \|^2_{H^1(\mathbb{R})}
 +c_3,
 $$
where $c_3>0$ depends
only on $\{\mathbf{f}_k\}_{k=1}^\infty$,
which along with \eqref{ues1 p1}
yields \eqref{ues1 1}.

Next, we prove \eqref{ues1 2}.
 By integrating  \eqref{ues1 p1a}
on $(0,t)$ we     obtain
\be\label{ues1 p6}
2  \int_0^t  \textrm{E}  \left (
\|\nabla  \mathbf{u} (s)\|^2_{L^2(\mathbb{R})}
 +
\|  \mathbf{u} (s) \|^2_{L^2(\mathbb{R})}
 \right ) ds
\le
 \textrm{E}  \left (
\| \mathbf{u} _0 \|^2_{L^2(\mathbb{R})}
\right )
 + t \sum_{k=1}^\infty
\|\mathbf{f}_k\|^2_{L^2(\mathbb{R})}.
 \ee
 On the other hand,
 by integrating \eqref{ues1 p4a}
 on $(0,t)$ we get
 $$
   2\int_0^t  \textrm{E}
\left (
 \| \Delta
 \mathbf{u} (s) \|^2_{L^2(\mathbb{R})}
 +
 \int_{\mathbb{R}}
  |\mathbf{u} (s)|^2
 |\nabla \mathbf{u} (s)|^2  dx
 \right )  ds
 $$
 \be\label{ues1 p7}
\le
\textrm{E}
\left (
 \| \nabla
 \mathbf{u} _0  \|^2_{L^2(\mathbb{R})}
 \right )
 +
c_1  \int_0^t
\textrm{E}
\left (
  \|\mathbf{u} (s) \|^2
 _{L^2(\mathbb{R})}
 \right ) ds
 + c_2 t.
\ee
It follows from \eqref{ues1 p6}-\eqref{ues1 p7}
that for all $t\ge 0$,
$$
   \int_0^t  \textrm{E}
\left (
  \|
 \mathbf{u} (s) \|^2_{H^2(\mathbb{R})}
 +
 \int_{\mathbb{R}}
  |\mathbf{u} (s)|^2
 |\nabla \mathbf{u} (s)|^2  dx
 \right )  ds
\le  c_4
 \|
 \mathbf{u} _0  \|^2_{H^1(\mathbb{R})}
  +
c_4 t   ,
$$
where $c_4>0$ depends
only on $\{\mathbf{f}_k\}_{k=1}^\infty$,
which implies \eqref{ues1 2} and thus
completes the proof.
 \end{proof}

 As an immediate consequence
 of Lemma \ref{ues1}, we
 obtain  the following uniform estimates
 which are useful when studying the
 tightness of the set of all invariant measures
 of \eqref{equ1.1}.

\begin{lemma} \label{ues1a}
If \eqref{1.1}  holds,
then for every
$R>0$, there exists
$T=T(R)>0$ such that
for all $t\ge T$,
 the solution $\mathbf{u}
(t,  \mathbf{u}_0)$ of \eqref{equ1.1}
with $\|\mathbf{u}_0 \|_{ H^1(\mathbb{R})}
  \le R$
 satisfies
 $$
\mathrm{E}
\left (
 \|
 \mathbf{u} (t,  \mathbf{u}_0) \|^2_{H^1(\mathbb{R})}
 \right )
 +
  \int_0^t e^{s-t}  \mathrm{E}
\left ( \|
 \mathbf{u} (s,  \mathbf{u}_0) \|^2_{H^2(\mathbb{R})}
 +
 \int_{\mathbb{R}}
 |\mathbf{u} (s,  \mathbf{u}_0)|^2
 |\nabla \mathbf{u} (s,  \mathbf{u}_0)|^2  dx
 \right )
 ds
 \le
   C,
 $$
 and
 $$
   {\frac 1t} \int_0^t  \mathrm{E}
\left (
  \|
 \mathbf{u} (s,\mathbf{u}_0) \|^2_{H^2(\mathbb{R})}
 +
 \int_{\mathbb{R}}
  |\mathbf{u} (s,\mathbf{u}_0)|^2
 |\nabla \mathbf{u} (s,\mathbf{u}_0)|^2  dx
 \right )  ds
\le
C   ,
$$
where $C>0$ is a constant  depending
only on $\{\mathbf{f}_k\}_{k=1}^\infty$.
\end{lemma}

We now derive the uniform
estimates on the tails of solutions
in $L^2(\mathbb{R})$.

\begin{lemma}\label{tail_2}
If \eqref{1.1} holds,
then for every   $\varepsilon>0$
  and
 $\mathbf{u}_0
 \in H^1
 (\mathbb{R}  )$,
 there exists
 $M=M(\varepsilon,
    \mathbf{u}_0) \ge 1$
 such that
 for all $m\ge M$ and $t\ge 0$,
 the solution
 ${\mathbf{ u}} (t,
\mathbf{u}_0 )$
of \eqref{equ1.1} satisfies
$$
 \mathrm{E}
\left (
 \int_{|x|>m}
 |
 {\mathbf{ u}} (t,
\mathbf{u}_0 )(x) |^2 dx
\right )
  <\varepsilon.
$$
\end{lemma}

\begin{proof}
As before,  let    $\theta$  be
the cut-off function given by \eqref{cutoff},
$\phi (x) =1-\theta (x)$, and
   $\phi_m(x)=\phi(\frac{x}{m})$
   for
$x\in \mathbb{R} $ and
     $m\in \mathbb{N}$.

    By  \eqref{equ1.1}  and
    It\^{o}'s formula,  we  get
     for all $t\ge 0$, $\mathrm{P}$-almost surely,
  \begin{align}\label{tail_2 p1}
& \|\phi_m\mathbf{u}
(t)
\|_{L^2
(\mathbb{R} ) }^2
-
\|\phi_m\mathbf{u}_0
\|_{L^2
(\mathbb{R} ) }^2\nonumber\\
&=2\int_0^t ( \Delta \mathbf{u}
(s), \phi^2_m\mathbf{u} (s ) )_{L^2
(\mathbb{R} ) } ds
+2
\int_0^t
( \mathbf{u}(s)\times \Delta \mathbf{u}
(s), \phi^2_m\mathbf{u}(s) ) _{L^2
(\mathbb{R} ) }ds\nonumber\\
&\quad-2
\int_0^t
((1+|\mathbf{u}(s)|^{2})\mathbf{u}(s),
\phi^2 _m\mathbf{u}(s) )_{L^2
(\mathbb{R} ) }ds
+2 \sum_{k=1}^\infty
\int_0^t
(\phi_m \mathbf{f}_k  , \phi_m\mathbf{u}(s) )_{L^2
(\mathbb{R} ) }
dW_k
\nonumber\\
&\quad+\sum_{k=1}^\infty
\int_0^t
\|\phi_m (\mathbf{u}(s) \times \mathbf{f}_k+\mathbf{f}_k)\| _{L^2
(\mathbb{R} ) }^2ds
+
\sum_{k=1}^\infty
\int_0^t
(
(\mathbf{u}(s) \times \mathbf{f}_k)\times \mathbf{f}_k, \phi^2_m\mathbf{u}(s))_{L^2
(\mathbb{R} )  }ds.
\end{align}
  By \eqref{tail_2 p1} we get
 for all $t> 0$, $\mathrm{P}$-almost surely,
   \begin{align}\label{tail_2 p2}
&
{\frac d{dt}}
\mathrm{E}
\left (
\|\phi_m\mathbf{u}
(t)
\|_{L^2
(\mathbb{R} ) }^2
\right )
 - 2 \mathrm{E}
\left (
   \Delta \mathbf{u}
(t), \phi^2_m\mathbf{u} (t )
\right )_{L^2
(\mathbb{R} ) }
+
 2  \mathrm{E}
\left (
(1+|\mathbf{u}(t)|^{2})\mathbf{u}(t),
\phi^2 _m\mathbf{u}(t) \right )_{L^2
(\mathbb{R} ) }
 \nonumber\\
&\quad=
\sum_{k=1}^\infty
 \mathrm{E}
\left (
\|\phi_m (\mathbf{u}(t) \times \mathbf{f}_k+\mathbf{f}_k)\| _{L^2
(\mathbb{R} ) }^2
\right )
+
\sum_{k=1}^\infty
 \mathrm{E}
\left (
(\mathbf{u}(t) \times \mathbf{f}_k)\times \mathbf{f}_k, \phi^2_m\mathbf{u}(t) \right )_{L^2
(\mathbb{R} )  } .
\end{align}
 Similar to \eqref{tail_1 p2} we have
 \begin{align}\label{tail_2 p2a}
 2 \mathrm{E}
\left (
   \Delta \mathbf{u}
(t), \phi^2_m\mathbf{u} (t )
\right )_{L^2
(\mathbb{R} ) }
\le
    \frac{c_1}{m}
  \mathrm{E}
\left (
\|\mathbf{u} (t ) \|_{H^1
 (\mathbb{R} ) }^2
 \right ),
\end{align}
where  $c_1>0$ is
a constant independent of $m$.
 For the last two terms on the right-hand side
 of \eqref{tail_2 p2},   we have
 $$\sum_{k=1}^\infty
 \mathrm{E}
\left (
\|\phi_m (\mathbf{u}(t) \times \mathbf{f}_k+\mathbf{f}_k)\| _{L^2
(\mathbb{R} ) }^2
\right )=\sum_{k=1}^\infty
 \mathrm{E}
\left (
\|\phi_m (\mathbf{u}(t) \times \mathbf{f}_k)\| _{L^2
(\mathbb{R} ) }^2+\|\phi_m\mathbf{f}_k\| _{L^2
(\mathbb{R} ) }^2
\right ),$$
 and
 $$\sum_{k=1}^\infty
 \mathrm{E}
\left(\left (
(\mathbf{u}(t) \times \mathbf{f}_k)\times \mathbf{f}_k, \phi^2_m\mathbf{u}(t) \right )_{L^2
(\mathbb{R} )  }\right)
   =
  -\sum_{k=1}^\infty
 \mathrm{E}
\left (
\|\phi_m (\mathbf{u}(t)\times\mathbf{f}_k)\| _{L^2
(\mathbb{R} ) }^2\right).$$
 Taking sum we have
\begin{align}\label{tail_2 p4}
 &\sum_{k=1}^\infty
 \mathrm{E}
\left (
\|\phi_m (\mathbf{u}(t) \times \mathbf{f}_k+\mathbf{f}_k)\| _{L^2
(\mathbb{R} ) }^2
\right )
+
\sum_{k=1}^\infty
 \mathrm{E}
\left(\left (
(\mathbf{u}(t) \times \mathbf{f}_k)\times \mathbf{f}_k, \phi^2_m\mathbf{u}(t) \right )_{L^2
(\mathbb{R} )  }\right)
   =
  \sum_{k=1}^\infty
\|\phi_m \mathbf{f}_k\|
_{L^2(\mathbb{R} ) }^2 .
\end{align}
By
  \eqref{tail_2 p2}-\eqref{tail_2 p4}
  we obtain
  for all $t> 0$ and $ m \in \mathbb{N}$,
\be\label{tail_2 p5}
{\frac d{dt}}
\mathrm{E}
\left (
\|\phi_m\mathbf{u}
(t)
\|_{L^2
(\mathbb{R} ) }^2
\right )
+2
\mathrm{E}
\left (
\|\phi_m\mathbf{u}
(t)
\|_{L^2
(\mathbb{R} ) }^2
\right )
 \leq
  \frac{c_1}{m}
  \mathrm{E}
\left (
\|\mathbf{u} (t ) \|_{H^1
 (\mathbb{R} ) }^2
 \right )
 +
 \sum_{k=1}^\infty
\|\phi_m \mathbf{f}_k\|
_{L^2(\mathbb{R} ) }^2.
\ee
Solving \eqref{tail_2 p5} we get
 for all $t\ge 0$
and $ m \in \mathbb{N}$,
$$
\mathrm{E}
\left (
\|\phi_m\mathbf{u}
(t)
\|_{L^2
(\mathbb{R} ) }^2
\right )
+\int_0^t
e^{s-t}
\mathrm{E}
\left (
\|\phi_m\mathbf{u}
(s)
\|_{L^2
(\mathbb{R} ) }^2
\right ) ds
$$
\be\label{tail2p6}\le
e^{-t}
\mathrm{E}
\left (
\|\phi_m\mathbf{u}_0
\|_{L^2
(\mathbb{R} ) }^2
\right )
  +
  \frac{c_1}{m}\int_0^t
  e^{s-t}
  \mathrm{E}
\left (
\|\mathbf{u} (s) \|_{H^1
 (\mathbb{R} ) }^2
 \right ) ds
 +
 \sum_{k=1}^\infty
\|\phi_m \mathbf{f}_k\|
_{L^2(\mathbb{R} ) }^2.
\ee
By \eqref{tail2p6} and
  Lemma \ref{ues1} we obtain
for all $t\ge 0$
and $ m \in \mathbb{N}$,
$$
\mathrm{E}
\left (
\|\phi_m\mathbf{u}
(t)
\|_{L^2
(\mathbb{R} ) }^2
\right )
+\int_0^t
e^{s-t}
\mathrm{E}
\left (
\|\phi_m\mathbf{u}
(s)
\|_{L^2
(\mathbb{R} ) }^2
\right ) ds
$$
$$
\le
\|\phi_m\mathbf{u}_0
\|_{L^2
(\mathbb{R} ) }^2
  +
  \frac{c_2}{m}
  (1+ \| \mathbf{u}_0
\|_{H^1
(\mathbb{R} ) }^2)
  +
 \sum_{k=1}^\infty
\|\phi_m \mathbf{f}_k\|
_{L^2(\mathbb{R} ) }^2,
$$
where $c_2>0$ depends only on
$\{\mathbf{f}_k\}_{k=1}^\infty$.
Since
$\sum_{k=1}^\infty
\| \mathbf{f}_k\|
_{L^2(\mathbb{R} ) }^2<\infty$, we
see  that
 $$
\|\phi_m \mathbf{u}_0
 \|_{L^2(\mathbb{R} )
 }^2
 +
  \sum_{k=1}^\infty
  \|\phi_m \mathbf{f}_k\|
_{L^2(\mathbb{R})  }^2
 \le
\int_{|x|>{\frac 12 m}}
 | \mathbf{u}_0 (x)|^2 dx
   +   \int_{|x|>{\frac 12 m}}
    \sum_{k=1}^\infty |\mathbf{f}_k (x) |^2 dx
    \to 0,
  $$
  as $m\to \infty$,
which together with \eqref{tail2p6}
indicates  that  for every $\varepsilon>0$,
there exists $M=M(\varepsilon,   \mathbf{u}_0)\ge 1$,
such that
for all $m\ge M$  and $t \ge 0$,
$$
\mathrm{E} \left (
\int_{|x|>  {\frac 34 m} }
 | \mathbf{u}  (t, \mathbf{u} _0) (x)|^2 dx
 \right )
 +\int_0^t
e^{s-t}
\mathrm{E}
\left (
\|\phi_m\mathbf{u}
(s)
\|_{L^2
(\mathbb{R} ) }^2
\right ) ds
$$
\be\label{tail_2 p7}
 \le
 \mathrm{E}
 \left (
 \|\phi_m \mathbf{u}  (t,  \mathbf{u}_0 )
 \|_{L^2(\mathbb{R} )
 }^2
 \right )
 +\int_0^t
e^{s-t}
\mathrm{E}
\left (
\|\phi_m\mathbf{u}
(s)
\|_{L^2
(\mathbb{R} ) }^2
\right ) ds
  <\varepsilon ,
\ee
which completes the proof.
 \end{proof}

 Following the proof of Lemma \ref{tail_2},
 by using Lemma \ref{ues1a}
 instead of Lemma \ref{ues1}, we obtain the
 following uniform tail-ends estimates.

\begin{lemma}\label{tail_2a}
If \eqref{1.1} holds,
then for every   $\varepsilon>0$
and $R>0$,
there exist $T=T(\varepsilon, R)>0$
and
  $M=M(\varepsilon) \ge 1$
 such that
 for all
 $t\ge T$ and
 $m\ge M$,
 the solution
 ${\mathbf{ u}} (t,
\mathbf{u}_0 )$
of \eqref{equ1.1}
with
$\|\mathbf{u}_0\|_{
 H^1
 (\mathbb{R} ) } \le R$,
 satisfies
 $$
 \mathrm{E}
\left (
 \int_{|x|>m}
 |
 {\mathbf{ u}} (t,
\mathbf{u}_0 )(x) |^2 dx
\right )
  <\varepsilon.
$$
\end{lemma}

 Next, we improve
 Lemma \ref{tail_2} and
   derive the uniform
estimates on the tails of solutions
in $H^1(\mathbb{R})$.

\begin{lemma}\label{tail_3}
If \eqref{1.1} holds,
then for every   $\varepsilon>0$,
  and
 $\mathbf{u}_0
 \in H^1
 (\mathbb{R} )$,
 there exists
 $M=M(\varepsilon,
    \mathbf{u}_0) \ge 1$
 such that
 for all $m\ge M$,
 the solution
 ${\mathbf{ u}} (t,
\mathbf{u}_0 )$
of \eqref{equ1.1} satisfies
$$
 \sup_{t\ge 0}
 \mathrm{E}
\left (
 \int_{|x|>m}
 \left ( |
 {\mathbf{ u}} (t,
\mathbf{u}_0 )(x) |^2
+
|\nabla
 {\mathbf{ u}} (t,
\mathbf{u}_0 )(x) |^2
\right )
dx
\right )  <\varepsilon.
$$
\end{lemma}

 \begin{proof}
Let
   $\phi_m  $ be
   the same cut-off function as in the proof of Lemma
   \ref{tail_2}.
    By \eqref{equ1.1} and  It\^{o}'s formula,
     we get
\begin{align*}
&d\| \phi_m \nabla\mathbf{u}(t)\|
_{L^2
(\mathbb{R}) }^2\nonumber\\&= -2
\left (    \Delta \mathbf{u}(t),  \text{div}
\left (\phi_m^2 \nabla\mathbf{u}(t)
\right )  \right )_{L^2
(\mathbb{R}) } dt - 2
\left (  \mathbf{u}(t)\times \Delta \mathbf{u}(t),
\text{div}
\left (   \phi_m^2
 \nabla\mathbf{u}(t) \right ) \right ) _{L^2
(\mathbb{R}) } dt \nonumber\\
 &-2(  \phi_m \nabla(1+|\mathbf{u}(t)|^{2})\mathbf{u}(t),
   \phi_m \nabla\mathbf{u}(t))_{L^2
(\mathbb{R}) } dt
 +\sum_{k= 1}^\infty (
 \phi_m
\nabla((\mathbf{u}(t)\times \mathbf{f}_k)\times \mathbf{f}_k),   \phi_m \nabla\mathbf{u}(t))_{L^2
(\mathbb{R}) }dt\nonumber\\
&+2\sum_{k= 1}^\infty
(  \phi_m \nabla(\mathbf{u}(t)\times \mathbf{f}_k+\mathbf{f}_k)  ,   \phi_m \nabla\mathbf{u}(t))_{L^2
(\mathbb{R}) } dW_k
+\sum_{k= 1}^\infty \|  \phi_m \nabla(\mathbf{u}(t)\times \mathbf{f}_k+\mathbf{f}_k)\|_{ _{L^2
(\mathbb{R}) }}^2dt,
 \end{align*}
and hence
for all $t>0$, $\mathrm{P}$-almost surely,
\begin{align}  \label{tail_3 p1}
&{\frac d{dt}}
\textrm{E} \left (
\| \phi_m\nabla\mathbf{u} (t)\|
_{L^2
(\mathbb{R}) }^2
\right )\nonumber\\
 &= -2 \textrm{E} \left (
\left (    \Delta \mathbf{u}(t),  \text{div}
\left ( \phi_m^2 \nabla\mathbf{u}(t)
\right )  \right )_{L^2
(\mathbb{R}) } \right )
- 2\textrm{E} \left (
\left (  \mathbf{u}(t)\times \Delta \mathbf{u}(t),
\text{div}
\left (  \phi_m^2
 \nabla\mathbf{u}(t) \right ) \right ) _{L^2
(\mathbb{R}) } \right )
 \nonumber\\
 &-2\textrm{E} \left ( (  \phi_m \nabla(1+|\mathbf{u}(t)|^{2})\mathbf{u}(t),   \phi_m \nabla\mathbf{u}(t))_{L^2
(\mathbb{R}) } \right )
+\sum_{k= 1}^\infty
\textrm{E} \left (
 \|   \phi_m \nabla(\mathbf{u}(t)\times \mathbf{f}_k+\mathbf{f}_k)\|_{  L^2
(\mathbb{R})  }^2 \right) \nonumber\\
&+\sum_{k= 1}^\infty
 \textrm{E} \left (
 (
 \phi_m
\nabla((\mathbf{u}(t)\times \mathbf{f}_k)\times \mathbf{f}_k),   \phi_m \nabla\mathbf{u}(t))_{L^2
(\mathbb{R}) }
\right ).
  \end{align}

 For the first  term on the right-hand side
 of \eqref{tail_3 p1},
by H\"{o}lder's inequality  we  have
\be\label{tail_3 p2}
-2 \textrm{E} \left (
\left (    \Delta \mathbf{u}(t),  \text{div}
\left ( \phi_m^2 \nabla\mathbf{u}(t)
\right )  \right )_{L^2
(\mathbb{R}) } \right )
\le   -2
\textrm{E} \left (
\| \phi_m \Delta \mathbf{u} (t)\|_{L^2
(\mathbb{R}) }^2
\right ) +
   \frac{c_1}{m}\|\mathbf{u} (t)\|_{H^2(\mathbb{R}) }^2,
\ee
where   $c_1>0 $ is independent of $m$.

For the second   term on the right-hand side
of \eqref{tail_3 p1}
  we have
 \begin{align}\label{tail_3 p3}
 &- 2\textrm{E} \left (
\left (  \mathbf{u}(t)\times \Delta \mathbf{u}(t),
\text{div}
\left (  \phi_m^2
 \nabla\mathbf{u}(t) \right ) \right ) _{L^2
(\mathbb{R}) } \right )\nonumber \\
&= - 2\textrm{E} \left (
\left (  \mathbf{u}(t)\times \Delta \mathbf{u}(t),
  \phi_m^2
 \Delta \mathbf{u}(t)   \right ) _{L^2
(\mathbb{R}) } \right )
  - 2\textrm{E} \left (
\left (  \mathbf{u}(t)\times \Delta \mathbf{u}(t),
 (\nabla  \phi_m^2 )
 \nabla \mathbf{u}(t)   \right ) _{L^2
(\mathbb{R}) } \right )\nonumber \\
&=
  - 2\textrm{E} \left (
\left (  \mathbf{u}(t)\times \Delta \mathbf{u}(t),
 (\nabla  \phi_m^2 )
 \nabla \mathbf{u}(t)   \right ) _{L^2
(\mathbb{R}) } \right )\nonumber \\
&
\le {\frac {c_2}m} \textrm{E} \left (
   \|\Delta \mathbf{u} (t) \|_{L^2
   (\mathbb{R}) }
   \left (
   \int_{\mathbb{R}}
 |\mathbf{u} (t)|^2
 |\nabla \mathbf{u} (t)|^2  dx\right )^{\frac 12}
 \right )  \nonumber\\
 &
 \le
  {\frac {c_2}{2m}} \textrm{E} \left (
   \|\Delta \mathbf{u} (t) \|^2_{L^2
   (\mathbb{R}) }  \right )
   +
     {\frac {c_2}{2m}}\textrm{E} \left (
  \int_{\mathbb{R}}
 |\mathbf{u} (t)|^2
 |\nabla \mathbf{u} (t)|^2  dx  \right ) ,
\end{align}
where $c_2>0$ is independent of $m$.

For the third  term on the right-hand side
of \eqref{tail_3 p1}
  we have
$$
-2
\textrm{E} \left (
(  \phi_m \nabla(1+|\mathbf{u}(t)|^{2})\mathbf{u}(t),
  \phi_m \nabla\mathbf{u}(t))_{L^2
(\mathbb{R}) }
\right )
$$
$$
=
 -2\textrm{E} \left (
 \|  \phi_m   \nabla\mathbf{u} (t) \|_{L^2
 (\mathbb{R}) }^2
+\int_{ \mathbb{R}}
  \phi_m^2 |\mathbf{u} (t) |^2 |\nabla
 \mathbf{u} (t)
  |^2 dx  \right )
  -4 \textrm{E} \left (
  \int_{ \mathbb{R}}
  \phi_m^2 |\mathbf{u} (t)  \nabla
 \mathbf{u} (t)
  |^2 dx
  \right )
  $$
 \be \label{tail_3 p4}
  \le
 -2\textrm{E} \left (
 \|  \phi_m   \nabla\mathbf{u} (t) \|_{L^2
 (\mathbb{R}) }^2
 \right ).
 \ee

 For the last  two terms
 on the right-hand side of \eqref{tail_3 p1} we have
 $$\sum_{k= 1}^\infty
 \textrm{E} \left (
 (
 \phi_m
\nabla((\mathbf{u}(t)\times \mathbf{f}_k)\times \mathbf{f}_k),   \phi_m \nabla\mathbf{u}(t))_{L^2
(\mathbb{R}) }
\right )
  +\sum_{k= 1}^\infty
\textrm{E} \left (
 \|   \phi_m \nabla(\mathbf{u}(t)\times \mathbf{f}_k+\mathbf{f}_k)\|_{  L^2
(\mathbb{R})  }^2 \right )
$$
$$
= \sum_{k= 1}^\infty
 \textrm{E} \left (
 \left (
 \phi_m^2
\nabla  \mathbf{u}(t),
\
(\mathbf{u}(t) \times \mathbf{f}_k)
\times \nabla \mathbf{f}_k
+ \mathbf{f}_k \times \nabla \mathbf{f}_k
\right ) _{L^2
(\mathbb{R}) }
\right )
$$
\be\label{tail_3 p5}
  +\sum_{k= 1}^\infty
\textrm{E} \left (
\left ( \phi_m^2
(
\mathbf{u}(t) \times \nabla  \mathbf{f}_k
+\nabla \mathbf{f}_k), \
\nabla \mathbf{u}(t) \times    \mathbf{f}_k
+
\mathbf{u}(t) \times \nabla  \mathbf{f}_k
+\nabla  \mathbf{f}_k \right )
  _{  L^2
(\mathbb{R})  }
\right ).
\ee
  For the first term on the right-hand side of
\eqref{tail_3 p5} we have
$$
\sum_{k=1}^\infty\textrm{E}
\left (
 \phi_m^2
 \nabla \mathbf{u} (t)  ,
\  ( \mathbf{u} (t) \times   \mathbf{f} _k )
\times \nabla \mathbf{f} _k
+  \mathbf{f} _k \times \nabla \mathbf{f}_k
\right ) _{L^2(\mathbb{R})}
$$
$$
\le
\sum_{k=1}^\infty\textrm{E}
\left (
\|\mathbf{f} _k \|
_{L^\infty(\mathbb{R})}
 \|  \nabla \mathbf{f} _k \|_{L^\infty(\mathbb{R})}
\|  \phi_m \nabla \mathbf{u} (t)\|_{L^2(\mathbb{R})}
\| \phi_m  \mathbf{u} (t) \|_{L^2(\mathbb{R})}
\right )
$$
$$
+
\sum_{k=1}^\infty
\textrm{E}
\left (
\|\nabla \mathbf{f} _k \|
_{L^\infty(\mathbb{R})}
 \|  \phi_m  \mathbf{f} _k \|_{L^2(\mathbb{R})}
\|   \phi_m \nabla \mathbf{u} (t)\|_{L^2(\mathbb{R})}
\right )
$$
$$
\le
{\frac 12}
\textrm{E}
\left (
\|  \phi_m\nabla \mathbf{u} (t)  \|^2_{L^2(\mathbb{R})}
\right )
+ \left (
\sum_{k=1}^\infty
 \|\mathbf{f} _k \|^2
_{W^{1,\infty}(\mathbb{R})}
\right )^2
\textrm{E}
\left (
\|  \phi_m \mathbf{u} (t)\|^2_{L^2(\mathbb{R})}
\right )
$$
 \be\label{tail_3 p5a}
+  \left (
\sum_{k=1}^\infty
 \| \phi_m \mathbf{f} _k \|
_{L^{2}(\mathbb{R})}
\|\nabla \mathbf{f} _k \|
_{L^{ \infty }(\mathbb{R})}
\right )^2.
\ee
 For the  second term on the right-hand side of
\eqref{tail_3 p5} we have
$$
\sum_{k=1}^\infty \textrm{E}
\left (  \phi_m^2
(
 \mathbf{u} (t) \times  \nabla  \mathbf{f} _k
+  \nabla \mathbf{f} _k) ,\
 \nabla \mathbf{u} (t) \times
  \mathbf{f} _k
  +   \mathbf{u} (t) \times
  \nabla \mathbf{f} _k
  +  \nabla \mathbf{f}_k
\right ) _{L^2(\mathbb{R})}
$$
$$
\le
\sum_{k=1}^\infty \textrm{E}
\left (
\|\mathbf{f} _k\|_{L^\infty(\mathbb{R})}
 \|\nabla \mathbf{f} _k\|_{L^\infty (\mathbb{R})}
 \| \phi_m \mathbf{u}(t) \|_{L^2(\mathbb{R})}
 \| \phi_m \nabla \mathbf{u}(t)  \|_{L^2(\mathbb{R})}
\right )
$$
$$
+
\sum_{k=1}^\infty \textrm{E}
\left (
\|\nabla \mathbf{f} _k\|_{L^\infty(\mathbb{R})}
 \| \phi_m   \mathbf{f} _k\|_{L^2 (\mathbb{R})}
 \| \phi_m \nabla \mathbf{u}(t) \|_{L^2(\mathbb{R})}
 +
 \|\nabla \mathbf{f} _k\|^2_{L^\infty (\mathbb{R})}
  \| \phi_m \mathbf{u}(t) \|^2_{L^2(\mathbb{R})}
\right )
$$
$$
+ 2
\sum_{k=1}^\infty \textrm{E}
\left (
\|\nabla \mathbf{f} _k\|_{L^\infty(\mathbb{R})}
 \| \phi_m \nabla \mathbf{f} _k\|_{L^2 (\mathbb{R})}
 \|  \phi_m \mathbf{u}(t) \|_{L^2(\mathbb{R})}
 +
 \| \phi_m \nabla \mathbf{f} _k\|^2_{L^2 (\mathbb{R})}
\right )
$$
$$
\le
{\frac 12}
\textrm{E}
\left (
\|  \phi_m \nabla \mathbf{u} (t)  \|^2_{L^2(\mathbb{R})}
\right )
+ 2 \left (
\sum_{k=1}^\infty
 \|\mathbf{f} _k \|^2
_{W^{1,\infty}(\mathbb{R})}
\right )^2
\textrm{E}
\left (
\|  \phi_m  \mathbf{u} (t)\|^2_{L^2(\mathbb{R})}
\right )
$$
\be\label{tail_3 p5b}
+  2 \left (
\sum_{k=1}^\infty
 \|\nabla \mathbf{f} _k \|
_{L^{ \infty}(\mathbb{R})}
\| \phi_m   \mathbf{f} _k \|
_{L^{ 2}(\mathbb{R})}
\right )^2
+\textrm{E}
\left (
\|  \phi_m  \mathbf{u} (t)\|^2_{L^2(\mathbb{R})}
\right )
+  2
\sum_{k=1}^\infty
 \| \phi_m \nabla \mathbf{f} _k \|^2
_{L^{ 2}(\mathbb{R})} .
\ee
By \eqref{tail_3 p5}-\eqref{tail_3 p5b}
we get
 $$\sum_{k= 1}^\infty
 \textrm{E} \left (
 (
 \phi_m
\nabla((\mathbf{u}(t)\times \mathbf{f}_k)\times \mathbf{f}_k),   \phi_m \nabla\mathbf{u}(t))_{L^2
(\mathbb{R}) }
\right )
  +\sum_{k= 1}^\infty
\textrm{E} \left (
 \|   \phi_m \nabla(\mathbf{u}(t)\times \mathbf{f}_k+\mathbf{f}_k)\|_{  L^2
(\mathbb{R})  }^2 \right )
$$
 $$
\le
\textrm{E}
\left (
\|  \phi_m \nabla \mathbf{u} (t)  \|^2_{L^2(\mathbb{R})}
\right )
+ c_3
\textrm{E}
\left (
\|  \phi_m  \mathbf{u} (t)\|^2_{L^2(\mathbb{R})}
\right )
+
   2
\sum_{k=1}^\infty
 \| \phi_m \nabla \mathbf{f} _k \|^2
_{L^{ 2}(\mathbb{R})}
$$
\be\label{tail_3 p6}
+  3  \left (
\sum_{k=1}^\infty
 \|\nabla \mathbf{f} _k \|^2
_{L^{ \infty}(\mathbb{R})} \right )
\left (
\sum_{k=1}^\infty
\| \phi_m   \mathbf{f} _k \|^2
_{L^{ 2}(\mathbb{R})}\right ) ,
\ee
where $c_3>0$ depends only on
$\{\mathbf{f} _k\}_{k=1}^\infty$.

It follows from
\eqref{tail_3 p1}-\eqref{tail_3 p4} and
 \eqref{tail_3 p6} that
$$
{\frac d{dt}}
\textrm{E}
\left (
 \|   \phi_m \nabla
 \mathbf{u} (t) \|^2_{L^2(\mathbb{R})}
 \right )
 + \textrm{E}
\left (
 2\|  \phi_m \Delta
 \mathbf{u} (t) \|^2_{L^2(\mathbb{R})}
 +
 \|  \phi_m \nabla
 \mathbf{u} (t) \|^2_{L^2(\mathbb{R})}
  \right )
 $$
 $$
\le {\frac {c_4}m}
 \textrm{E}
\left (
  \|\mathbf{u} (t) \|^2
 _{H^2(\mathbb{R})}
 +\int_{\mathbb{R}}
 |\mathbf{u} (t) | ^2
 |\nabla \mathbf{u} (t) |^2 dx
 \right )
  + c_4
\textrm{E}
\left (
\|  \phi_m  \mathbf{u} (t)\|^2_{L^2(\mathbb{R})}
\right )
$$
$$
+
   c_4
\sum_{k=1}^\infty
\left (
\| \phi_m   \mathbf{f} _k \|^2
_{L^{ 2}(\mathbb{R})}
+
 \| \phi_m \nabla \mathbf{f} _k \|^2
_{L^{ 2}(\mathbb{R})}
\right ),
$$
 where $c_4>0$ depends only on
$\{\mathbf{f} _k\}_{k=1}^\infty$.
 Solving the inequality, we get  for all
$t\ge 0$,
$$
\textrm{E}
\left (
 \|  \phi_m \nabla
 \mathbf{u} (t) \|^2_{L^2(\mathbb{R})}
 \right )
 \le  e^{-t}
 \textrm{E}
\left (
 \|  \phi_m \nabla
 \mathbf{u} _0\|^2_{L^2(\mathbb{R})}
 \right )
 $$
 $$
 +  {\frac {c_4}m}
  \int_0^t e^{s-t}
  \textrm{E}
\left (
  \|\mathbf{u} (s) \|^2
 _{H^2(\mathbb{R})}
 +\int_{\mathbb{R}}
 |\mathbf{u} (s) | ^2
 |\nabla \mathbf{u} (s) |^2 dx
 \right ) ds
 $$
\be\label{tail_3 p7}
  + c_4  \int_0^t e^{s-t}
  \textrm{E}
\left (
\|  \phi_m  \mathbf{u} (s)\|^2_{L^2(\mathbb{R})}
\right )  ds
   +
   c_4
\sum_{k=1}^\infty
\left (
\| \phi_m   \mathbf{f} _k \|^2
_{L^{ 2}(\mathbb{R})}
+
 \| \phi_m \nabla \mathbf{f} _k \|^2
_{L^{ 2}(\mathbb{R})}
\right ).
\ee
  By \eqref{tail_3 p7} and
  Lemma \ref{ues1} we obtain
for all $t\ge 0$
and $ m \in \mathbb{N}$,
$$
\textrm{E}
\left (
 \|  \phi_m \nabla
 \mathbf{u} (t) \|^2_{L^2(\mathbb{R})}
 \right )
 \le
 \|  \phi_m \nabla
 \mathbf{u} _0\|^2_{L^2(\mathbb{R})}
  +  {\frac {c_5}m}
 (1+ \| \mathbf{u}_0\|^2_{H^1(\mathbb{R})} )
   $$
\be\label{tail_3 p8}
  + c_4  \int_0^t e^{s-t}
  \textrm{E}
\left (
\|  \phi_m  \mathbf{u} (s)\|^2_{L^2(\mathbb{R})}
\right )  ds
   +
   c_4
\sum_{k=1}^\infty
\left (
\| \phi_m   \mathbf{f} _k \|^2
_{L^{ 2}(\mathbb{R})}
+
 \| \phi_m \nabla \mathbf{f} _k \|^2
_{L^{ 2}(\mathbb{R})}
\right ),
\ee
  where $c_5>0$ depends only on
$\{\mathbf{f}_k\}_{k=1}^\infty$.
By \eqref{1.1} we find that
  $$
\|\phi_m \mathbf{u}_0
 \|_{H^1(\mathbb{R} )
 }^2
 +
 c_4
\sum_{k=1}^\infty
\left (
\| \phi_m   \mathbf{f} _k \|^2
_{L^{ 2}(\mathbb{R})}
+
 \| \phi_m \nabla \mathbf{f} _k \|^2
_{L^{ 2}(\mathbb{R})}
\right )
 $$
 $$
 \le
\int_{|x|>{\frac 12 m}}
 (| \mathbf{u}_0 (x)|^2
 +|\nabla  \mathbf{u}_0 (x)|^2)
  dx
   +   \int_{|x|>{\frac 12 m}}
    \sum_{k=1}^\infty
    (|\mathbf{f}_k (x) |^2
    +
    |\nabla \mathbf{f}_k (x) |^2 )dx
    \to 0,
  $$
  as $m\to \infty$,
which together with \eqref{tail_2 p7}
and
\eqref{tail_3 p8} implies that
 for every $\varepsilon>0$,
there exists $M=M(\varepsilon,   \mathbf{u}_0)\ge 1$,
such that
for all $m\ge M$  and $t \ge 0$,
$$
\mathrm{E} \left (
\int_{|x|>  {\frac 34 m} }
 | \nabla \mathbf{u}  (t, \mathbf{u} _0) (x)|^2 dx
 \right )
  \le
 \mathrm{E}
 \left (
 \|\phi_m \nabla \mathbf{u}  (t,  \mathbf{u}_0 )
 \|_{L^2(\mathbb{R} )
 }^2
 \right )
   <\varepsilon ,
$$
which together with Lemma \ref{tail_2}
completes the proof.
   \end{proof}

   By
   Lemma \ref{tail_2a}
   and   the argument of Lemma \ref{tail_3},
   we also have the following uniform estimate.

\begin{lemma}\label{tail_3a}
If \eqref{1.1} holds,
then for every   $\varepsilon>0$
and $R>0$,
there exist $T=T(\varepsilon, R)>0$
and
  $M=M(\varepsilon) \ge 1$
 such that
 for all
 $t\ge T$ and
 $m\ge M$,
 the solution
 ${\mathbf{ u}} (t,
\mathbf{u}_0 )$
of \eqref{equ1.1}
with
  $\|\mathbf{u}_0\|_{
 H^1
 (\mathbb{R})  } \le R$,
 satisfies
$$
 \mathrm{E}
\left (
 \int_{|x|>m}
 \left ( |
 {\mathbf{ u}} (t,
\mathbf{u}_0 )(x) |^2
+
|\nabla
 {\mathbf{ u}} (t,
\mathbf{u}_0 )(x) |^2
\right )
dx
\right )  <\varepsilon.
$$
\end{lemma}

\subsection{Feller property}
In this section, we establish the Feller
property of the  Markov  transition semigroup
$\{\mathbf{P}_t\}_{t\ge 0}$
in terms of the norm topology
instead of the weak topology
in $H^1
(\mathbb{R})$.

\begin{lemma}\label{fp}
If \eqref{1.1} holds, then
the family $\{\mathbf{P}_t\}_{t\ge 0}$
is Feller  in $H^1(\mathbb{R})$; that is,
 $\mathbf{P}_t \varphi \in C_b (H^1(\mathbb{R}))$
 for every $\varphi \in C_b (H^1(\mathbb{R}))$
 and $t\ge 0$.
 \end{lemma}

 \begin{proof}
 Let $t_0> 0$,
   $\varphi \in C_b (H^1(\mathbb{R}))$
  and
 $\mathbf{u}_{0},
    \mathbf{u}_{0,n} \in H^1(\mathbb{R})$
    such that   $ \mathbf{u}_{0,n}
    \to \mathbf{u}_{0}$ in
    $ H^1(\mathbb{R})$.
    We only need to show
    $(\mathbf{P}_{t_0} \varphi )
    (\mathbf{u}_{0,n})
    \to (\mathbf{P}_{t_0} \varphi )
    (\mathbf{u}_{0})$; that is,
     \be\label{fp p1}
     \lim_{n\to \infty}
     \textrm{E}
     \left (
     \varphi (\mathbf{u} (t_0, \mathbf{u}_{0,n}))
     \right )
   = \textrm{E}
     \left (
     \varphi (\mathbf{u} (t_0, \mathbf{u}_{0}))
     \right ),
     \ee
     where
     $\mathbf{u} (\cdot, \mathbf{u}_{0,n})$
     and  $\mathbf{u} (\cdot, \mathbf{u}_{0})$
     are solutions of \eqref{equ1.1}
     with initial data
     $ \mathbf{u}_{0,n} $ and
     $ \mathbf{u}_{0 } $, respectively.
     Since   $ \mathbf{u}_{0,n}
    \to \mathbf{u}_{0}$ in
    $ H^1(\mathbb{R})$, we see that
    the sequence
    $\{ \mathbf{u}_{0,n} \}_{n=1}^\infty$
    is bounded in $ H^1(\mathbb{R})$,
    and thus by the uniform estimates of solutions we infer that
    \be\label{fp p2}
    \textrm{E}
    \left (
    \sup_{0\le t\le t_0}
    \|\mathbf{u} (t, \mathbf{u}_{0,n})
    \|^2_{H^1(\mathbb{R})}
    +\int_0^{t_0}
    \| \mathbf{u} (s,
  \mathbf{u}_{0,n})\|_{H^2(\mathbb{R})}^2 ds
     \right )
      \le c ,
     \ee
     and
      \be\label{fp p3}
    \textrm{E}
    \left (
    \sup_{0\le t\le t_0}
    \|\mathbf{u} (t, \mathbf{u}_{0})
    \|^2_{H^1(\mathbb{R})}
    +\int_0^{t_0}
    \| \mathbf{u} (s,
  \mathbf{u}_{0})\|_{H^2(\mathbb{R})}^2 ds
     \right )
      \le c ,
     \ee
     where $c >0$  is independent of $n\in \mathbb{N}$.
     It follows from
     \eqref{fp p2}-\eqref{fp p3}  that
     for every $R>0$,
     $$
     \mathrm{P}
     \left (
    \sup_{0\le t\le t_0}
    \|\mathbf{u} (t, \mathbf{u}_{0,n})
    \|_{H^1(\mathbb{R})} ^2
     +\int_0^{t_0}
    \| \mathbf{u} (s,
  \mathbf{u}_{0,n})\|_{H^2(\mathbb{R})}^2 ds
    >R
     \right )
       \le {\frac {c}{R } },
     $$
     and
      $$
     \mathrm{P}
     \left (
    \sup_{0\le t\le t_0}
    \|\mathbf{u} (t, \mathbf{u}_{0})
    \|_{H^1(\mathbb{R})} ^2
     +\int_0^{t_0}
    \| \mathbf{u} (s,
  \mathbf{u}_{0})\|_{H^2(\mathbb{R})}^2 ds
    >R
     \right )
       \le {\frac {c}{R } }.
       $$
     Therefore,  for every
        $\delta>0$, there exists
        $R_{\delta}>0$ such that
        for all $n\in \mathbb{N}$,
      \be\label{fp p4}
  \mathrm{ P}
     \left (
    \sup_{0\le t\le t_0}
    \|\mathbf{u} (t, \mathbf{u}_{0,n})
    \|_{H^1(\mathbb{R})} ^2
     +\int_0^{t_0}
    \| \mathbf{u} (s,
  \mathbf{u}_{0,n})\|_{H^2(\mathbb{R})}^2 ds
    >R _{\delta}
     \right )
       \le  {\frac 12} \delta,
       \ee
       and
           \be\label{fp p5}
     \mathrm{P}
     \left (
    \sup_{0\le t\le t_0}
    \|\mathbf{u} (t, \mathbf{u}_{0})
    \|_{H^1(\mathbb{R})} ^2
     +\int_0^{t_0}
    \| \mathbf{u} (s,
  \mathbf{u}_{0})\|_{H^2(\mathbb{R})}^2 ds
    >R _{\delta}
     \right )
       \le  {\frac 12} \delta.
       \ee

    Given $n\in \mathbb{N}$, define the stopping
    times:
    $$
    \tau_{1,n}
  =\inf\left\{t \ge 0;  \|\mathbf{u} (t,
  \mathbf{u}_{0,n})
   \|^2_{H^1(\mathbb{R}) } +
   \int_{0}^{t}\| \mathbf{u} (s,
  \mathbf{u}_{0,n})\|_{H^2(\mathbb{R})}^2 ds
  >R _{\delta}   \right\},$$
and
 $$
    \tau_{2,n}
  =\inf\left\{t \ge 0;  \|\mathbf{u} (t,
  \mathbf{u}_{0})
   \|^2_{H^1(\mathbb{R}) } +
   \int_{0}^{t}\| \mathbf{u} (s,
  \mathbf{u}_{0})\|_{H^2(\mathbb{R})}^2 ds
  >R _{\delta}   \right\}.$$
  As usual, $\inf \emptyset =+\infty$.
  Denote by
  $\tau_n =\tau_{1,n} \wedge \tau_{2,n}$.

    Let
 $
 \mathbf{v}_n (t)
 = \mathbf{u} (t, \mathbf{u}_{0,n})
 -
 \mathbf{u} (t, \mathbf{u}_{0})$
 for $t \ge 0 $.
 Then by \eqref{coin p0} we  get
 for all $t\ge 0$, $\mathrm{P}$-almost surely,
 $$
 \|\nabla   \mathbf{v}_n (t)\|^2
 _{L^2(
 \mathbb{R})}
 + 2 \int_0^t
 \|
  \Delta    \mathbf{v}_n (s)\|^2
 _{L^2(
 \mathbb{R})} ds
 +2 \int_0^t
 \|    \nabla  \mathbf{v}_n (s)\|^2
 _{L^2(
 \mathbb{R})} ds
 $$
 $$
 = \|\nabla   \mathbf{v}_{0,n} \|^2
 _{L^2(
 \mathbb{R})}
 -2
 \int_0^t
 (\Delta \mathbf{v}_n (s),
 \mathbf{v}_n (s) \times
 \Delta \mathbf{u} (s, \mathbf{u}_{0,n} ) ) ds
 $$
 $$
 - 2
 \int_0^t
 (\nabla  \mathbf{v}_n (s),
 \nabla (|\mathbf{u} (s, \mathbf{u}_{0,n} )
 |^2\mathbf{u} (s, \mathbf{u}_{0,n} )
 -  |\mathbf{u} (s, \mathbf{u}_{0} )
 |^2\mathbf{u} (s, \mathbf{u}_{0} ) ))
  ds
   $$
   $$
   -\sum_{k=1}^\infty \int_0^t
   (\Delta \mathbf{v}_n (s),
   (\mathbf{v}_n (s)  \times \mathbf{f}_k)
   \times \mathbf{f}_k ) ds
   +
   \sum_{k=1}^\infty \int_0^t
   \|\nabla
   (\mathbf{v}_n (s)  \times \mathbf{f}_k )
   \|^2_{L^2(\mathbb{R})} ds
   $$
  \be\label{fp p6}
   + 2 \sum_{k=1}^\infty\int_0^t
   (\nabla \mathbf{v}_n (s),
    \mathbf{v}_n (s)  \times \nabla \mathbf{f}_k  )
    dW_k.
 \ee
 By \eqref{fp p6} we have
 for all $t\ge 0$,
 $$
 \mathrm{E}
 \left (
 \|\nabla   \mathbf{v}_n (t\wedge \tau_n )\|^2
 _{L^2(
 \mathbb{R})}
 \right )
 + 2\mathrm{E}
 \left (
  \int_0^{t\wedge \tau_n}
 \|
  \Delta    \mathbf{v}_n (s)\|^2
 _{L^2(
 \mathbb{R})} ds \right )
 +2 \mathrm{E}
 \left (
 \int_0^{ t\wedge \tau_n}
 \|    \nabla  \mathbf{v}_n (s)\|^2
 _{L^2(
 \mathbb{R})} ds \right )
 $$
 $$
 = \mathrm{E}
 \left (
 \|\nabla   \mathbf{v}_{0,n} \|^2
 _{L^2(
 \mathbb{R})} \right )
 -2\mathrm{E}
 \left (
 \int_0^{ t\wedge \tau_n}
 (\Delta \mathbf{v}_n (s),
 \mathbf{v}_n (s) \times
 \Delta \mathbf{u} (s, \mathbf{u}_{0,n} ) ) ds
 \right )
 $$
 $$
 - 2\mathrm{E}
 \left (
 \int_0^ { t\wedge \tau_n}
 (\nabla  \mathbf{v}_n (s),
 \nabla (|\mathbf{u} (s, \mathbf{u}_{0,n} )
 |^2\mathbf{u} (s, \mathbf{u}_{0,n} )
 -  |\mathbf{u} (s, \mathbf{u}_{0} )
 |^2\mathbf{u} (s, \mathbf{u}_{0} ) ))
  ds \right )
   $$
   $$
   -\mathrm{E}
 \left (
 \sum_{k=1}^\infty \int_0^ { t\wedge \tau_n}
   (\Delta \mathbf{v}_n (s),
   (\mathbf{v}_n (s)  \times \mathbf{f}_k)
   \times \mathbf{f}_k ) ds
   \right )
   $$
\be\label{fp p7}
   + \mathrm{E}
 \left (
   \sum_{k=1}^\infty \int_0^ { t\wedge \tau_n}
   \|\nabla
   (\mathbf{v}_n (s)  \times \mathbf{f}_k )
   \|^2_{L^2(\mathbb{R})} ds
   \right ).
\ee

   To deal with the
     second term on the right-hand side
 of
 \eqref{fp p7}, we will use
   the following  embedding  inequality:
   \be\label{fp p8}
   \| \mathbf{u}\|_{L^\infty (\mathbb{R})}
   \le c_1 \| \mathbf{u}\|^{\frac 12}_{L^2 (\mathbb{R})}
   \| \nabla \mathbf{u}\|^{\frac 12}_{L^2 (\mathbb{R})},
   \quad \forall \  \mathbf{u} \in H^1  (\mathbb{R}) ,
   \ee
   where $c_1>0$ is an absolute constant.
  Note that  for all $t\in [0, t_0]$,
   $$
   \left | -2\mathrm{E}
 \left (
 \int_0^{ t\wedge \tau_n}
 (\Delta \mathbf{v}_n (s),
 \mathbf{v}_n (s) \times
 \Delta \mathbf{u} (s, \mathbf{u}_{0,n} ) ) ds
 \right ) \right |
 $$
       $$
  \le   2\mathrm{E}
 \left (
 \int_0^{ t\wedge \tau_n}
 \|\Delta \mathbf{v}_n (s)\|_{L^2 (\mathbb{R})}
 \|\mathbf{v}_n (s) \|_{L^\infty (\mathbb{R})}
 \|\Delta \mathbf{u} (s, \mathbf{u}_{0,n} )\|_{L^2 (\mathbb{R})} ds
 \right )
 $$
\be\label{fp p8a}
  \le  {\frac 14}   \mathrm{E}
 \left (
 \int_0^{ t\wedge \tau_n}
 \|\Delta \mathbf{v}_n (s)\|^2_{L^2 (\mathbb{R})}
 ds
 \right )
 +  4 \mathrm{E}
 \left (
 \int_0^{ t\wedge \tau_n}
 \|\mathbf{v}_n (s) \|^2_{L^\infty (\mathbb{R})}
 \|\Delta \mathbf{u} (s, \mathbf{u}_{0,n} )\|^2_{L^2 (\mathbb{R})} ds
 \right )  .
\ee
For the second term on the right-hand side
of \eqref{fp p8a},
 by \eqref{fp p8}
we have for all $t\in [0,t_0]$,
$$
4 \mathrm{E}
 \left (
 \int_0^{ t\wedge \tau_n}
 \|\mathbf{v}_n (s) \|^2_{L^\infty (\mathbb{R})}
 \|\Delta \mathbf{u} (s, \mathbf{u}_{0,n} )\|^2_{L^2 (\mathbb{R})} ds
 \right )
 $$
 $$
  \le   4c_1^2  \mathrm{E}
 \left (
 \int_0^{ t\wedge \tau_n}
 \|\mathbf{v}_n (s) \| _{L^2 (\mathbb{R})}
 \|\nabla \mathbf{v}_n (s) \| _{L^2 (\mathbb{R})}
 \|\Delta \mathbf{u} (s, \mathbf{u}_{0,n} )\|^2_{L^2 (\mathbb{R})} ds
 \right )
 $$
 $$
  \le    4c_1^2 \mathrm{E}
 \left (
 \sup_{s\in [0,t_0]}\|\mathbf{v}_n (s) \| _{L^2 (\mathbb{R})}
 \int_0^{ t\wedge \tau_n}
  \|\nabla \mathbf{v}_n (s) \| _{L^2 (\mathbb{R})}
 \|\Delta \mathbf{u} (s, \mathbf{u}_{0,n} )\|^2_{L^2 (\mathbb{R})} ds
 \right )
 $$
 $$
  \le    4c_1^2  \mathrm{E}
 \left (
 \sup_{s\in [0,t_0]}\|\mathbf{v}_n (s) \| _{L^2 (\mathbb{R})}
 \int_0^{ t\wedge \tau_n}
  (\|  \mathbf{u}  (s,
  \mathbf{u}_{0,n}
     \| _{H^1 (\mathbb{R})}
  +
  \| \mathbf{u}  (s,
  \mathbf{u}_{0}
  ) \| _{H^1(\mathbb{R})}
  )
 \|\Delta \mathbf{u} (s, \mathbf{u}_{0,n} ) \|^2_{L^2 (\mathbb{R})} ds
 \right )
 $$
 $$
  \le    8c_1^2R _{\delta}^{\frac 12}
  \mathrm{E}
 \left (
 \sup_{s\in [0,t_0]}\|\mathbf{v}_n (s) \| _{L^2 (\mathbb{R})}
 \int_0^{ t\wedge \tau_n}
  \|\Delta \mathbf{u} (s, \mathbf{u}_{0,n}  )\|^2_{L^2 (\mathbb{R})} ds
 \right )
 $$
 \be\label{fp p8b}
  \le
  8c_1^2R _{\delta}^{\frac 32}   \mathrm{E}
 \left (
 \sup_{s\in [0,t_0]}\|\mathbf{v}_n (s) \| _{L^2 (\mathbb{R})}
  \right )  .
  \ee
  By \eqref{fp p8a}-\eqref{fp p8b} we get
   for all $t\in [0, t_0]$,
   $$
   \left | -2\mathrm{E}
 \left (
 \int_0^{ t\wedge \tau_n}
 (\Delta \mathbf{v}_n (s),
 \mathbf{v}_n (s) \times
 \Delta \mathbf{u} (s, \mathbf{u}_{0,n} ) ) ds
 \right ) \right |
 $$
 \be\label{fp p9}
  \le  {\frac 14}   \mathrm{E}
 \left (
 \int_0^{ t\wedge \tau_n}
 \|\Delta \mathbf{v}_n (s)\|^2_{L^2 (\mathbb{R})}
 ds
 \right )
 + 8c_1^2R _{\delta}^{\frac 32}   \mathrm{E}
 \left (
 \sup_{s\in [0,t_0]}\|\mathbf{v}_n (s) \| _{L^2 (\mathbb{R})}
  \right )  .
\ee

 For the third term on the right-hand side of \eqref{fp p7}
 we have for all $t\in [0,t_0]$,
 $$
  2\mathrm{E}
 \left (
 \int_0^ { t\wedge \tau_n}
 (\Delta \mathbf{v}_n (s),
  |\mathbf{u} (s, \mathbf{u}_{0,n} )
 |^2\mathbf{u} (s, \mathbf{u}_{0,n} )
 -  |\mathbf{u} (s, \mathbf{u}_{0} )
 |^2\mathbf{u} (s, \mathbf{u}_{0} )  )
  ds \right )
   $$
   $$
   \le
  2\mathrm{E}
 \left (
 \int_0^ { t\wedge \tau_n}
 \left (
 \int_{\mathbb{R}}
 |\Delta \mathbf{v}_n (s)|
 |  \mathbf{v}_n (s)|
 (
  |\mathbf{u} (s, \mathbf{u}_{0,n} )| ^2
 + |\mathbf{u} (s, \mathbf{u}_{0} )
 |^2
 +|\mathbf{u} (s, \mathbf{u}_{0,n} )|  |\mathbf{u} (s, \mathbf{u}_{0} )|
  ) dx\right )
  ds \right )
   $$
 $$
   \le
  3\mathrm{E}
 \left (
 \int_0^ { t\wedge \tau_n}
 \|\Delta \mathbf{v}_n (s)\|_{L^2(\mathbb{R})}
 \| \mathbf{v}_n (s)\|_{L^2(\mathbb{R})}
 (
  \|\mathbf{u} (s, \mathbf{u}_{0,n} )\| ^2
  _{L^\infty
  (\mathbb{R})}
 + \|\mathbf{u} (s, \mathbf{u}_{0} )
 \|^2 _{L^\infty (\mathbb{R})}
  )
  ds \right )
   $$
   $$
   \le
  3c_1^2 \mathrm{E}
 \left (
 \int_0^ { t\wedge \tau_n}
 \|\Delta \mathbf{v}_n (s)\|_{L^2(\mathbb{R})}
 \| \mathbf{v}_n (s)\|_{L^2(\mathbb{R})}
 (
  \|\mathbf{u} (s, \mathbf{u}_{0,n} )\| ^2
  _{H^1
  (\mathbb{R})}
 + \|\mathbf{u} (s, \mathbf{u}_{0} )
 \|^2 _{H^1 (\mathbb{R})}
  )
  ds \right )
   $$
   $$
   \le
  3c_1^2 R _{\delta} \mathrm{E}
 \left (
 \int_0^ { t\wedge \tau_n}
 \|\Delta \mathbf{v}_n (s)\|_{L^2(\mathbb{R})}
 \| \mathbf{v}_n (s)\|_{L^2(\mathbb{R})}
   ds \right )
   $$
      \be\label{fp p10}
   \le
  {\frac 14}  \mathrm{E}
 \left (
 \int_0^ { t\wedge \tau_n}
 \|\Delta \mathbf{v}_n (s)\|^2 _{L^2(\mathbb{R})}
 ds \right )
 +9c_1^4 R^2 _{\delta}
 \mathrm{E}
 \left (
 \int_0^ { t\wedge \tau_n}
 \| \mathbf{v}_n (s)\|^2_{L^2(\mathbb{R})}
   ds \right ).
   \ee

    For the last two terms
    on the right-hand side
    of \eqref{fp p7} we have
   $$
   -\mathrm{E}
 \left (
 \sum_{k=1}^\infty \int_0^ { t\wedge \tau_n}
   (\Delta \mathbf{v}_n (s),
   (\mathbf{v}_n (s)  \times \mathbf{f}_k)
   \times \mathbf{f}_k ) ds
   \right )
   $$
$$
   + \mathrm{E}
 \left (
   \sum_{k=1}^\infty \int_0^ { t\wedge \tau_n}
   \|\nabla
   (\mathbf{v}_n (s)  \times \mathbf{f}_k )
   \|^2_{L^2(\mathbb{R})} ds
   \right )
$$
$$
\le
\sum_{k=1}^\infty
\|\mathbf{f}_k\|^2_{L^\infty (\mathbb{R})}
\mathrm{E}
 \left (  \int_0^ { t\wedge \tau_n}
   \|\Delta
   \mathbf{v}_n (s) \|_{L^2(\mathbb{R})}
   \|
   \mathbf{v}_n (s) \|_{L^2(\mathbb{R})}
   ds
   \right )
   $$
   $$
   +2
   \sum_{k=1}^\infty
\|\mathbf{f}_k\|^2_{W^{1,\infty}  (\mathbb{R})}
\mathrm{E}
 \left (  \int_0^ { t\wedge \tau_n}
   (\|\nabla
   \mathbf{v}_n (s) \|^2_{L^2(\mathbb{R})}
   +
   \|\mathbf{v}_n (s) \|^2_{L^2(\mathbb{R})}
   )
   ds
   \right )
   $$
   \be\label{fp p11}
  \le
  {\frac 14}
  \mathrm{E}
 \left (  \int_0^ { t\wedge \tau_n}
   \|\Delta
   \mathbf{v}_n (s) \|^2_{L^2(\mathbb{R})}
   ds
   \right )
   +
   c_2 \mathrm{E}
 \left (  \int_0^ { t\wedge \tau_n}
   (\|
   \mathbf{v}_n (s) \|^2_{L^2(\mathbb{R})}
   +
   \| \nabla
   \mathbf{v}_n (s) \|^2_{L^2(\mathbb{R})}
   )
   ds
   \right ),
   \ee
      where $c_2>0$ depends only
      on $\{\mathbf{f}_k\}_{k=1}^\infty$.

      It follows from \eqref{fp p7}
      and \eqref{fp p9}-\eqref{fp p11} that
      for all $t\in [0,t_0]$,
       $$
        \mathrm{E}
 \left (
 \| \nabla \mathbf{v}_n (t\wedge \tau_n )\|^2_{L^2(\mathbb{R} )}
 \right )
  \le
  \| \nabla \mathbf{v}_n (0)\|^2_{L^2(\mathbb{R} )}
  $$
 \be\label{fp p12}
  +c_3  R^{\frac 32}
  _\delta  \mathrm{E}
 \left (\sup_{s\in [0, t_0]}\| \mathbf{v}_n  (s)\|
   _{L^2(\mathbb{R})}
     \right )
     +
     c_3 R^2_\delta  \mathrm{E}
 \left (  \int_0^ { t\wedge \tau_n}
   \|
   \mathbf{v}_n (s) \|^2_{H^1 (\mathbb{R})}
   ds
   \right ),
 \ee
       where $c_3>0$ depends only
      on $\{\mathbf{f}_k\}_{k=1}^\infty$.

      On the other hand, by \eqref{conin p4} with $d=1$ we have
      for all $t\in [0,t_0]$,
      $$
      \mathrm{E}
 \left (
 \| \mathbf{v}_n (t\wedge \tau_n )\|^2_{L^2(\mathbb{R} )}
 \right )
  \le
 \| \mathbf{v}_n (0)\|^2_{L^2(\mathbb{R} )}
 $$
 $$
 +
c_4   \mathrm{E}
 \left (
  \int_0^{t\wedge \tau_n}
 \| \mathbf{u}   (s, \mathbf{u} _0 )\|^2_{H^1(\mathbb{R} )}
 \| \mathbf{u}    (s, \mathbf{u} _0 ) \|^\frac{2}{3}_{H^2(\mathbb{R} )}
\| \mathbf{v}_n  (s)\|^2
   _{L^2(\mathbb{R})} ds
   \right )
 $$
 $$
 \le
 \| \mathbf{v}_n (0)\|^2_{L^2(\mathbb{R} )}
  +
c_4  R_\delta  \mathrm{E}
 \left (\sup_{s\in [0, t_0]}\| \mathbf{v}_n  (s)\|^2
   _{L^2(\mathbb{R} )}
  \int_0^{t\wedge \tau_n}
  \| \mathbf{u}    (s, \mathbf{u} _0 ) \|^\frac{2}{3}_{H^2(\mathbb{R} )}
   ds
   \right )
 $$
     \be\label{fp p13}
 \le
 \| \mathbf{v}_n (0)\|^2_{L^2(\mathbb{R} )}
  +
c_4  R^\frac{4}{3}_\delta t^\frac{2}{3}  \mathrm{E}
 \left (\sup_{s\in [0, t_0]}\| \mathbf{v}_n  (s)\|^2
   _{L^2(\mathbb{R} )}
     \right ),
     \ee
where $c_4>0$ is an absolute constant.

By \eqref{fp p12}-\eqref{fp p13} we obtain
   for all $t\in [0,t_0]$,
       $$
        \mathrm{E}
 \left (
 \|   \mathbf{v}_n (t\wedge \tau_n )\|^2_{H^1
 (\mathbb{R} )}
 \right )
  \le
  \|   \mathbf{v}_n (0)\|^2_{H^1(\mathbb{R} )}
  +
    c_3 R^2_\delta  \mathrm{E}
 \left (  \int_0^ { t }
   \|
   \mathbf{v}_n (s\wedge \tau_n ) \|^2_{H^1 (\mathbb{R})}
   ds
   \right )
  $$
 \be\label{fp p20}
  +c_3  R^{\frac 32}
  _\delta  \mathrm{E}
 \left (\sup_{s\in [0, t_0]}\| \mathbf{v}_n  (s)\|
   _{L^2(\mathbb{R})}
     \right )
   +
c_4 R^\frac{4}{3}_\delta t^\frac{2}{3}  \mathrm{E}
 \left (\sup_{s\in [0, t_0]}\| \mathbf{v}_n  (s)\|^2
   _{L^2(\mathbb{R} )}
     \right ) .
     \ee
 Applying Gronwall's inequality to \eqref{fp p20}
 to obtain that for all $t\in[0,t_0]$,
         $$
        \mathrm{E}
 \left (
 \|   \mathbf{v}_n (t\wedge \tau_n )\|^2_{H^1
 (\mathbb{R} )}
 \right )
  \le
  e^{c_3R^2_\delta t}  \|   \mathbf{v}_n (0)\|^2_{H^1(\mathbb{R} )}
  $$
  $$
   +
   e^{c_3R^2_\delta t}
   \left ( c_3R^{\frac 32} _\delta
    \mathrm{E}
 \left (\sup_{s\in [0, t_0]}\| \mathbf{v}_n  (s)\|
   _{L^2(\mathbb{R})}
     \right )
   +
c_4  R^{\frac 43} _\delta t^\frac{2}{3} \mathrm{E}
 \left (\sup_{s\in [0, t_0]}\| \mathbf{v}_n  (s)\|^2
   _{L^2(\mathbb{R} )}
     \right )
     \right )  ,
    $$
    which along with Lemma \ref{conin}
    implies that
        $$
       \lim_{n\to \infty} \mathrm{E}
 \left (
 \|   \mathbf{v}_n (t_0 \wedge \tau_n )\|^2_{H^1
 (\mathbb{R} )}
 \right ) =0,
 $$
 that is,
 \be\label{fp p21}
       \lim_{n\to \infty} \mathrm{E}
 \left (
 \|   \mathbf{u} (t_0 \wedge \tau_n,
 \mathbf{u}_{0,n}
  )
  -
   \mathbf{u} (t_0 \wedge \tau_n,
 \mathbf{u}_{0 }
  )
  \|^2_{H^1
 (\mathbb{R} )}
 \right ) =0.
 \ee

 Given $\eta>0$,
 by \eqref{fp p4}-\eqref{fp p5} we get
 $$
\mathrm{ P} \left (\omega \in \Omega:
  \|   \mathbf{u} (t_0  ,
 \mathbf{u}_{0,n}
  )
  -
   \mathbf{u} (t_0  ,
 \mathbf{u}_{0 }
  )
  \| _{H^1
 (\mathbb{R} )}
 >\eta
 \right )
 $$
 $$
=
 \mathrm{P} \left ( \omega \in \Omega: \tau_n <t_0, \
 \|   \mathbf{u} (t_0  ,
 \mathbf{u}_{0,n}
  )
  -
   \mathbf{u} (t_0  ,
 \mathbf{u}_{0 }
  )
  \| _{H^1
 (\mathbb{R} )}
 >\eta \right )
 $$
 $$
 +
 \mathrm{P} \left ( \omega \in \Omega: \tau_n  \ge t_0,  \
 \|   \mathbf{u} (t_0  ,
 \mathbf{u}_{0,n}
  )
  -
   \mathbf{u} (t_0  ,
 \mathbf{u}_{0 }
  )
  \| _{H^1
 (\mathbb{R} )}
 >\eta
 \right )
 $$
 $$
\le \delta
+
 \mathrm{P} \left ( \omega \in \Omega: \tau_n  \ge t_0,  \
 \|   \mathbf{u} (t_0  ,
 \mathbf{u}_{0,n}
  )
  -
   \mathbf{u} (t_0  ,
 \mathbf{u}_{0 }
  )
  \| _{H^1
 (\mathbb{R} )}
 >\eta
 \right )
 $$
 $$
\le \delta
+
 \mathrm{P }\left ( \omega \in \Omega:
 \|   \mathbf{u} (t_0 \wedge \tau_n ,
 \mathbf{u}_{0,n}
  )
  -
   \mathbf{u} (t_0\wedge \tau_n  ,
 \mathbf{u}_{0 }
  )
  \| _{H^1
 (\mathbb{R} )}
 >\eta
 \right )
 $$
 \be\label{fp p22}
\le \delta
+ \eta^{-2}
 \mathrm{E}
 \left (
 \|   \mathbf{u} (t_0 \wedge \tau_n ,
 \mathbf{u}_{0,n}
  )
  -
   \mathbf{u} (t_0\wedge \tau_n  ,
 \mathbf{u}_{0 }
  )
  \|^2 _{H^1
 (\mathbb{R} )}
  \right ) .
  \ee
 First letting $n\to \infty$ and then $\delta \to 0$
 in \eqref{fp p22}, we get
 for every  $\eta>0$,
\be\label{fp p23}
 \lim_{n\to \infty}
 \mathrm{P} \left (\omega \in \Omega:
  \|   \mathbf{u} (t_0  ,
 \mathbf{u}_{0,n}
  )
  -
   \mathbf{u} (t_0  ,
 \mathbf{u}_{0 }
  )
  \| _{H^1
 (\mathbb{R} )}
 >\eta
 \right ) =0,
\ee
 that is,
  $\mathbf{u} (t_0  ,
 \mathbf{u}_{0,n}
  )
  \to  \mathbf{u} (t_0  ,
 \mathbf{u}_{0 }
  )$ in $H^1(\mathbb{R})$ in probability.
  Then \eqref{fp p1}
  follows from \eqref{fp p23},
  the fact $\varphi\in C_b(H^1(\mathbb{R}))$
  and the dominated convergence theorem.
   \end{proof}

\subsection{Tightness of time averaged measures}
 In this subsection, we  prove the tightness of
the time averaged measures
 $\mu_n$,  $n\in \mathbb{N}$, which is defined by
 \be\label{avet1}
 \mu_n
 (\Gamma )
 ={\frac 1n}
 \int_0^n
 \mathrm{P}(\mathbf{u} (t, \mathbf{u}_0)\in \Gamma) dt,
 \ee
 where $\Gamma \in \mathcal{B} (H^1(\mathbb{R}))$,
 and $\mathbf{u} (\cdot, \mathbf{u}_0)$ is the
 solution of \eqref{equ1.1}  with initial data
 $  \mathbf{u}_0  \in H^1(\mathbb{R})$.

\begin{lemma}\label{ave_1}
If \eqref{1.1} holds, then the   sequence
$\{\mu_n\}_{n=1}^\infty$ given by
\eqref{avet1} is   tight  in  $H^1(\mathbb{R})$.
\end{lemma}

\begin{proof}
Note that, there exists a constant $c_1>0$ depending only
on
$ \mathbf{u}_0 $ such that for all $n\in \mathbb{N}$,
\be\label{ave_1 p1}
{\frac 1n}
\int_0^n
\mathrm{E}
\left (\|\mathbf{u} (t, \mathbf{u}_0) \|^2_{H^2(\mathbb{R})}
\right ) dt \le c_1  .
\ee
On the other hand, by Lemma \ref{tail_3}
we find that for every $\varepsilon>0$ and
$m\in \mathbb{N}$, there exists
a positive integer
$n_m$ depending on $m, \varepsilon
$ and $ \mathbf{u}_0 $  such that
for all $t\ge 0$,
 \be\label{ave_1 p2}
 \mathrm{E}
\left (
\int_{|x| \ge {\frac 12} n_m}
(|\mathbf{u} (t, \mathbf{u}_0) (x)|^2
+|\nabla \mathbf{u} (t, \mathbf{u}_0) (x)|^2)
dx
\right )
<2^{-4m} \varepsilon.
\ee
Let $\phi_m$ be the cut-off function
used in the proof of Lemma \ref{tail_2}, and
$$
\mathbf{v}_m (t)=\phi_{n_m} \mathbf{u} (t, \mathbf{u}_0),
\quad
 \mathbf{w} _m (t)= (1-\phi_{n_m} ) \mathbf{u} (t, \mathbf{u}_0).
$$
Then by \eqref{ave_1 p2}  we have for all $t\ge 0$,
\be\label{ave_1 p3a}
 \mathrm{E}
\left (
\|\mathbf{v}_m (t)\|^2_{H^1(\mathbb{R})}
\right )
<2^{-4m} \varepsilon.
\ee
By \eqref{ave_1 p1} we find that
 there exists a constant $c_2>0$ depending only
on
$ \mathbf{u}_0 $ such that for all $n\in \mathbb{N}$,
\be\label{ave_1 p3}
{\frac 1n}
\int_0^n
\mathrm{E}
\left (\|\mathbf{w} _m (t ) \|^2_{H^2(\mathbb{R})}
\right ) dt \le c_2  .
\ee
Define the sets:
 $$\mathscr{W}_m^\varepsilon:=\left\{\mathbf{u}
 \in
 H^2(\mathbb{R})
 :
 \mathbf{u} =0
 \ \text{for } |x| \ge n_m
 \ \text{ and }
 \|\mathbf{u} \|_{ H^2(\mathbb{R})}\leq \frac{2^m\sqrt{c_2}}{\sqrt{\varepsilon}}\right\},
 $$
  and
$$\mathscr{Y}_m^\varepsilon:
=\left\{\mathbf{u}\in
H^1(\mathbb{R})
:
\| \mathbf{u} -\mathbf{v}\|_{H^1(\mathbb{R})}
\le 2^{-m}
\ \ \text{for some } \mathbf{v} \in \mathscr{W}_1^\varepsilon
\right \} ,
$$
and
$$\mathscr{Y}^\varepsilon=\bigcap_{m=1}^\infty
\mathscr{Y}_m^\varepsilon.
$$
Then $\mathscr{Y}^\varepsilon$
is compact in  $H^1(\mathbb{R})$.
In addition, by \eqref{ave_1 p3a} and \eqref{ave_1 p3}
one can verify that
$\mu_n (\mathscr{Y}^\varepsilon)> 1-\varepsilon $
for all $n\in \mathbb{N}$,
and hence $\{\mu_n\}_{n=1}^\infty$ is tight
in $H^1(\mathbb{R})$.
 \end{proof}

 Based on Lemma \ref{ave_1}, by
  the Krylov-Bogoliubov theory, we also obtain
  the existence of invariant measures
  of \eqref{equ1.1} by the
  Feller property (Lemma \ref{fp})
  with respect to  the norm topology
  instead of the weak topology in
  $H^1(\mathbb{R})$.

\subsection{Limits of invariant measures} In this subsection, we consider the limits  of invariant measures
of \eqref{1.3*}
with respect to $\varepsilon$.
Given   $\varepsilon
\in [0,1]$,
 denote by $\mathcal{I}^\varepsilon$
 the set  of all   invariant measures of equation \eqref{1.3*} corresponding to $\varepsilon$.
 We will  prove
 the union $\bigcup_{\varepsilon\in [0,1]}\mathcal{I}^\varepsilon$ is tight in $H^1(\mathbb{R})$.
 To that end,
 we recall
 that all the uniform estimates established
 in  the  previous sections are valid
 for the solutions of \eqref{1.3*}
 which are actually uniform for all
 $\varepsilon \in [0,1]$.
 In particular,
Lemmas \ref{ues1a} and \ref{tail_3a}
 apply to the solutions
 of \eqref{1.3*} which are presented below.

 \begin{lemma} \label{ues1aa}
If \eqref{1.1}  holds,
then for every
$R>0$, there exists
$T=T(R)>0$ such that
for all $t\ge T$ and $\varepsilon
\in [0,1]$,
 the solution $\mathbf{u}^\varepsilon
(t,  \mathbf{u}_0)$ of \eqref{1.3*}
with $\|\mathbf{u}_0 \|_{ H^1(\mathbb{R})}
  \le R$
 satisfies
 $$
\mathrm{E}
\left (
 \|
 \mathbf{u}^\varepsilon (t,  \mathbf{u}_0) \|^2_{H^1(\mathbb{R})}
 \right )
 +
  \int_0^t e^{s-t}  \mathrm{E}
\left ( \|
 \mathbf{u}^\varepsilon (s,  \mathbf{u}_0) \|^2_{H^2(\mathbb{R})}
 +
 \int_{\mathbb{R}}
 |\mathbf{u}^\varepsilon (s,  \mathbf{u}_0)|^2
 |\nabla \mathbf{u}^\varepsilon (s,  \mathbf{u}_0)|^2  dx
 \right )
 ds
 \le
   C,
 $$
 and
 $$
   {\frac 1t} \int_0^t  \mathrm{E}
\left (
  \|
 \mathbf{u}^\varepsilon (s, \mathbf{u}_0) \|^2_{H^2(\mathbb{R})}
 +
 \int_{\mathbb{R}}
  |\mathbf{u}^\varepsilon (s,\mathbf{u}_0)|^2
 |\nabla \mathbf{u}^\varepsilon (s,\mathbf{u}_0)|^2  dx
 \right )  ds
\le
C,
$$
where $C>0$ is a constant  depending
  on $\{\mathbf{f}_k\}_{k=1}^\infty$
  but not on $\mathbf{u}_0$ or
  $\varepsilon$.
\end{lemma}

\begin{lemma}\label{tail_3aa}
If \eqref{1.1} holds,
then for every   $\delta>0$
and $R>0$,
there exist $T=T(\delta, R)>0$
and
  $M=M(\delta) \ge 1$
 such that
 for all
 $t\ge T$,
 $m\ge M$ and $\varepsilon
 \in [0,1]$,
 the solution
 ${\mathbf{ u}}^\varepsilon  (t,
\mathbf{u}_0 )$
of \eqref{1.3*}
with
  $\|\mathbf{u}_0\|_{
 H^1
 (\mathbb{R})  } \le R$
 satisfies
$$
 \mathrm{E}
\left (
 \int_{|x|>m}
 \left ( |
 {\mathbf{ u}}^\varepsilon  (t,
\mathbf{u}_0 )(x) |^2
+
|\nabla
 {\mathbf{ u}}^\varepsilon  (t,
\mathbf{u}_0 )(x) |^2
\right )
dx
\right )  <\delta.
$$
\end{lemma}

 Next, We   prove
 the union $\bigcup_{\varepsilon\in [0,1]}\mathcal{I}^\varepsilon$ is tight in $H^1(\mathbb{R})$.

\begin{proposition}\label{tigu1}
If \eqref{1.1} holds, then
the  set $\bigcup_{\varepsilon\in [0,1]}\mathcal{I}^\varepsilon$ is tight in $H^1(\mathbb{R})$.
\end{proposition}

\begin{proof}
By Lemma \ref{ues1aa} we find that
for every
$ \mathbf{u}_0 \in H^1(\mathbb{R})$,
there exists $T_1=T_1(\mathbf{u}_0)>0$
such that for all $t\ge T_1$
and $\varepsilon\in [0,1]$,
the solution
$\mathbf{u}^\varepsilon
 (t, \mathbf{u}_0)$
 of \eqref{1.3*} satisfies
 \be\label{tigu1 p1}
{\frac 1t}
\int_0^t
\mathrm{E}
\left (\|\mathbf{u}^\varepsilon  (s, \mathbf{u}_0) \|^2_{H^2(\mathbb{R})}
\right ) ds \le c_1  ,
\ee
where $c_1>0$ is a constant  depending
  on $\{\mathbf{f}_k\}_{k=1}^\infty$
  but not on $\mathbf{u}_0$ or
  $\varepsilon$.
By Lemma \ref{tail_3aa}
 we infer that for every
$ \mathbf{u}_0 \in H^1(\mathbb{R})$,
 $\delta>0$ and
$m\in \mathbb{N}$, there exists
$T_2=T_2(\delta, m, \mathbf{u}_0)>T_1$
and
$n_m=n_m (\delta, m)\in \mathbb{N}$
   such that
for all $t\ge T_2 $ and $\varepsilon\in [0,1]$,
 \be\label{tigu1 p2}
 \mathrm{E}
\left (
\int_{|x| \ge {\frac 12} n_m}
(|\mathbf{u}^\varepsilon (t, \mathbf{u}_0) (x)|^2
+|\nabla \mathbf{u}^\varepsilon (t, \mathbf{u}_0) (x)|^2)
dx
\right )
<2^{-4m} \delta.
\ee
Let $\phi_m$ be the cut-off function
as in the proof of Lemma \ref{tail_2}, and
$$
\mathbf{v}_m ^\varepsilon(t)=\phi_{n_m}
 \mathbf{u}^\varepsilon (t, \mathbf{u}_0),
\quad
 \mathbf{w} ^\varepsilon_m (t)= (1-\phi_{n_m} ) \mathbf{u}^\varepsilon (t, \mathbf{u}_0).
$$
By \eqref{tigu1 p2}  we have
for all $t\ge T_2 $ and $\varepsilon\in [0,1]$,
 \be\label{tigu1 p3}
 \mathrm{E}
\left (
\|\mathbf{v}_m^\varepsilon  (t)\|^2_{H^1(\mathbb{R})}
\right )
<2^{-4m} \delta.
\ee
By \eqref{tigu1 p1} we
infer  that
 there exists a constant $c_2>0$ depending only
on    $\{\mathbf{f}_k\}_{k=1}^\infty$
  but not on $\mathbf{u}_0$ or
  $\varepsilon$   such that for all
   $t\ge T_2 $,
$m\in \mathbb{N}$
 and $\varepsilon\in [0,1]$,
 \be\label{tigu1 p4}{\frac 1t}
\int_0^t
\mathrm{E}
\left (\|\mathbf{w}^\varepsilon
 _m (s ) \|^2_{H^2(\mathbb{R})}
\right ) ds \le c_2.
\ee
Denote by
 $$
  \mathscr{W}_1^{m, \delta}  =\left\{\mathbf{u}
 \in
 H^2(\mathbb{R})
 :
 \mathbf{u} =0
 \ \text{for } |x| \ge n_m
 \ \text{ and }
 \|\mathbf{u} \|_{ H^2(\mathbb{R})}\leq \frac{2^m\sqrt{c_2}}{\sqrt{\delta}}\right\},
 $$
  and
$$ \mathscr{W}_2^{m, \delta}
=\left\{\mathbf{u}\in
H^1(\mathbb{R})
:
\| \mathbf{u} -\mathbf{v}\|_{H^1(\mathbb{R})}
\le 2^{-m}
\ \ \text{for some } \mathbf{v} \in
\mathscr{W}_1^{m, \delta}
\right \} ,
$$
and
$$
\mathscr{W}^{\delta}
=\bigcap_{m=1}^\infty
 \mathscr{W}_2^{m, \delta}.
$$
Then $ \mathscr{W}^{\delta}$
is compact in  $H^1(\mathbb{R})$,
and for every
$\mu^\varepsilon \in \mathcal{I}^
\varepsilon$,
  \be\label{tigu1 p5}\mu^\varepsilon(\mathscr{W}^{\delta})=\lim_{n\rightarrow\infty}\mu^\varepsilon(\mathscr{W}_n^{\delta}),
  \ee
 where $\mathscr{W}_n^{\delta}=\bigcap_{m=1}^n \mathscr{W}_2^{m, \delta}$.
   For every $n\in\mathbb{N}$,
by  the invariance of $\mu^\varepsilon$,
Fubini's theorem and Fatou's lemma, we have
\begin{align}\label{tigu1 p6}
&\mu^\varepsilon(
H^1(\mathbb{R})\setminus
\mathscr{W}_n^{\delta} )\nonumber\\&
 =   \int_{H^1
(\mathbb{R})}
\mathrm{P}(\mathbf{u}^\varepsilon
(s, \mathbf{u}_0) \notin  \mathscr{W}_n^{\delta} )  d\mu^\varepsilon
(\mathbf{u}_0)
 \nonumber\\
&= \limsup_{t\rightarrow\infty}
\int_{H^1(\mathbb{R})
} \left (
\frac{1}{t}
\int_{0}^{t}\mathrm{P}(\mathbf{u}^\varepsilon
(s, \mathbf{u}_0) \notin  \mathscr{W}_n^{\delta} )ds
\right )
d\mu^\varepsilon
(\mathbf{u}_0)\nonumber\\
&\leq\int_{H^1(\mathbb{R})}\limsup_{t\rightarrow\infty}
\left ({\frac{1}{t}}
 \int_{0}^{t}\sum_{m=1}^n
 \mathrm{P}(\mathbf{u}^\varepsilon(s, \mathbf{u}_0) \notin  \mathscr{W}_2^{m,\delta} )ds
 \right )  d\mu^\varepsilon
 (\mathbf{u}_0)
 \nonumber\\
&\leq \int_{H^1(\mathbb{R}) }\limsup_{t\rightarrow\infty}
\left (
\frac{1}{t}\int_{0}^{t}\sum_{m=1}^n\mathrm{P}(\mathbf{w}_m^\varepsilon
(s)
\notin  \mathscr{W}_1^{m,\delta} )ds
\right ) d\mu^\varepsilon
(\mathbf{u}_0)
\nonumber\\
&\quad+\int_{H^1(\mathbb{R}) }\limsup_{t\rightarrow\infty}
\left (
\frac{1}{t}\int_{0}^{t}\sum_{m=1}^n\mathrm{P}\left(\mathbf{w}_m ^{\varepsilon} (s) \in \mathscr{W}_1^{m,\delta}~ {\rm and}~ \|\mathbf{v}_m ^{\varepsilon }(s)\|_{H^1
(\mathbb{R}) }>  \frac{1}{2^m}\right)ds
\right ) d\mu^\varepsilon
(\mathbf{u}_0)
\nonumber\\
&\leq \sum_{m=1}^n\frac{\delta}{2^{2m }c_2}\int_{H^1(\mathbb{R}) }\limsup_{t\rightarrow\infty}
\left (
\frac{1}{t}\int_{0}^{t}\mathrm{E}\|\mathbf{w}_m^{\varepsilon }(s)\|^2_{H^2(\mathbb{R})}ds
\right ) d\mu^\varepsilon
(\mathbf{u}_0)
\nonumber\\
&\quad+\sum_{m=1}^n
2^{2m}\int_{H^1(\mathbb{R}) }\limsup_{t\rightarrow\infty}
\left (
\frac{1}{t}\int_{0}^{t}\mathrm{E}\|\mathbf{v}_m^{\varepsilon}(s)\|^2_{H^1(\mathbb{R}) }ds
\right ) d\mu^\varepsilon(\mathbf{u}_0).
\end{align}
For the first term on the right-hand
side of \eqref{tigu1 p6},
  by  \eqref{tigu1 p4}  we have
\begin{align}\label{tigu1 p7}
& \sum_{m=1}^n\frac{\delta}{2^{2m }c_2}\int_{H^1(\mathbb{R}) }\limsup_{t\rightarrow\infty}
\left (
\frac{1}{t}\int_{0}^{t}\mathrm{E}\|\mathbf{w}_m^{\varepsilon }(s)\|^2_{H^2(\mathbb{R})}ds
\right ) d\mu^\varepsilon
(\mathbf{u}_0) \nonumber\\
  &\leq  \sum_{m=1}^n \frac{\delta}{2^{2m }}<  \frac{\delta}{3}.
\end{align}
For the second  term on the right-hand
side of \eqref{tigu1 p6},
  by  \eqref{tigu1 p3}  we have
  \be\label{tigu1 p8}
\sum_{m=1}^n
2^{2m}\int_{H^1(\mathbb{R}) }\limsup_{t\rightarrow\infty}
\left (
\frac{1}{t}\int_{0}^{t}\mathrm{E}\|\mathbf{v}_m^{\varepsilon}(s)\|^2_{H^1(\mathbb{R}) }ds
\right ) d\mu^\varepsilon(\mathbf{u}_0)
   \leq
\sum_{m=1}^n\frac{\delta}{2^{2m }}
<  \frac{\delta}{3}.
\ee
Then, from
\eqref{tigu1 p6}-\eqref{tigu1 p8}, we obtain for every $n\in \mathbb{N}$,
\begin{align}\label{tigu1 p9}\mu^\varepsilon(\mathscr{W}_n^{\delta})>
1- \delta.
\end{align}
By \eqref{tigu1 p5}
and \eqref{tigu1 p9} we get
for every
$\mu^\varepsilon \in \mathcal{I}^
\varepsilon$,
 $$ \mu^\varepsilon(\mathscr{W}^{\delta}) \ge 1-\delta,
$$
which completes the proof.
 \end{proof}

We now  show the
uniform  convergence in probability of solutions
of \eqref{1.3*}.

\begin{lemma}\label{ucsp}
If \eqref{1.1}
holds, then for
every $L>0$,
$T>0$,
  $\delta>0$
  and $\varepsilon_0\in [0,1]$,
  the solution
  $\mathbf{u}^\varepsilon(t, \mathbf{u}_0)$ of \eqref{1.3*} satisfies
\be\label{ucsp 1}\lim_{\varepsilon\rightarrow \varepsilon_0}\
\sup_{\|\mathbf{u}_0
\|_{H^1(\mathbb{R})}
\le L}
\
\mathrm{P}\left(\sup_{t\in [0,T]}\|\mathbf{u}^\varepsilon(t, \mathbf{u}_0)-\mathbf{u}^{\varepsilon_0}(t, \mathbf{u}_0)\|_{H^1(\mathbb{R}) }\geq \delta\right)=0.
\ee
\end{lemma}

\begin{proof}
     Note that
     there exists $c_1=c_1(L,T)>0$
     such that
     for all
     $\varepsilon\in [0,1]$
     and
      $ \mathbf{u}_0
\in {H^1(\mathbb{R})}$
with   $\|\mathbf{u}_0
\|_{H^1(\mathbb{R})}
\le L$, the
solutions of \eqref{1.3*}
satisfy
  \be\label{ucsp p1}
    \textrm{E}
    \left (
    \sup_{0\le t\le T}
    \|\mathbf{u}^\varepsilon
     (t, \mathbf{u}_{0})
    \|^2_{H^1(\mathbb{R})}
    +\int_0^{T}
    \| \mathbf{u}^\varepsilon (s,
  \mathbf{u}_{0})\|_{H^2(\mathbb{R})}^2 ds
     \right )
      \le c_1 .
     \ee
      It follows from
     \eqref{ucsp p1}    that
     for every $\eta>0$,
     there exists
         $R
         =R(\eta, L, T)>0$ such that
      for all
     $\varepsilon\in [0,1]$
     and
      $ \mathbf{u}_0
\in {H^1(\mathbb{R})}$
with   $\|\mathbf{u}_0
\|_{H^1(\mathbb{R})}
\le L$,
      \be\label{ucsp p2}
   \mathrm{P}
     \left (
    \sup_{0\le t\le T}
    \|\mathbf{u}^\varepsilon
     (t, \mathbf{u}_{0})
    \|_{H^1(\mathbb{R})} ^2
     +\int_0^{T}
    \| \mathbf{u}^\varepsilon (s,
  \mathbf{u}_{0})\|_{H^2(\mathbb{R})}^2 ds
    >R
     \right )
       \le  {\frac 12} \eta.
       \ee
         Define the stopping
    times by
    $$
    \tau_{1,\varepsilon}
  =\inf\left\{t \ge 0;  \|\mathbf{u}
  ^\varepsilon  (t,
  \mathbf{u}_{0})
   \|^2_{H^1(\mathbb{R}) } +
   \int_{0}^{t}\| \mathbf{u}^\varepsilon
    (s,
  \mathbf{u}_{0})\|_{H^2(\mathbb{R})}^2 ds
  >R   \right\},$$
and
 $$
    \tau_{1, \varepsilon_0}
  =\inf\left\{t \ge 0;  \|\mathbf{u}^
  {\varepsilon_0}  (t,
  \mathbf{u}_{0})
   \|^2_{H^1(\mathbb{R}) } +
   \int_{0}^{t}\| \mathbf{u}^{\varepsilon_0} (s,
  \mathbf{u}_{0})\|_{H^2(\mathbb{R})}^2 ds
  >R     \right\}.$$
  As usual, $\inf \emptyset =+\infty$.
  Let
  $\tau_\varepsilon =
  \tau_{1, \varepsilon}
  \wedge \tau_{1, \varepsilon_0}$.
  Then by
  \eqref{ucsp p2}
  we find that
     \be\label{ucsp p3}
   \mathrm{P }\left (
   \tau_\varepsilon <T
   \right )
   \le \eta.
   \ee

By \eqref{1.3*} we have
for all $t\in [0, T]$, $\mathrm{P}$-almost surely,
$$
d
(\mathbf{u}^{\varepsilon}(t)
-\mathbf{u}^{\varepsilon_0}(t)
)
=\Delta (\mathbf{u}^{\varepsilon}(t)
-\mathbf{u}^{\varepsilon_0} (t)) dt
+ (\mathbf{u}^{\varepsilon}(t)
-\mathbf{u}^{\varepsilon_0}(t)
) \times \Delta  \mathbf{u}^{\varepsilon}(t)
 dt
 $$
 $$
 +
 \mathbf{u}^{\varepsilon_0}
 \times
 \Delta (\mathbf{u}^{\varepsilon}(t)
-\mathbf{u}^{\varepsilon_0}(t)
) dt
-(\mathbf{u}^{\varepsilon}(t)
-\mathbf{u}^{\varepsilon_0}(t)
) dt
-\left (
|\mathbf{u}^{\varepsilon}(t)|^2
 \mathbf{u}^{\varepsilon}(t)
 -
 |\mathbf{u}^{\varepsilon_0}(t)|^2
 \mathbf{u}^{\varepsilon_0}(t)
\right ) dt
 $$
$$
 +{\frac 12}
 (\varepsilon^2
 -\varepsilon_0^2)
 \sum_{k=1}^\infty
  (\mathbf{u}^{\varepsilon}(t)
 \times
 \mathbf{f}_k
 +
 \mathbf{f}_k
  ) dt
  +{\frac 12}
  \varepsilon^2_0
  \sum_{k=1}^\infty
   \left ( (\mathbf{u}^{\varepsilon}(t)
-\mathbf{u}^{\varepsilon_0}(t)
)\times  \mathbf{f}_k
\right )dt
$$
   \be\label{ucsp p4}
 +
 (\varepsilon
 -\varepsilon_0)
 \sum_{k=1}^\infty
 (\mathbf{u}^{\varepsilon}(t)
 \times
 \mathbf{f}_k
 +
 \mathbf{f}_k
  )dW_k
  +
  \varepsilon_0
  \sum_{k=1}^\infty \left (
  (\mathbf{u}^{\varepsilon}(t)
-\mathbf{u}^{\varepsilon_0}(t)
)\times  \mathbf{f}_k
\right )  dW_k.
\ee

By \eqref{ucsp p4} and
It\^{o}'s formula we get
for all $t\in [0, T]$, $\mathrm{P}$-almost surely,
  $$
\|\mathbf{u}^{\varepsilon}(t)
-\mathbf{u}^{\varepsilon_0} (t)
    \|^2_{H^1(\mathbb{R})}
    $$
    $$
    + 2
    \int_0^{t}
    \left(\|\mathbf{u}^{\varepsilon}(s)
-\mathbf{u}^{\varepsilon_0}(s)
    \|^2_{H^1(\mathbb{R})}
    +
    \|\nabla \mathbf{u}^{\varepsilon}(s)
-\nabla \mathbf{u}^{\varepsilon_0}(s)
    \|^2_{L^2(\mathbb{R})}
    +
    \|\Delta \mathbf{u}^{\varepsilon}(s)
-\Delta \mathbf{u}^{\varepsilon_0}(s)
    \|^2_{L^2(\mathbb{R})}
\right) ds
    $$
    $$
    =
    -2\int_0^ {t}
    \left (
    \Delta (\mathbf{u}^{\varepsilon}(s)
-\mathbf{u}^{\varepsilon_0}(s)
), \ (\mathbf{u}^{\varepsilon}(s)
-\mathbf{u}^{\varepsilon_0}(s)
)\times
\Delta \mathbf{u}^{\varepsilon}(s)
\right )_{L^2(\mathbb{R})} ds
    $$
    $$
    +2\int_0^ {t}
    \left (
       \mathbf{u}^{\varepsilon}(s)
-\mathbf{u}^{\varepsilon_0}(s)
 , \  \mathbf{u}^{\varepsilon_0}(s)
 \times  \Delta
  (\mathbf{u}^{\varepsilon}(s)
-\mathbf{u}^{\varepsilon_0}(s)
)
\right )_{L^2(\mathbb{R})} ds
    $$
    $$
    -2\int_0^ {t}
   \left (
   (I- \Delta)
  (\mathbf{u}^{\varepsilon}(s)
-\mathbf{u}^{\varepsilon_0}(s)
) , \
|\mathbf{u}^{\varepsilon}(s)|^2
 \mathbf{u}^{\varepsilon}(s)
 -
 |\mathbf{u}^{\varepsilon_0}(s)|^2
 \mathbf{u}^{\varepsilon_0}(s)
  \right )_{L^2(\mathbb{R})} ds
    $$
   $$
 +
 (\varepsilon^2
 -\varepsilon_0^2)
 \sum_{k=1}^\infty
   \int_0^ {t}
   \left (
   (I- \Delta)
  (\mathbf{u}^{\varepsilon}(s)
-\mathbf{u}^{\varepsilon_0}(s)
) , \
   \mathbf{u}^{\varepsilon}(s)
 \times
 \mathbf{f}_k
 +
 \mathbf{f}_k
  \right )_{L^2(\mathbb{R})} ds
  $$
  $$
  +
  \varepsilon^2_0
  \sum_{k=1}^\infty
   \int_0^ {t}
   \left (
   (I- \Delta)
  (\mathbf{u}^{\varepsilon}(s)
-\mathbf{u}^{\varepsilon_0}(s)
) , \
     (\mathbf{u}^{\varepsilon}(s)
-\mathbf{u}^{\varepsilon_0}(s)
)\times  \mathbf{f}_k
\right )_{L^2(\mathbb{R})} ds
$$
 $$
 +\sum_{k=1}^\infty\int_0^ {t}
 \|
 (\varepsilon
 -\varepsilon_0)
  (\mathbf{u}^{\varepsilon}(s)
 \times
 \mathbf{f}_k
 +
 \mathbf{f}_k
  )
  +
  \varepsilon_0
  (\mathbf{u}^{\varepsilon}(s)
-\mathbf{u}^{\varepsilon_0}(s)
)\times  \mathbf{f}_k
\|^2_{H^1(\mathbb{R})} ds
 $$
 $$
 +
 2(\varepsilon
 -\varepsilon_0)
 \sum_{k=1}^\infty
 \int_0^ {t}
   \left (
   (I- \Delta)
  (\mathbf{u}^{\varepsilon}(s)
-\mathbf{u}^{\varepsilon_0}(s)
) , \
   \mathbf{u}^{\varepsilon}(s)
 \times
 \mathbf{f}_k
 +
 \mathbf{f}_k
  \right )_{L^2(\mathbb{R})}dW_k
  $$
 \be\label{ucsp p5}
  +
  2 \varepsilon_0
  \sum_{k=1}^\infty
  \int_0^ {t}
   \left (
   (I- \Delta)
  (\mathbf{u}^{\varepsilon}(s)
-\mathbf{u}^{\varepsilon_0}(s)
) , \
    (\mathbf{u}^{\varepsilon}(s)
-\mathbf{u}^{\varepsilon_0}(s)
)\times  \mathbf{f}_k
\right )_{L^2(\mathbb{R})}  dW_k.
\ee
 For the first term on the right-hand side
 of  \eqref{ucsp p5},
 by Young's inequality we have
  $$
    -2\int_0^ {t}
    \left (
    \Delta (\mathbf{u}^{\varepsilon}(s)
-\mathbf{u}^{\varepsilon_0}(s)
), \ (\mathbf{u}^{\varepsilon}(s)
-\mathbf{u}^{\varepsilon_0}(s)
)\times
\Delta \mathbf{u}^{\varepsilon}(s)
\right )_{L^2(\mathbb{R})} ds
    $$
 $$
  \le 2\int_0^ {t}
    \|
    \Delta (\mathbf{u}^{\varepsilon}(s)
-\mathbf{u}^{\varepsilon_0}(s))
\|_{L^2(\mathbb{R})} \|\mathbf{u}^{\varepsilon}(s)
-\mathbf{u}^{\varepsilon_0}(s)
\|_{L^\infty(\mathbb{R})}
\|
\Delta \mathbf{u}^{\varepsilon}(s)
\|_{L^2(\mathbb{R})} ds
    $$
\be\label{ucsp p6}
   \le  \int_0^ {t}
    \|
    \Delta (\mathbf{u}^{\varepsilon}(s)
-\mathbf{u}^{\varepsilon_0}(s))
\|_{L^2(\mathbb{R})}^2 ds
+ \int_0^t
 \|\mathbf{u}^{\varepsilon}(s)
-\mathbf{u}^{\varepsilon_0}(s)
\|^2_{H^1(\mathbb{R})}
\|
\Delta \mathbf{u}^{\varepsilon}(s)
\|^2_{L^2(\mathbb{R})} ds.
 \ee
 Similar to \eqref{fp p10}-\eqref{fp p11} and \eqref{ucsp p6},
 using  Young's inequality
 to estimate all the other terms
 on the right-hand side
 of \eqref{ucsp p5}, we obtain
 for all
 $\varepsilon, \varepsilon_0 \in [0,1]$ and
 $0\le t \le t_1\le T$, $\mathrm{P}$-almost surely,
$$
\|\mathbf{u}^{\varepsilon}(t)
-\mathbf{u}^{\varepsilon_0} (t)
    \|^2_{H^1(\mathbb{R})}
    \le
  c_2 (\varepsilon -\varepsilon_0)^2
 \int_0^t
 \left(1 + \|\mathbf{u}^{\varepsilon}(s)\|^2_{H^1(\mathbb{R})}
 \right)ds
 $$
     $$
  +
     c_2\int_0^t
     \|\mathbf{u}^{\varepsilon}(s)
-\mathbf{u}^{\varepsilon_0}(s)
\|^2_{H^1(\mathbb{R})}
\left (1+
\|\mathbf{u}^{\varepsilon}(s)\|^2
_{H^2(\mathbb{R})}
+\|\mathbf{u}^{\varepsilon}(s)\|^4
_{H^1(\mathbb{R})}
+ \|\mathbf{u}^{\varepsilon_0}(s)\|^4
_{H^1(\mathbb{R})}
\right )ds
     $$
  $$
 +
 2(\varepsilon
 -\varepsilon_0)
 \sum_{k=1}^\infty
 \int_0^ {t}
   \left (
   (I- \Delta)
  (\mathbf{u}^{\varepsilon}(s)
-\mathbf{u}^{\varepsilon_0}(s)
) , \
   \mathbf{u}^{\varepsilon}(s)
 \times
 \mathbf{f}_k
 +
 \mathbf{f}_k
  \right )_{L^2(\mathbb{R})}dW_k
  $$
 \be\label{ucsp p7}
  +
  2 \varepsilon_0
  \sum_{k=1}^\infty
  \int_0^ {t}
   \left ( \nabla
  (\mathbf{u}^{\varepsilon}(s)
-\mathbf{u}^{\varepsilon_0}(s)
) , \
    (\mathbf{u}^{\varepsilon}(s)
-\mathbf{u}^{\varepsilon_0}(s)
)\times  \nabla \mathbf{f}_k
\right )_{L^2(\mathbb{R})}  dW_k,
\ee
where $c_2>0$ depends
only on $\{\mathbf{f}_k \}_{k=1}^\infty
$.
By \eqref{ucsp p7}, we obtain
 for all
 $\varepsilon, \varepsilon_0 \in [0,1]$ and
 $0\le t \le t_1\le T$,
  $\mathrm{P}$-almost surely,
$$
\sup_{r\in [0,t]}
\|\mathbf{u}^{\varepsilon}(r\wedge \tau_\varepsilon)
-\mathbf{u}^{\varepsilon_0}
( r\wedge \tau_\varepsilon)
    \|^2_{H^1(\mathbb{R})}
    \le
  c_2 (\varepsilon -\varepsilon_0)^2
 \int_0^{t\wedge \tau_\varepsilon}
 \left(1 + \|\mathbf{u}^{\varepsilon}(s)\|^2_{H^1(\mathbb{R})}
 \right)ds
 $$
     $$
    +
     c_2\int_0^ {t\wedge \tau_\varepsilon}
     \|\mathbf{u}^{\varepsilon}(s)
-\mathbf{u}^{\varepsilon_0}(s)
\|^2_{H^1(\mathbb{R})}
\left (1+
\|\mathbf{u}^{\varepsilon}(s)\|^2
_{H^2(\mathbb{R})}
+\|\mathbf{u}^{\varepsilon}(s)\|^4
_{H^1(\mathbb{R})}
+ \|\mathbf{u}^{\varepsilon_0}(s)\|^4
_{H^1(\mathbb{R})}
\right )ds
     $$
  $$
 +
 2|\varepsilon
 -\varepsilon_0|
 \left(\sup_{r\in [0,t]}
 \left |
 \sum_{k=1}^\infty
 \int_0^ {r\wedge \tau_\varepsilon}
   \left (
   (I- \Delta)
  (\mathbf{u}^{\varepsilon}(s)
-\mathbf{u}^{\varepsilon_0}(s)
) , \
   \mathbf{u}^{\varepsilon}(s)
 \times
 \mathbf{f}_k
 +
 \mathbf{f}_k
  \right )_{L^2(\mathbb{R})}dW_k \right |\right)
  $$
$$
  +
  2 \varepsilon_0
  \left(\sup_{r\in [0,t]}
  \left |
  \sum_{k=1}^\infty
  \int_0^  {r\wedge \tau_\varepsilon}
  \left ( \nabla
  (\mathbf{u}^{\varepsilon}(s)
-\mathbf{u}^{\varepsilon_0}(s)
) , \
    (\mathbf{u}^{\varepsilon}(s)
-\mathbf{u}^{\varepsilon_0}(s)
)\times  \nabla \mathbf{f}_k
\right )_{L^2(\mathbb{R})}   dW_k \right |\right)
$$
$$
 \le
  c_2 (\varepsilon -\varepsilon_0)^2
 \int_0^{ t_1\wedge \tau_\varepsilon}
 \left(1 + \|\mathbf{u}^{\varepsilon}(s)\|^2_{H^1(\mathbb{R})}
 \right)ds
 $$
     $$
    +
     c_2\int_0^ {t }
     \sup_{r\in [0,s]}
     \|\mathbf{u}^{\varepsilon}( r\wedge \tau_\varepsilon)
-\mathbf{u}^{\varepsilon_0}(r\wedge \tau_\varepsilon)
\|^2_{H^1(\mathbb{R})}
1_{(0,\tau_\varepsilon)} (s)
\left (1+
\|\mathbf{u}^{\varepsilon}(s)\|^2
_{H^2(\mathbb{R})}
+\|\mathbf{u}^{\varepsilon}(s)\|^4
_{H^1(\mathbb{R})}
+ \|\mathbf{u}^{\varepsilon_0}(s)\|^4
_{H^1(\mathbb{R})}
\right )ds
     $$
  $$
 +
 2|\varepsilon
 -\varepsilon_0|
 \left(\sup_{r\in [0, t_1]}
 \left |
 \sum_{k=1}^\infty
 \int_0^ {r\wedge \tau_\varepsilon}
   \left (
   (I- \Delta)
  (\mathbf{u}^{\varepsilon}(s)
-\mathbf{u}^{\varepsilon_0}(s)
) , \
   \mathbf{u}^{\varepsilon}(s)
 \times
 \mathbf{f}_k
 +
 \mathbf{f}_k
  \right )_{L^2(\mathbb{R})}dW_k \right |\right)
  $$
 \be\label{ucsp p8}
  +
  2 \varepsilon_0
  \left(\sup_{r\in [0,t_1]}
  \left |
  \sum_{k=1}^\infty
  \int_0^  {r\wedge \tau_\varepsilon}
  \left ( \nabla
  (\mathbf{u}^{\varepsilon}(s)
-\mathbf{u}^{\varepsilon_0}(s)
) , \
    (\mathbf{u}^{\varepsilon}(s)
-\mathbf{u}^{\varepsilon_0}(s)
)\times  \nabla \mathbf{f}_k
\right )_{L^2(\mathbb{R})}
    dW_k \right |\right).
\ee

 Now for a fixed
 $t_1\in [0,T]$,
applying
 Gronwall's lemma
 to  \eqref{ucsp p8}
 for $t\in [0,t_1]$, we obtain
 for all
 $\varepsilon, \varepsilon_0 \in [0,1]$ and
 $t\in [0, t_1]$, $\mathrm{P}$-almost surely,
$$
\sup_{r\in [0,t]}
\|\mathbf{u}^{\varepsilon}(r\wedge \tau_\varepsilon)
-\mathbf{u}^{\varepsilon_0}
( r\wedge \tau_\varepsilon)
    \|^2_{H^1(\mathbb{R})}
    $$
    $$
 \le
  c_2 (\varepsilon -\varepsilon_0)^2
  \varphi_\varepsilon (t)
 \int_0^{t_1 \wedge \tau_\varepsilon}
 \left(1 + \|\mathbf{u}^{\varepsilon}(s)\|_{H^1(\mathbb{R})}^2
 \right)ds
 $$
   $$
 +
 2|\varepsilon
 -\varepsilon_0| \varphi_\varepsilon (t)
 \left(\sup_{r\in [0, t_1]}
 \left |
 \sum_{k=1}^\infty
 \int_0^ {r\wedge \tau_\varepsilon}
   \left (
   (I- \Delta)
  (\mathbf{u}^{\varepsilon}(s)
-\mathbf{u}^{\varepsilon_0}(s)
) , \
   \mathbf{u}^{\varepsilon}(s)
 \times
 \mathbf{f}_k
 +
 \mathbf{f}_k
  \right )_{L^2(\mathbb{R})}dW_k \right |\right)
  $$
 \be\label{ucsp p9}
  +
  2 \varepsilon_0
  \varphi _\varepsilon (t)
  \left(\sup_{r\in [0, t_1]}
  \left |
  \sum_{k=1}^\infty
  \int_0^  {r\wedge \tau_\varepsilon}
  \left ( \nabla
  (\mathbf{u}^{\varepsilon}(s)
-\mathbf{u}^{\varepsilon_0}(s)
) , \
    (\mathbf{u}^{\varepsilon}(s)
-\mathbf{u}^{\varepsilon_0}(s)
)\times  \nabla \mathbf{f}_k
\right )_{L^2(\mathbb{R})}
    dW_k \right |\right),
\ee
 where for all $t\in [0, t_1]$,
  $$
 \varphi_\varepsilon (t)
 =
 e^{\int_0^{t\wedge \tau_\varepsilon}
    c_2
    \left(1+
\|\mathbf{u}^{\varepsilon}(s)\|^2
_{H^2(\mathbb{R})}
+\|\mathbf{u}^{\varepsilon}(s)\|^4
_{H^1(\mathbb{R})}
+ \|\mathbf{u}^{\varepsilon_0}(s)\|^4
_{H^1(\mathbb{R})}
  \right) ds
  }
  $$
 \be \label{ucsp p10}
  \le
  e^{\int_0^{t\wedge \tau_\varepsilon}
    c_2
    \left(1+ 2R^2+
\|\mathbf{u}^{\varepsilon}(s)\|^2
_{H^2(\mathbb{R})}
 \right ) ds
  }
  \le
  e^{
    c_2
    (1+ 2R^2)T + c_2R
  }.
  \ee
 Let $c_3 =  e^{
    c_2
    (1+ 2R^2)T + c_2R
  }.$ Then by \eqref{ucsp p9}-\eqref{ucsp p10}
  we get
  for all
 $\varepsilon, \varepsilon_0 \in [0,1]$
 and $t_1\in [0,T]$,
  $\mathrm{P}$-almost surely,
$$
\sup_{r\in [0,t_1]}
\|\mathbf{u}^{\varepsilon}(r\wedge \tau_\varepsilon)
-\mathbf{u}^{\varepsilon_0}
( r\wedge \tau_\varepsilon)
    \|^2_{H^1(\mathbb{R})}
  \le
  c_2c_3 (\varepsilon -\varepsilon_0)^2
 (1 + R)t_1
  $$
   $$
 +
 2|\varepsilon
 -\varepsilon_0|c_3
 \left(\sup_{r\in [0, t_1]}
 \left |
 \sum_{k=1}^\infty
 \int_0^ {r\wedge \tau_\varepsilon}
   \left (
   (I- \Delta)
  (\mathbf{u}^{\varepsilon}(s)
-\mathbf{u}^{\varepsilon_0}(s)
) , \
   \mathbf{u}^{\varepsilon}(s)
 \times
 \mathbf{f}_k
 +
 \mathbf{f}_k
  \right )_{L^2(\mathbb{R})}dW_k \right |\right)
  $$
 \be\label{ucsp p11}
  +
  2 \varepsilon_0 c_3
  \left(\sup_{r\in [0, t_1]}
  \left |
  \sum_{k=1}^\infty
  \int_0^  {r\wedge \tau_\varepsilon}
  \left ( \nabla
  (\mathbf{u}^{\varepsilon}(s)
-\mathbf{u}^{\varepsilon_0}(s)
) , \
    (\mathbf{u}^{\varepsilon}(s)
-\mathbf{u}^{\varepsilon_0}(s)
)\times  \nabla \mathbf{f}_k
\right )_{L^2(\mathbb{R})}
     dW_k \right |\right).
\ee

 Taking the expectation
 of
 \eqref{ucsp p11}
 to obtain
 for all
 $\varepsilon, \varepsilon_0 \in [0,1]$
 and $t_1\in [0,T]$,
$$
\textrm{E}
\left (
 \sup_{r\in [0, t_1]}
\|\mathbf{u}^{\varepsilon}(r\wedge \tau_\varepsilon)
-\mathbf{u}^{\varepsilon_0}
( r\wedge \tau_\varepsilon)
    \|^2_{H^1(\mathbb{R})}
    \right )
  \le
  c_2c_3 (\varepsilon -\varepsilon_0)^2
 (1 + R)T
  $$
   $$
 +
 2|\varepsilon
 -\varepsilon_0|c_3
 \textrm{E}
\left (
 \sup_{r\in [0, t_1]}
 \left |
 \sum_{k=1}^\infty
 \int_0^ {r\wedge \tau_\varepsilon}
   \left (
   (I- \Delta)
  (\mathbf{u}^{\varepsilon}(s)
-\mathbf{u}^{\varepsilon_0}(s)
) , \
   \mathbf{u}^{\varepsilon}(s)
 \times
 \mathbf{f}_k
 +
 \mathbf{f}_k
  \right )_{L^2(\mathbb{R})}dW_k \right |
  \right )
  $$
 \be\label{ucsp p12}
  +
  2 \varepsilon_0 c_3
  \textrm{E}
\left (
  \sup_{r\in [0, t_1]}
  \left |
  \sum_{k=1}^\infty
  \int_0^  {r\wedge \tau_\varepsilon}
  \left ( \nabla
  (\mathbf{u}^{\varepsilon}(s)
-\mathbf{u}^{\varepsilon_0}(s)
) , \
    (\mathbf{u}^{\varepsilon}(s)
-\mathbf{u}^{\varepsilon_0}(s)
)\times  \nabla \mathbf{f}_k
\right )_{L^2(\mathbb{R})}
     dW_k \right |
\right ).
\ee

Using the BDG inequality,   Young's inequality and \eqref{1.1}, we  get
 $$
 2|\varepsilon
 -\varepsilon_0|c_3
 \textrm{E}
\left (
 \sup_{r\in [0, t_1]}
 \left |
 \sum_{k=1}^\infty
 \int_0^ {r\wedge \tau_\varepsilon}
   \left (
   (I- \Delta)
  (\mathbf{u}^{\varepsilon}(s)
-\mathbf{u}^{\varepsilon_0}(s)
) , \
   \mathbf{u}^{\varepsilon}(s)
 \times
 \mathbf{f}_k
 +
 \mathbf{f}_k
  \right )_{L^2(\mathbb{R})}dW_k \right |
  \right )
  $$
 $$
  +
  2 \varepsilon_0 c_3
  \textrm{E}
\left (
  \sup_{r\in [0, t_1]}
  \left |
  \sum_{k=1}^\infty
  \int_0^  {r\wedge \tau_\varepsilon}
  \left ( \nabla
  (\mathbf{u}^{\varepsilon}(s)
-\mathbf{u}^{\varepsilon_0}(s)
) , \
    (\mathbf{u}^{\varepsilon}(s)
-\mathbf{u}^{\varepsilon_0}(s)
)\times  \nabla \mathbf{f}_k
\right )_{L^2(\mathbb{R})}
   dW_k \right |
\right )
$$
$$
\le
 2|\varepsilon
 -\varepsilon_0|c_3c_4
 \textrm{E}
\left (
 \sum_{k=1}^\infty
 \int_0^ {t_1 \wedge \tau_\varepsilon}
   \left |\left (
   (I- \Delta)
  (\mathbf{u}^{\varepsilon}(s)
-\mathbf{u}^{\varepsilon_0}(s)
) , \
   \mathbf{u}^{\varepsilon}(s)
 \times
 \mathbf{f}_k
 +
 \mathbf{f}_k
  \right )  \right |^2ds
  \right )^{\frac 12}
  $$
 $$
  +
  2 \varepsilon_0 c_3c_4
  \textrm{E}
\left (
  \sum_{k=1}^\infty
  \int_0^  {t_1\wedge \tau_\varepsilon}
   \left |
   \left ( \nabla
  (\mathbf{u}^{\varepsilon}(s)
-\mathbf{u}^{\varepsilon_0}(s)
) , \
    (\mathbf{u}^{\varepsilon}(s)
-\mathbf{u}^{\varepsilon_0}(s)
)\times  \nabla \mathbf{f}_k
\right )    \right |^2ds
\right )^{\frac 12}
$$
 $$
\le
 2|\varepsilon
 -\varepsilon_0|c_3c_4
 \textrm{E}
\left (
 \sum_{k=1}^\infty
 \int_0^ {t_1 \wedge \tau_\varepsilon}
    \|
   \mathbf{u}^{\varepsilon}(s)
-\mathbf{u}^{\varepsilon_0}(s)
\|^2_{H^1(\mathbb{R})}
  \| \mathbf{u}^{\varepsilon}(s)
 \times
 \mathbf{f}_k
 +
 \mathbf{f}_k\|^2_{H^1(\mathbb{R})}
  ds
  \right )^{\frac 12}
  $$
 $$
  +
  2 \varepsilon_0 c_3c_4
  \textrm{E}
\left (
  \int_0^  {t_1\wedge \tau_\varepsilon}
   \|\nabla (
   \mathbf{u}^{\varepsilon}(s)
-\mathbf{u}^{\varepsilon_0}(s))
\|_{L^2(\mathbb{R})}^2
 \|  \mathbf{u}^{\varepsilon}(s)
-\mathbf{u}^{\varepsilon_0}(s)
 \|_{L^2(\mathbb{R})}^2
  \sum_{k=1}^\infty
   \|\nabla \mathbf{f}_k\|^2_{L^\infty
   (\mathbb{R})}
ds
\right )^{\frac 12}
$$
$$
\le
 2|\varepsilon
 -\varepsilon_0|c_5
 \textrm{E}
\left (
 \int_0^ {t_1 \wedge \tau_\varepsilon}
    \|
   \mathbf{u}^{\varepsilon}(s)
-\mathbf{u}^{\varepsilon_0}(s)
\|^2_{H^1(\mathbb{R})}
  \left (1+ \| \mathbf{u}^{\varepsilon}(s)
  \|^2_{H^1(\mathbb{R})} \right )
  ds
  \right )^{\frac 12}
  $$
 $$
  +
  2 \varepsilon_0 c_6
  \textrm{E}
\left (
  \int_0^  {t_1\wedge \tau_\varepsilon}
   \|
  \mathbf{u}^{\varepsilon}(s)
-\mathbf{u}^{\varepsilon_0}(s)
\|^2_{H^1(\mathbb{R})}
\|
   \mathbf{u}^{\varepsilon}(s)
-\mathbf{u}^{\varepsilon_0}(s)
\|_{L^2(\mathbb{R})}^2
 ds
\right )^{\frac 12}
$$
 $$
\le
 2|\varepsilon
 -\varepsilon_0|c_5
 (1+R)^{\frac 12}
 \textrm{E}
\left (
 \int_0^ {t_1 \wedge \tau_\varepsilon}
    \|
   \mathbf{u}^{\varepsilon}(s)
-\mathbf{u}^{\varepsilon_0}(s)
\|^2_{H^1(\mathbb{R})}
  ds
  \right )^{\frac 12}
  $$
  $$
  +
  2 \varepsilon_0 c_6
  \textrm{E} \left (
  \sup_{r\in [0, t_1]}
   \|
  \mathbf{u}^{\varepsilon}(r\wedge \tau_\varepsilon)
-\mathbf{u}^{\varepsilon_0}( r\wedge \tau_\varepsilon)
\| _{H^1(\mathbb{R})}
\left (
  \int_0^  {t_1}
 \|
   \mathbf{u}^{\varepsilon}(s\wedge \tau_\varepsilon)
-\mathbf{u}^{\varepsilon_0}(s\wedge \tau_\varepsilon)
\|_{L^2(\mathbb{R})}^2
 ds
\right )^{\frac 12}
\right )
$$
 $$
\le
  |\varepsilon
 -\varepsilon_0|^2c_5^2
 (1+R)
 +
 \textrm{E}
\left (
 \int_0^ {t_1  }
    \|
   \mathbf{u}^{\varepsilon}(s\wedge \tau_\varepsilon)
-\mathbf{u}^{\varepsilon_0}(s\wedge \tau_\varepsilon)
\|^2_{H^1(\mathbb{R})}
  ds
  \right )
  $$
 $$
  +
  {\frac 12}
   \textrm{E} \left (
  \sup_{r\in [0, t_1]}
   \|
  \mathbf{u}^{\varepsilon}(r\wedge \tau_\varepsilon)
-\mathbf{u}^{\varepsilon_0}( r\wedge \tau_\varepsilon)
\| ^2_{H^1(\mathbb{R})}
\right )
+
 2 \varepsilon_0^2 c_6^2
  \textrm{E}
\left (
  \int_0^  {t_1 }
   \|
  \mathbf{u}^{\varepsilon}(s\wedge \tau_\varepsilon)
-\mathbf{u}^{\varepsilon_0}(s\wedge \tau_\varepsilon)
\|^2_{L^2(\mathbb{R})}
 ds
\right )
$$
$$
\le
  |\varepsilon
 -\varepsilon_0|^2c_5^2
 (1+R)
 + (1+ 2\varepsilon_0^2 c_6^2)
 \int_0^ {t_1  }
 \textrm{E}
\left (
    \sup_{r\in [0,s]}
    \|
   \mathbf{u}^{\varepsilon}(r\wedge \tau_\varepsilon)
-\mathbf{u}^{\varepsilon_0}(r\wedge \tau_\varepsilon)
\|^2_{H^1(\mathbb{R})}
  \right )  ds
  $$
\be\label{ucsp p13}
  +
  {\frac 12}
   \textrm{E} \left (
  \sup_{r\in [0, t_1]}
   \|
  \mathbf{u}^{\varepsilon}(r\wedge \tau_\varepsilon)
-\mathbf{u}^{\varepsilon_0}( r\wedge \tau_\varepsilon)
\| ^2_{H^1(\mathbb{R})}
\right ) ,
\ee
where all the positive  constants
$c_4 $, $c_5 $ and $c_6 $
depend only on
$\{\mathbf{f}_k\}_{k=1}^\infty$.

It follows from \eqref{ucsp p12}-\eqref{ucsp p13} that
for all
 $\varepsilon, \varepsilon_0 \in [0,1]$
 and $t_1\in [0,T]$,
$$
\textrm{E}
\left (
 \sup_{r\in [0, t_1]}
\|\mathbf{u}^{\varepsilon}(r\wedge \tau_\varepsilon)
-\mathbf{u}^{\varepsilon_0}
( r\wedge \tau_\varepsilon)
    \|^2_{H^1(\mathbb{R})}
    \right )
  \le
  2c_2c_3 (\varepsilon -\varepsilon_0)^2
 (1 + R)T
  $$
   \be\label{ucsp p14}
   +
   2|\varepsilon
 -\varepsilon_0|^2c_5^2
 (1+R)
 + 2(1+ 2\varepsilon_0^2 c_6^2)
 \int_0^ {t_1  }
 \textrm{E}
\left (
    \sup_{r\in [0,s]}
    \|
   \mathbf{u}^{\varepsilon}(r\wedge \tau_\varepsilon)
-\mathbf{u}^{\varepsilon_0}(r\wedge \tau_\varepsilon)
\|^2_{H^1(\mathbb{R})}
  \right )  ds .
  \ee
 Applying
 Gronwall's inequality
 to \eqref{ucsp p14}
 for $t_1 \in [0,T]$, we get
 for all
 $\varepsilon, \varepsilon_0 \in [0,1]$
 and $t_1\in [0,T]$,
$$
\textrm{E}
\left (
 \sup_{r\in [0, t_1]}
\|\mathbf{u}^{\varepsilon}(r\wedge \tau_\varepsilon)
-\mathbf{u}^{\varepsilon_0}
( r\wedge \tau_\varepsilon)
    \|^2_{H^1(\mathbb{R})}
    \right )
    $$
    $$
  \le
  \left (
  2c_2c_3 (\varepsilon -\varepsilon_0)^2
 (1 + R)T
   +
   2|\varepsilon
 -\varepsilon_0|^2c_5^2
 (1+R)
 \right ) e^{
 2(1+ 2\varepsilon_0^2 c_6^2)t_1
 },
   $$
and hence
 for all
 $\varepsilon, \varepsilon_0 \in [0,1]$
 and $t_1=T$, we obtain
 $$
\textrm{E}
\left (
 \sup_{r\in [0, T]}
\|\mathbf{u}^{\varepsilon}(r\wedge \tau_\varepsilon)
-\mathbf{u}^{\varepsilon_0}
( r\wedge \tau_\varepsilon)
    \|^2_{H^1(\mathbb{R})}
    \right )
    $$
   \be\label{ucsp p15}
  \le   2(\varepsilon -\varepsilon_0)^2
  (1 + R)
  \left (
   c_2c_3
 T
   +
    c_5^2
 \right ) e^{
 2(1+ 2\varepsilon_0^2 c_6^2)T
 }.
  \ee

By \eqref{ucsp p3} and
  Chebyshev's inequality, we
  get
  for all
 $\varepsilon, \varepsilon_0 \in [0,1]$
 and
 $\mathbf{u}_0
 \in H^1(\mathbb{R})$
 with
 $\|\mathbf{u}_0\|_{
   H^1(\mathbb{R})} \le L$,
 \begin{align}\label{ucsp p16}
&\mathrm{P}\left(\sup_{t\in [0,T]}\|\mathbf{u}^\varepsilon(t, \mathbf{u}_0)-\mathbf{u}^{\varepsilon_0}(t, \mathbf{u}_0)\|_{ H^1(\mathbb{R})}\geq \delta\right)\nonumber\\
&\leq\mathrm{P}\left(\left\{\sup_{t\in [0,T]}\|\mathbf{u}^\varepsilon(t, \mathbf{u}_0)-\mathbf{u}^{\varepsilon_0}(t, \mathbf{u}_0)\|_{ H^1(\mathbb{R})}\geq \delta\right\}\bigcap \left\{ \tau_\varepsilon \ge T\right\}\right)\nonumber\\
&\quad+\mathrm{P}\left(\left\{\sup_{t\in [0,T]}\|\mathbf{u}^\varepsilon(t, \mathbf{u}_0)-\mathbf{u}^{\varepsilon_0}(t, \mathbf{u}_0)\|_{ H^1(\mathbb{R})}\geq \delta\right\}\bigcap \left\{\tau_\varepsilon
< T\right\}\right)\nonumber\\
&\leq \mathrm{P}\left(\sup_{t\in [0,  T]}\|\mathbf{u}^\varepsilon(t
\wedge \tau_\varepsilon , \mathbf{u}_0)-\mathbf{u}^{\varepsilon_0}(t
\wedge
\tau_\varepsilon
, \mathbf{u}_0)\|_{
H^1(\mathbb{R})
}\geq \delta\right)+\mathrm{P}\left( \tau_\varepsilon
<  T\right)\nonumber\\
&\leq \frac{1}{\delta^2}\mathrm{E}\left(\sup_{t\in [0, T]}\|\mathbf{u}^\varepsilon(t
\wedge \tau_\varepsilon
, \mathbf{u}_0)-\mathbf{u}^{\varepsilon_0}(t
\wedge \tau_\varepsilon , \mathbf{u}_0)\|^2_{ H^1(\mathbb{R})}\right)
+\eta.
\end{align}

By \eqref{ucsp p15}-\eqref{ucsp p16}
we find that
 for all
 $\varepsilon, \varepsilon_0 \in [0,1]$,
$$
\sup_{
   \|\mathbf{u}_0\|_{
   H^1(\mathbb{R})} \le L}
 \mathrm{P}\left(\sup_{t\in [0,T]}\|\mathbf{u}^\varepsilon(t, \mathbf{u}_0)-\mathbf{u}^{\varepsilon_0}(t, \mathbf{u}_0)\|_{ H^1(\mathbb{R})}\geq \delta\right)
 $$
$$
  \le   2(\varepsilon -\varepsilon_0)^2
  (1 + R)\delta^{-2}
  \left (
   c_2c_3
 T
   +
    c_5^2
 \right ) e^{
 2(1+ 2\varepsilon_0^2 c_6^2)T
 } +\eta .
 $$
 First letting $\varepsilon\to
 \varepsilon_0$ and then
 $\eta \to 0$, we obtain
 \eqref{ucsp 1},
 which completes the proof.
   \end{proof}

We are now in a position
to prove  Theorem  \ref{th1}.

$\mathbf{Proof~ of ~Theorem~ \ref{th1}}$.
Note that Theorem \ref{th1} (i)
follows from
 Theorem \ref{eim_wk}.
It remains to show Theorem \ref{th1} (ii).

By Proposition \ref{tigu1} we
know  the set $\bigcup_{\varepsilon\in (0,1]}\mathcal{I}^\varepsilon$ is tight
in $H^1(\mathbb{R})$.
Then   for any sequence $\{\mu^{\varepsilon_n}\}_{n=1}^\infty \subseteq \bigcup_{\varepsilon\in (0,1]}\mathcal{I}^\varepsilon$
with $\varepsilon_n
\to \varepsilon_0
\in [0,1]$, there exists a subsequence (still denoted by $\{\mu^{\varepsilon_n}\}_{n=1}^\infty$) and
a probability measure
$\mu^{\varepsilon_0}$
in $H^1(\mathbb{R})$
 such that $\mu^{\varepsilon_n}\to  \mu^{\varepsilon_0}$
 weakly.
 By  Lemma \ref{ucsp}, we
 find from
 \cite{li2022}
 (or \cite{D1} for $\varepsilon_0 =0$)  that   $\mu^{\varepsilon_0}$ is an invariant measure of the limiting equation \eqref{1.3*} with $\varepsilon=\varepsilon_0$,
 which concludes the proof.

\section*{Acknowledgments}
D. Huang is supported by the National Natural Science Foundation of China (Grant No. 12471228). Z. Qiu is supported by the National Natural Science Foundation of China
(Grant No. 12401305), the National Science Foundation for Colleges and Universities in
Jiangsu Province (Grant No. 24KJB110011) and the National Science Foundation of Jiangsu
Province (Grant No. BK20240721).

\end{document}